\documentclass[reqno]{amsart}
\usepackage{amssymb}
\usepackage{pdfsync}
\usepackage[usenames]{color}
\numberwithin{equation}{section}

\newtheorem{theorem}{Theorem}[section]
\newtheorem{proposition}[theorem]{Proposition}
\newtheorem{corollary}[theorem]{Corollary}

\theoremstyle{definition}

\theoremstyle{definition} 
\newtheorem{remark}[theorem]{Remark}
\newtheorem{remarks}[theorem]{Remarks}
\newtheorem{example}[theorem]{Example}

\newcommand{\bea}{\begin{eqnarray}}
\newcommand{\eea}{\end{eqnarray}}
\newcommand{\beas}{\begin{eqnarray*}}
\newcommand{\eeas}{\end{eqnarray*}}
\newcommand{\beq}{\begin{equation}}
\newcommand{\eeq}{\end{equation}}

\def\ein{\hookrightarrow}

\def\eps{\varepsilon}

\newcommand{\cA}{\mathcal A}
\newcommand{\cB}{\mathcal B}

\newcommand{\cD}{\mathcal D}
\newcommand{\cE}{\mathcal E}

\newcommand{\cH}{\mathcal H}
\newcommand{\cR}{\mathcal R}

\newcommand{\cN}{\mathcal N}
\newcommand{\cMH}{\mathcal{MH}}
\newcommand{\cSM}{\mathcal{SM}}
\newcommand{\cX}{\tilde{Z}}
\newcommand{\BB}{\mathbb B}
\newcommand{\C}{\mathbb C}
\newcommand{\R}{\mathbb R}

\newcommand{\NN}{\mathbb N}
\newcommand{\LL}{\mathbb L}
\newcommand{\EE}{\mathbb E}
\newcommand{\FF}{\mathbb F}
\newcommand{\TT}{\mathbb T}

\newcommand{\XX}{\mathbb X}

\newcommand{\Ver}{|\!|}

\newcommand{\ve}{\varepsilon}

\newcommand{\cxx}{\textsf{x}}
\newcommand{\cy}{\sf{y}}

\DeclareMathSymbol{\complement}{\mathord}{AMSa}{"7B}

\def\vv<#1>{\langle #1\rangle}
\def\Vv<#1>{\bigl\langle #1\bigr\rangle}

\begin{document}


\title[Qualitative behavior of Stefan problems with surface tension]
{Qualitative behavior of solutions\\ for thermodynamically consistent
Stefan problems with surface tension}

\author[J.~Pr\"uss]{Jan Pr\"uss}
\address{Institut f\"ur Mathematik \\
         Martin-Luther-Universit\"at Halle-Witten\-berg\\
         D-60120 Halle, Germany}
\email{jan.pruess@mathematik.uni-halle.de}

\author[G.~Simonett]{Gieri Simonett}
\address{Department of Mathematics\\
         Vanderbilt University \\
         Nashville, TN~37240, USA}
\email{gieri.simonett@vanderbilt.edu}

\author[R.~Zacher]{Rico Zacher}
\address{Institut f\"ur Mathematik \\
         Martin-Luther-Universit\"at Halle-Witten\-berg\\
         D-60120 Halle, Germany}
\email{rico.zacher@mathematik.uni-halle.de}

\thanks{The research of G.S.\ was partially
supported by the NSF Grant DMS-0600870. The research of R.Z.\ was
partially supported by the Deutsche Forschungsgemeinschaft (DFG)}


\begin{abstract}
The qualitative behavior of a thermodynamically consistent two-phase
Stefan problem with surface tension and with or without kinetic
undercooling is studied. It is shown that these problems generate
local semiflows in well-defined state manifolds. If a solution does
not exhibit singularities in a sense made precise below, it is
proved that it exists globally in time and its orbit is relatively compact. 
In addition, stability and
instability of equilibria is studied. In particular, it is shown
that multiple spheres of the same radius are unstable, reminiscent
of the onset of Ostwald ripening.
\end{abstract}
\maketitle
\vspace{-.5cm}
\section{Introduction}
The Stefan problem is a model for phase transitions
in liquid-solid systems and accounts for heat diffusion and
exchange of latent heat in a homogeneous medium.
The strong formulation of this model corresponds to a
free boundary problem involving a parabolic diffusion
equation for each phase and transmission conditions prescribed at
the interface separating the phases.
\smallskip

a) In order to describe the physical situation,
let us consider a domain $\Omega$ that is occupied by a liquid and a solid
phase, say water and ice, that are separated by an interface
$\Gamma$.  Due to melting or freezing, the corresponding regions
occupied by water and ice
will change and, consequently, the interface $\Gamma$ 
will also change its position and shape.
This leads to a free boundary problem. 

The basic physical law governing this process is conservation of energy. 
The unknowns are the temperatures $v_i$, $i=1,2,$ of the two phases, 
and the position of the interface $\Gamma$
separating the phases.
The conservation laws can then be expressed by
a diffusion equation for $v_i$
in the respective regions $\Omega_i$ occupied by the different phases,
and by the so-called Stefan condition which accounts for
the exchange of latent heat due to melting or solidifying.
In the {\sl classical} Stefan problem one assumes in addition that the temperatures
$v_i$ coincide with the melting temperature at the interface 
$\Gamma$ (where the two phases are in contact), that is, one requires 
\begin{equation}
\label{Gamma}
v_1=v_2=0\quad\text{on}\quad \Gamma,
\end{equation}
where $0$ is the (properly scaled) melting temperature.
Molecular considerations 
suggest that the condition \eqref{Gamma}
on the free boundary $\Gamma$ be replaced by the
{\sl Gibbs-Thomson correction}
\begin{equation}
\label{GT}
v_1=v_2=-\sigma \cH\quad\text{on}\quad \Gamma,
\end{equation}
where $\sigma$ is a positive constant, called surface tension,
and where $\cH$ denotes the mean curvature of
$\Gamma$. 
We will occasionally refer to the Stefan problem with
condition \eqref{GT} as the {\sl classical Stefan problem with
surface tension}.

It should be emphasized that the Stefan problem 
with Gibbs-Thomson correction \eqref{GT} differs from the
classical Stefan problem in a  much more fundamental way
than just in the modification of an interface condition.
This becomes evident, for instance, by the fact that
the classical Stefan problem allows for a comparison principle,
a property that is no longer shared by the 
Stefan problem with surface tension.
A striking difference is also provided
by the fact that
in the classical Stefan model, the temperature 
completely determines the phases, that is, the liquid
region can be characterized by the condition $v>0$,
whereas $v<0$ characterizes the solid region.
The inclusion of surface tension will no longer allow
to determine the phases merely by the sign of $v$.

The main reason for introducing the Gibbs-Thomson correction 
\eqref{GT} stems from the need to account for so-called \emph{supercooling},
in which a fluid supports temperatures below its freezing point,
or \emph{superheating}, the analogous phenomena for solids;
or dendrite formation, in which simple shapes evolve into 
complicated tree-like structures. 
The effect of supercooling can be in the order of hundreds of 
degrees for certain materials, see \cite[Chapter 1]{Ch77} and \cite{Vi96}.
We also refer to \cite{Ca86, Ch77, FR91, Gu86, Gu88, Gu88b, Har73, La80, MS63, MS64,Vi93}
for additional information.

The Stefan problem has been studied in the mathematical
literature for over a century, see \cite{Ru71, Me92} and
\cite[pp. 117--120]{Vi96} for a historic account.
The classical Stefan problem 
is known to admit  unique long time weak solutions, see for instance 
\cite{Fr68,Fr82,Ka65} and \cite[pp. 496--503]{LSU68}.  
It is important to point out that
the existence of weak solutions is closely tied 
to the maximum principle.
\\
Important results concerning the regularity of weak solutions for  
the multidimensional classical
{one-phase} Stefan problem  
were established in 
\cite{Ca77,Ca78,CF79,FK75,KN77,KN78,Ma82},
and regularity results for
the classical {two-phase} Stefan problem 
are contained in   
\cite{ACS96,ACS98,CE83,DiB80,DiB82,Na87,Sa83,Zi82},
to list only a few references.
We remark that classical solutions
for the Stefan problem with condition \eqref{Gamma}
were first established in \cite{Han81,Me80}.
We also refer to \cite{PSS07} for a more detailed account of the literature concerning
the classical Stefan problem.

Although the Stefan problem with the
Gibbs-Thomson correction \eqref{GT} has been around
for many decades, only few analytical results 
concerning existence, regularity and qualitative properties of solutions are known.
The authors in \cite{FR91} consider the case
with small surface tension $0<\sigma\ll 1$ 
and linearize the problem about $\sigma=0$. 
Assuming the existence of smooth solutions for the case $\sigma=0$, that is, 
for the classical Stefan problem, the authors prove existence
and uniqueness of a weak solution for the 
\textit{linearized} problem and
then investigate the effect of small surface tension on
the shape of $\Gamma(t)$. 
Existence of long time weak solutions is established in \cite{AW00,Luc90,Rog04}.
A more detailed discussion will be given below. 
A proof for existence -- without uniqueness -- of  local time classical solutions 
is obtained in \cite{Ra91,Ra93}. In \cite{Me94}, the way in which a spherical ball of ice 
in a supercooled fluid melts down is investigated. 
The case of a strip-like geometry, where the free surface $\Gamma$ is given as the graph
of a function, is considered in \cite{EPS03}, and existence as well as uniqueness of  local time
classical solutions is established. Moreover, it is shown that solutions instantaneously regularize 
to become
analytic in space and time. The approach is based on the theory of maximal regularity,
which also forms the basis for the local existence theory in the current paper.
In \cite{PrSi08} linearized stability and instability of equilibria is studied.
Finally, the authors in \cite{HaGu10} consider a strip-like geometry over a torus
and establish asymptotic stability of flat surfaces.

If the diffusion equation $\partial_{t}v_i-\Delta v_i=0$ 
in $\Omega_i$ 
is replaced by the elliptic equation $\Delta v_i=0$,
then the resulting problem is the quasi-stationary 
Stefan problem with surface tension, which has also been 
termed the Mullins-Sekerka problem
(or the Hele-Shaw problem with surface tension). 
Existence, uniqueness, regularity (and global existence in some cases) 
of classical solutions for the 
quasi-stationary approximation has recently been investigated in 
\cite{Ba97,Che93,CHY96,ES96a,ES96b,ES97,ES98,FR01}.
Global existence of weak solutions has been established in \cite{Rog05},
see also \cite{LuSt97, GaSt98} for related results.

The challenge of developing efficient and accurate numerical methods
for free boundary problems arising from sharp-interface theories
has recently driven the development of regularized diffuse-interface, or phase field, theories.
It is of outmost importance to have a through understanding of the
sharp-interface models in order to evaluate the quality of predictions
of the associated phase field models.
As remarked in \cite{Gur07}, a phase field theory may in general possess
a variety of sharp-interface limits, and in the absence of a sound
sharp-interface theory to serve as a target, the problem of developing
a physically meaningful diffuse-interface theory is ill-posed.
\vspace{1mm}

b) In this paper we consider a general model for phase transitions that is thermodynamically consistent,
following the ideas in \cite{Gur07} and \cite{Ish06}, see also \cite{Gu86,Gu88,Gu88b} for earlier work. 
It involves the thermodynamic quantities of absolute temperature, free energy, internal energy, and entropy,
and is complemented by constitutive equations for the free energies and the heat fluxes in the bulk regions.
An important assumption is that there be no entropy production on the interface.
In particular, the interface is assumed to carry no mass and no energy
except for surface tension. 

To be more precise, let $\Omega\subset \R^{n}$ be a bounded domain of class $C^{2}$, $n\geq2$.
$\Omega$ is occupied by a material that can undergo phase changes: at time $t$, phase $i$ occupies
the subdomain $\Omega_i(t)$ of
$\Omega$, respectively, with $i=1,2.$
We assume that $\partial \Omega_1(t)\cap\partial \Omega=\emptyset$; this means
that no {\em boundary contact} can occur.
The closed compact hypersurface $\Gamma(t):=\partial \Omega_1(t)\subset \Omega$
forms the interface between the phases.
By the {\em Stefan problem with surface tension} we mean the following problem:
find a family of closed compact hypersurfaces $\{\Gamma(t)\}_{t\geq0}$ contained in $\Omega$
and an appropriately smooth function $u:\R_+\times\bar{\Omega}\rightarrow\R$ such that
\begin{equation}
\label{stefan}
\left\{\begin{aligned}
\kappa (u)\partial_t u-{\rm div}(d(u)\nabla u)&=0 &&\text{in}&&\Omega\setminus\Gamma(t)\\
\partial_{\nu_\Omega} u &=0 &&\text{on}&&\partial \Omega \\
[\![u]\!]&=0 &&\text{on}&&\Gamma(t)\\
[\![\psi(u)]\!]+\sigma \cH &=\gamma(u) V &&\text{on}&&\Gamma(t) \\
[\![d(u)\partial_\nu u]\!] &=(l(u)-\gamma(u)V) V &&\text{on}&&\Gamma(t)\\
 u(0)&=u_0 &&\\
 \Gamma(0)&=\Gamma_0. &&
\end{aligned}\right.
\end{equation}
Here
 $u(t)$ denotes the (absolute) temperature,
 $\nu(t)$ the outer normal field of $\Omega_1(t)$,
 $V(t)$ the normal velocity of $\Gamma(t)$,
 $\cH(t)=\cH(\Gamma(t))=-{\rm div}_{\Gamma(t)} \nu(t)/(n-1)$ the mean curvature of $\Gamma(t)$, and
 $[\![v]\!]=v_2|_{\Gamma(t)}-v_1|_{\Gamma(t)}$ the jump of a quantity $v$ across $\Gamma(t)$.
The sign of the mean curvature $\cH$ is chosen to be negative at a point $x\in\Gamma$ if
$\Omega_1\cap B_r(x)$ is convex for some sufficiently small $r>0$. Thus if $\Omega_1$ is a ball
of radius $R$ then $\cH=-1/R$ for its boundary $\Gamma$.

Several quantities are derived from the free energies $\psi_i(u)$ as follows:
\begin{itemize}
\item
 $\epsilon_i(u)= \psi_i(u)+u\eta_i(u)$, the internal energy in phase $i$,
\item
 $\eta_i(u) =-\psi_i^\prime(u)$, the entropy,
\item
 $\kappa_i(u)= \epsilon^\prime_i(u)=-u\psi_i^{\prime\prime}(u)>0$, the  heat capacity,
\item
$l(u)=u[\![\psi^\prime(u)]\!]=-u[\![\eta(u)]\!]$, the latent heat.
\end{itemize}
Furthermore, $d_i(u)>0$ denotes the coefficient of heat conduction in Fourier's law,
 $\gamma(u)\geq0$ the coefficient of  kinetic undercooling, and $\sigma>0$ the coefficient of surface tension.
As is commonly done, we assume that there exists a unique (constant) {\em melting temperature} $u_m$,
characterized by the equation $[\![\psi(u_m)]\!]=0.$
Finally, system \eqref{stefan} is to be completed by constitutive equations for
the free energies $\psi_i$ in the bulk phases $\Omega_i(t)$. 

In the sequel we drop the index $i$, as there is no danger of confusion; we just keep in mind that the coefficients depend on the phases.
The temperature is assumed to be continuous across the interface, 
as indicated by the condition $[\![u]\!]=0$ in \eqref{stefan}.
However, the free energy and the conductivities depend on the respective phases,
and hence the jumps $[\![\psi(u)]\!]$,
$[\![\kappa(u)]\!]$, $[\![\eta(u)]\!]$, $[\![d(u)]\!]$ are in general non-zero at the interface.
In this paper we assume that the coefficient of surface tension is constant. 

Next we show that the model \eqref{stefan} is consistent with the first and second law of thermodynamics, 
postulating conservation of energy and growth of entropy.
Indeed,
the {\sf total energy} of the system is given by
\begin{equation}
\label{energy}
{\sf E}(u,\Gamma) = \int_{\Omega\setminus\Gamma} \epsilon(u)dx +
\frac{1}{n-1}\int_\Gamma \sigma\, ds,
 \end{equation}
 and by the transport and surface transport theorem we have for smooth solutions
 \begin{align*}
 \frac{d}{dt}{\sf E}(u(t),\Gamma(t)) &= -\int_\Gamma \{[\![d(u)\partial_\nu u]\!] +[\![\epsilon(u)]\!]V
  + \sigma \cH V\}\,ds\\
&= -\int_\Gamma \{[\![d(u)\partial_\nu u]\!] -(l(u)-\gamma(u)V))V\}\,ds=0,
 \end{align*}
and thus, energy is conserved. 

Let us point out that it is essential
that $\sigma>0$ is constant, i.e.\ is independent of temperature.
The reason for this lies in the fact that in case
 $\sigma=\sigma(u)$ depends on the temperature, the surface energy will be $\int_\Gamma \epsilon_\Gamma(u)\,ds$ instead
 of $\int_\Gamma\sigma \,ds$, where $\epsilon_\Gamma(u)=\sigma(u)+u\eta_\Gamma(u)$, $\eta_\Gamma(u)=-\sigma^\prime(u)$, and one has to
take into account the surface entropy $\int_\Gamma \eta_\Gamma \,ds$ as well as balance of surface energy. This means
that the Stefan law needs to be replaced by a dynamic boundary condition of the form
\begin{equation}
\label{sigma(u)}
\kappa_\Gamma(u) \partial_{t,n} u - {\rm div}_\Gamma (d_\Gamma(u)\nabla_\Gamma u)
= [\![d\partial_\nu u]\!] -\big(l(u)-\gamma(u)V+l_\Gamma(u)\cH\big)V,
\end{equation}
where $\partial_{t,n}$ denotes the time derivative in normal
direction, $\kappa_\Gamma(u)=\epsilon^\prime_\Gamma(u)$ and
$l_\Gamma(u)= u\sigma^\prime(u)$. We intend to study such more
complex problems elsewhere and restrict our attention to the case
of constant $\sigma$ here.

The fifth equation in \eqref{stefan} is usually called the {\em
Stefan law}. It shows that energy is conserved across the interface.
The fourth equation is the {\em Gibbs-Thomson law} (with
kinetic undercooling if $\gamma(u)>0$) which implies together with
Stefan's law that entropy production on the interface is nonnegative
if $\gamma\geq0$. In case $\gamma\equiv 0$, i.e.\ in the absence of
kinetic undercooling,  there is no entropy production on the
interface. In fact,
the {\sf total entropy} of the system, given by
 \begin{equation}
 \label{entropy}
 \Phi(u,\Gamma)= \int_{\Omega\setminus\Gamma} \eta(u)dx,
 \end{equation}
 satisfies
 \begin{align*}
 \frac{d}{dt}\Phi(u(t),\Gamma(t)) &=\int_\Omega\frac{1}{u^2} d(u)|\nabla u|^2\,dx
 - \int_\Gamma \frac{1}{u}\{ [\![d(u)\partial_\nu u]\!]+u[\![\eta(u)]\!]V\}\,ds\\
 &=\int_\Omega \frac{1}{u^2}d(u)|\nabla u|^2\, dx
 + \int_\Gamma \frac{1}{u}\gamma(u) V^2\,ds\ge 0.
 \end{align*}
In particular, the negative total entropy is a Lyapunov functional
for problem \eqref{stefan}.
Even more,
$-\Phi$ is a strict Lyapunov functional in the sense that it is strictly decreasing along
 smooth solutions which are non-constant in time.
Indeed, if at some time
$t_0\geq0$ we have
 \begin{equation*}
 \frac{d}{dt}\Phi(u(t_0),\Gamma(t_0)) =\int_\Omega \frac{1}{u^2}d(u)|\nabla u|^2\, dx
 + \int_\Gamma \frac{1}{u}\gamma(u) V^2\,ds = 0,
 \end{equation*}
then $\nabla u(t_0)=0$
in $\Omega$ and $\gamma(u(t_0))V(t_0)=0$ on $\Gamma(t_0)$. This implies $u(t_0)=const$ in $\Omega$, hence $\cH(t_0)=-[\![\psi(u(t_0))]\!]/\sigma=const$. Since $\Omega$ is bounded, we may conclude that $\Gamma(t_0)$ is a union of finitely many, say $m$, disjoint spheres of equal radius,
i.e.\ $(u(t_0),\Gamma(t_0))$ is an equilibrium.
Therefore, the {\em limit sets} of solutions in the state manifold $\cSM_\gamma$,
see \eqref{phasemanif0}-\eqref{phasemanifg} for a definition, are contained in the
$(mn+1)$ dimensional manifold of equilibria
\begin{equation}
\label{equilibria-I}
\begin{split}
 \cE=\Big\{\big(u_\ast,\cup_{1\le l\le m}S_{R_\ast}(x_l)\big): &\,u_\ast>0,\
R_\ast=\sigma/[\![\psi(u_\ast)]\!],\ \bar B_{R_\ast}(x_l)\subset
\Omega,\\
&S_{R_\ast}(x_l)\cap S_{R_\ast}(x_k)=\emptyset,\,l\neq k \Big\},
\end{split}
\end{equation}
where $S_{R_\ast}(x_l)$ and $B_{R_\ast}(x_l)$ denote the sphere and the ball with radius $R_\ast$ and
center $x_l$, respectively.
\smallskip

c) Another interesting observation is the following. Consider the
critical points of the functional $\Phi(u,\Gamma)$ with constraint
${\sf E}(u,\Gamma)={\sf E}_0$, say on $C(\bar{\Omega})\times
\cMH^2(\Omega)$, see Section 3.1 for the definition of
$\cMH^2(\Omega)$. Then by the method of Lagrange multipliers, there
is $\mu\in\R$ such that at a critical point $(u_*,\Gamma_*)$ we have
\begin{equation}
\label{VarEq} 
\Phi^\prime(u_*,\Gamma_*)+\mu {\sf E}^\prime(u_*,\Gamma_*)=0.
\end{equation}
The derivatives of the functionals are given by
\begin{equation*}
\langle \Phi^\prime(u,\Gamma) |(v,h)\rangle =
(\eta^\prime(u)|v)_{L_2(\Omega)} - (
[\![\eta(u)]\!]|h)_{L_2(\Gamma)},
\end{equation*}
and
$$\langle {\sf E}^\prime(u,\Gamma) |(v,h)\rangle =
(\epsilon^\prime(u)|v)_{L_2(\Omega)} - ( [\![\epsilon(u)]\!]+\sigma
\cH(\Gamma)|h)_{L_2(\Gamma)}.$$ Setting first $h=0$ and varying $v$
in \eqref{VarEq} we obtain $ \eta^\prime(u_*) + \mu
\epsilon^\prime(u_*)=0$ in $\Omega$, and then varying $h$ we get
$$[\![\eta(u_*)]\!] +\mu\big([\![\epsilon(u_*)]\!]+\sigma \cH(\Gamma_*)\big)=0\text{ on $\Gamma_*$}.$$
The relations $\eta(u)=-\psi^\prime(u)$ and
$\epsilon(u)=\psi(u)-u\psi^\prime(u)$ imply
$0=-\psi^{\prime\prime}(u_*)(1+\mu u_*)$, and this shows that
$u_*=-1/\mu$ is constant in $\Omega$, since
$\kappa(u)=-u\psi^{\prime\prime}(u)>0$ for all $u>0$ by assumption.
This further implies $[\![\psi(u_*)]\!]+\sigma \cH(\Gamma_*)=0$,
i.e.\ the Gibbs-Thomson relation. Since $u_*$ is constant we see
that $\cH(\Gamma_*)$ is constant, hence $\Gamma_*$ is a sphere
whenever connected, and a union of finitely many disjoint spheres of
equal size otherwise. Thus the critical points of the entropy
functional for prescribed energy are precisely the equilibria of the
problem.

Going further, suppose we have an equilibrium $e_*:=(u_*,\Gamma_*)$
where the total entropy has a local maximum, w.r.t.\ the constraint
${\sf E}={\sf E}_0$ constant. Then $\cD:=[\Phi+\mu {\sf
E}]^{\prime\prime}(e_*)$ is negative semi-definite on the kernel of
${\sf E}^\prime(e_*)$, where $\mu$ is the fixed Lagrange multiplier
found above. The kernel of ${\sf E}^\prime(e)$ is given by the
identity
\begin{equation*}
\int_\Omega \kappa(u) v\, dx - \int_\Gamma\big([\![\epsilon(u)]\!] +
\sigma \cH(\Gamma)\big) h\, ds=0,
\end{equation*}
which at equilibrium yields
\begin{equation}
\label{kE}
 \int_\Omega \kappa_\ast v\, dx + \int_{\Gamma_\ast} l_\ast h\, ds=0,
\end{equation}
where $\kappa_\ast :=\kappa(u_\ast)$ and $l_\ast:=l(u_\ast)$.
On the other hand, a straightforward calculation yields with
$z=(v,h)$
\begin{equation}
\label{2var} 
-\langle \cD z|z\rangle
=\frac{1}{u_\ast}\left[\frac{1}{u_\ast}\int_\Omega\kappa_\ast
v^2\,dx - \sigma \int_{\Gamma_\ast}\! h\cdot\cH^\prime(\Gamma_*) h\,
ds\right].
\end{equation}
As  $\kappa_\ast$ and $\sigma$ are positive, we see that the form
$\langle \cD z|z\rangle$ is negative semi-definite  as soon as
$\cH^\prime ({\Gamma_*}) $ is negative semi-definite.
 We have
\begin{equation*}
\cH^\prime(\Gamma_\ast) = 1/{R^2_\ast}+(1/(n-1))\Delta_{\Gamma_\ast},
\end{equation*}
where $\Delta_{\Gamma_\ast}$ denotes the
Laplace-Beltrami operator on $\Gamma_\ast$ and $R_\ast$ means the radius of
the equilibrium sphere. To derive necessary conditions for an equilibrium $e_*$ to be a
local maximum of entropy, we consider two cases.
\smallskip\\
\noindent {1.} Suppose that $\Gamma_\ast$ is not connected, i.e.
$\Gamma_\ast$ is a union of $m$ disjoint spheres $\Gamma^k_\ast$.
Set $v=0$, and let $h=h_k\neq 0$ be constant on $\Gamma^k_\ast$ with $\sum_k
h_k=0$. Then the constraint \eqref{kE} holds, and
\begin{equation*}
\langle \cD z|z\rangle= (\sigma/u_\ast R^2_\ast)(|\Gamma_*|/m)
\,\sum_{k=1}^m h_k^2
>0,
\end{equation*}
hence $\cD$ cannot be negative semi-definite in this case. Thus if $e_\ast$
is an equilibrium with maximal total entropy, then $\Gamma_\ast$ must be
connected, and hence both phases are connected.
\smallskip\\
\noindent 2. Assume that $\Gamma_\ast$ is connected. Setting
$v=l_*/(\kappa_\ast|1)_\Omega$ and $h= -1/|\Gamma_\ast|$, we see
that the property that $\cD$ is negative semi-definite on the kernel of ${\sf E}^\prime(e_\ast)$ implies 
\begin{equation}
\label{var-sc} \zeta_\ast:=\frac{\sigma
u_\ast(\kappa_\ast|1)_\Omega}{l^2_\ast R^2_\ast |\Gamma_\ast|}\le 1.
\end{equation}
This is exactly the stability condition found in Theorem \ref{lin-stability}.

In summary, we obtain:
\begin{itemize}
\item
The equilibria of \eqref{stefan} are precisely the critical points
of the entropy functional with prescribed energy.
\vspace{1mm}
\item
The entropy functional with prescribed energy does not have a local maximum
$e_\ast=(u_\ast,\Gamma_\ast)$ in case $\Gamma_\ast$ is not connected.
\vspace{1mm}
\item
A necessary condition for a
critical point $e_\ast=(u_\ast,\Gamma_\ast)$ to be a local maximum of the
entropy functional with prescribed energy is that $\Gamma_\ast$ is connected and that
inequality \eqref{var-sc} holds.
\end{itemize}
It will be shown in Theorems \ref{lin-stability} and \ref{stability} below that
\begin{itemize}
\item
 $(u_*,\Gamma_*)\in\cE$ is stable if $\Gamma_*$ is connected and $\zeta_\ast<1$. 
\vspace{1mm}
\item
The latter is exactly the case if the reduced energy functional,
\begin{equation*}
[u\mapsto \varphi(u)={\sf E}(u, S_{R(u)}(x_0))],\quad R(u)=\sigma/[\![\psi(u)]\!],
\end{equation*}
has a strictly negative derivative at $u_\ast$.
\vspace{1mm}
\item
Any solution starting in a neighborhood of a stable equilibrium
exists globally and converges to another
stable equilibrium exponentially fast. 
\vspace{1mm}
\item
$(u_*,\Gamma_*)\in\cE$ is always unstable if $\Gamma_*$ is
disconnected, or if $\zeta_\ast>1$. \vspace{1mm}
\end{itemize}
Hence multiple spheres (of the same radius) are always unstable for \eqref{stefan}.
This situation is reminiscent of the onset of {\sl Ostwald ripening},
a process that manifests itself in the way that larger structures grow 
while smaller ones shrink and disappear.
Here we refer to  
\cite{AlFu03, AlFuKa03,AlFuKa04,AlFuKa04b,GORS09,HNO05},
\cite{Niet99}-\cite{NietVe04}, and the references therein
for various aspects and results on Ostwald ripening.
In particular, we mention that the authors in 
 \cite{AlFu03, AlFuKa03, AlFuKa04,AlFuKa04b} use the quasi-stationary Stefan problem
with surface tension (i.e., the Mullins-Sekerka problem) to model Ostwald ripening.
Under proper scaling assumptions, the way
sphere-like particles evolve is analyzed.
Interesting and illuminating connections between various versions 
of the Stefan problem (mostly the Mullins-Sekerka problem)
and Ostwald ripening are given in \cite{HNO05,Niet99,Niet00,Niet03,NietO01,Niet01b,NietVe04}.
It would be of considerable interest to also pursue the effect of coarsening in
the framework of the Stefan problem \eqref{stefan}.
\smallskip

d) 
Now we want to relate problem \eqref{stefan} to the pertinent Stefan problems that have been studied in the mathematical literature so far.  For this purpose we linearize
$h(u):=[\![\psi(u)]\!]$ near the {melting temperature} $u_m$, defined by $h(u_m)=0$. 
Then for the relative temperature $v=u-u_m$ we have
$h(u)\approx h^\prime(u_m) v $, hence with $l_m=l(u_m)$ and $\gamma_m=\gamma(u_m)$, the Gibbs-Thomson law approximately becomes
\begin{equation}
\label{approx-Stefan}
(l_m/u_m) v + \sigma \cH =\gamma_m V.
\end{equation}
This is the classical Gibbs-Thomson law (with kinetic undercooling in case $\gamma_m\!>\!0$).
Similarly, assuming that $u$ is close to $u_m$ and $V$ is small, the  Stefan law becomes approximately
\begin{equation}
\label{classic-Stefan}
[\![d\partial_\nu v]\!]=l_m V.
\end{equation}

As mentioned above, existence results for the Stefan problem with the
classical Gibbs-Thomson law $v=-\alpha \cH$ and the classical Stefan
law \eqref{classic-Stefan} in case $\kappa_1=\kappa_2$ can be found
in \cite{AW00, EPS03,FR91,HaGu10,Luc90,Me94,Ra91,Ra93, Rog04}. 
The Stefan problem with the linearized transmission conditions
\eqref{approx-Stefan}-\eqref{classic-Stefan} in case $\kappa_1=\kappa_2$ has been studied
in \cite{CR92,Ra91,Ra93,Vi85}, see also \cite{Kn07} for the
one-phase case.

In the recent publication 
\cite{Ha12} 
the author
also obtains nonlinear stability of single spheres for the Stefan problem 
with the linearized transmission conditions  \eqref{approx-Stefan}-\eqref{classic-Stefan} in 
case that all physical constants are taken to be 1, $\gamma_m=0$, and the (appropriately modified) 
stability condition $\zeta_\ast<1$ is satisfied.
The method relies on higher order energy estimates and requires higher order regularity
and compatibility conditions for the initial data, see also the remarks in e).

Our results go beyond the results in \cite{Ha12} in several significant ways:
we obtain existence and uniqueness results for arbitrary initial configurations,
with only minimal regularity assumptions on the data.
We also provide instability results, either in the case of connected equilibria 
with $\zeta_\ast>1$, or in the case of multiple spheres.
It should be mentioned that the linearized stability analysis 
of multiple spheres is considerably more involved
than the case of a single sphere.
Moreover, we allow for general material laws and we include kinetic undercooling.
Lastly, our setting allows for an interpretation of the equilibria in terms of the entropy
functional.

It should also be noted that the results and methods of our paper are not restricted to
the thermodynamically consistent Stefan problem, but also apply to the case
with linearized transmission conditions. In fact, the essential mathematcial difficulties
encountered in the stability-instability analysis of equilibria are already present in the latter 
situation.

Linear instability has been observed before in \cite{CJT95}
for the particular case that $\Omega=\R^2$, where suitable boundary conditions for the
temperature at infinity are imposed. It is then shown in \cite{CJT95} that equilibria are 
linearly unstable. This setting formally implies $\zeta_\ast=\infty$, which is in agreement
with the instability condition $\zeta_\ast>1$ of this paper.

If $\kappa_1=\kappa_2=0$ then we obtain a thermodynamically
consistent quasi-stationary approximation of the Stefan problem with
surface tension (and kinetic undercooling). Existence and
global existence of classical solutions for the quasi-stationary
approximation with \eqref{approx-Stefan} with $\gamma_m\neq 0$ and the classical Stefan
law \eqref{classic-Stefan} has been investigated in \cite{Wu96,Kn07}.


As mentioned before, we assume that the behavior in the bulk phases is described
by constitutive equations for the free energies $\psi_i(u)$.
A common assumption is that the heat capacities be constant 
and equal in the respective phases. Then we necessarily have
$0\equiv [\![\kappa]\!]= [\![\epsilon^\prime(u)]\!]=-u[\![\psi^{\prime\prime}(u)]\!],$
which implies that the function $h(u)=[\![\psi(u)]\!]$ is linear, i.e. $h(u)=h_0+h_1 u$, and then $l(u)= h_1u$.  The melting temperature is here given by $0<u_m=-h_0/h_1$.

If the heat capacities $\kappa_i$ are constant in the phases but not necessarily equal, the internal energies depend linearly on the temperature and the free and the inner energies are of the form
\begin{equation*}
\psi_i(u)= a_i+b_iu-\kappa_i u \ln u,
\quad \epsilon_i(u)=a_i+\kappa_iu,
\end{equation*}
hence $h(u) = \alpha +\beta u -\delta u\ln u$,
with constants $\alpha,\beta,\delta\in\R$. 
Concerning existence of equilibria, these special cases will be discussed in more detail in Section 4.

In \cite{Luc91} the author considers a Stefan problem based
on thermodynamical principles for the case where the internal energy
is given by 
$$e=w+\varphi\quad\text{with} \quad w=(1/u_m-1/u),$$ 
where the phase function $\varphi:\Omega\setminus\Gamma\to \{0,1\}$ 
assumes the distinct values $0$ and $1$ in the respective bulk phases,
and where $u$ and $u_m$ denote the absolute and the melting temperature.
In this situation, the free energy of the system can be described by
\begin{equation*}
\psi_i(u)={1}/{u_m}-{1}/{(2u)}+\varphi(1-{u}/{u_m}).
\end{equation*}
Setting $d_i(u)=1/u^2$,
the diffusion equation in the bulk phases, expressed for the new variable $w$, becomes
$\partial_t e=\Delta w$.
The total energy and total entropy are then given by
\begin{equation*}
\begin{split}
{\sf E}(w,\Gamma)=\int_{\Omega\setminus\Gamma} (w+\varphi)\,dx +\frac{\sigma|\Gamma|}{n-1}, 
\ {\Phi}(w,\Gamma)=\int_{\Omega\setminus\Gamma}\!\!\left\{-\frac{w^2}{2}
+\frac{w+\varphi}{u_m}-\frac{1}{2u_m^2}\right\}dx,
\end{split}
\end{equation*}
where $|\Gamma|$ denotes the surface area of $\Gamma$.
Therefore, the function
\begin{equation*}
{L}(w,\Gamma):=\int_{\Omega\setminus\Gamma} \frac{w^2}{2}\,dx+\frac{\sigma |\Gamma |}{(n-1)u_m}
=-{\Phi}(w,\Gamma) + \frac{{\sf E}(w,\Gamma)}{u_m}
-\frac{|\Omega|}{2u^2_m},
\end{equation*}
termed total entropy in \cite{Luc91}, is a Lyapunov functional for the system.

On a more elementary and ad-hoc level, assuming equal and constant heat capacities 
$\kappa=\kappa_i$, constant latent heat $l_m$, and constant heat conductivity coefficients
$d_i$, one can assign to the Stefan problem subject to the classical 
conditions \eqref{approx-Stefan}--\eqref{classic-Stefan} 
an energy and a Lyapunov functional through the relations
\begin{equation*}
{\sl E}(v,\Gamma)=\int_{\Omega\setminus\Gamma} (\kappa v +l_m\varphi)\, dx,
\quad
L(v,\Gamma)=\int_{\Omega\setminus\Gamma} \kappa v^2/2\,dx + \sigma u_m |\Gamma|/(n-1),
\end{equation*}
where $v=u-u_m$ denotes the relative temperature, and where the function 
$\varphi$ has the same meaning as above, see also 
\cite{PrSi08} for the case  $\kappa_1\neq \kappa_2$.
The Lyapunov functional $L$ plays an important role in the construction of long time weak solutions
in \cite{Luc90,Luc91}, see also \cite{Rog04}. 
The authors in \cite{Luc90,Luc91,Rog04} 
consider the Stefan problem subject to the linearized transmission conditions 
\eqref{approx-Stefan}--\eqref{classic-Stefan} with $\gamma_m=0$, and they assume
equal and constant heat capacities $\kappa_i$, constant latent heat, and constant
heat conduction coefficients. 
The weak solutions obtained  exist on any given, fixed time interval
$(0,T)$ and have the feature that they lead to a sharp interface $\Gamma(t)$,
in contrast to the weak solutions previously obtained in \cite{Vi83,Vi89}.
A serious drawback of the results in
\cite{Luc90,Luc91,Rog04,Vi83,Vi89} is caused by the lack of uniqueness of solutions. 
This renders further assertions concerning asymptotic properties of solutions
rather difficult, if not impossible.
\smallskip

e) The novelty of our contribution lies in the fact that
we consider rather general phase transition models that are thermodynamically consistent.
In particular, we allow for different heat capacities, and kinetic undercooling
can be included in the model.
In the mathematical literature it is commonly assumed that the heat capacities $\kappa_i$
be equal. However, this assumption is somewhat questionable, as it implies that the internal energies $\eps_i$
can only differ by a constant.

We obtain unique strong solutions, but existence is only guaranteed for short time intervals.
This, however, is to be expected, as  solutions can develop singularities in finite time,
say in the way that topological changes in the geometry may occur.
We give a complete analysis for the equilibrium states of \eqref{stefan},
and we investigate the asymptotic behavior of solutions that
start out close to equilibria.
It is of significant interest to note that the equilibrium states can be characterized
as the critical points of the total entropy subject to the constraint that the total energy be
conserved. Moreover, we obtain that the equilibrium case where the dispersed phase consists of multiple
balls (necessarily of the same radius) always leads to an unstable configuration.
As already mentioned, this is reminiscent of the onset of Ostwald ripening.
Additionally, we prove that solutions exist globally and have relatively compact orbits, 
provided they do not exhibit singularities, see Theorem~\ref{Qual}.
It appears that this manuscript is the first work to provide
such qualitative results for a thermodynamically consistent Stefan problem.

A major difficulty in the mathematical treatment of
the Stefan problem \eqref{stefan} is due to the fact that the boundary
$\Gamma(t)$, and thus the geometry, is unknown and ever changing.
A widely used method to overcome this inherent difficulty
is to choose a fixed reference surface $\Sigma$ and then represent
the moving surface $\Gamma(t)$ as the graph of a function
in normal direction of $\Sigma$.
This way, one obtains a time-dependent (unknown) diffeomorphism
from $\Sigma$ onto $\Gamma(t)$, and in a next step
 this diffeomorphism is extended to a diffeomorphism of fixed
reference regions $\Omega_i^\Sigma $ onto the unknown domains $\Omega_{i}(t)$.
The treatment of the free boundary problem \eqref{stefan}
then proceeds by transforming the equations into a new
system of equations defined on the fixed domain
$\Omega\setminus\Sigma$ from which both the solution
and the parameterizing function have to be determined.
In the context of the Stefan problem this approach has first been
used by Hanzawa~\cite{Han81}.
This step in carried out in Section~2.

Section~3 is devoted to results on local
well-posedness for problem \eqref{stefan}, based on the approach in
\cite{EPS03} and \cite{DPZ08}. 
We show that solutions do  not lose regularity. Thus, solutions give rise to a semiflow in the state manifold $\cSM_\gamma$, 
and this property allows us to use methods from the theory of dynamical systems
to further investigate geometric properties of solutions, such as the structure
of the $\omega$-limit set, and convergence results for global solutions, see 
for instance Theorem~\ref{Qual}.

In Section 4
we discuss  equilibria and their linear stability properties. Here
we rely on previous work in \cite{PrSi08}.
However, we should like to point out that the stability results 
given here are considerably more general than those in \cite{PrSi08},
where we only considered  the situation of a connected disperse phase.

In Section~5
we establish the corresponding stability properties for the
nonlinear problem, employing the {\em generalized principle of
linearized stability}, extending the results of \cite{PSZ09} to the
situation considered here. The main result of this section shows
convergence of solutions to an equilibrium which start out near
stable equilibria. 
Moreover, we give a rigorous proof of the instability result.
The main difficulty in proving the stability result lies in the fact
that equilibria are not isolated, but rather form a manifold,
caused by the fact that the equilibrium problem is invariant under
translations and scaling.  
This implies that the standard approach of linearized stability cannot be
applied directly, and for this reasons we need to apply ideas developed
in \cite{PSZ09}.
We emphasize, however, that the situation considered here 
is much more complicated than in \cite{PSZ09},
as the compatibility conditions implied by line 4 of equation (1.3)  
force us to work in a nonlinear manifold
(the state manifold $\cSM_\gamma$),
rather than in an open subset of a vector space.

In order to prove the stability/instabilty results, we parametrize
$\cSM_\gamma$ locally over the tangent space $\tilde Z_\gamma$
in a neighborhood of $(0,0)$.
The flow is then decomposed into a 
part $\tilde z$ evolving in the tangent space $\tilde Z_\gamma$,
and a small component $\bar z$ that corresponds to the image
of $\tilde z$ under the parametrization of $\cSM_\gamma$.
The flow of $\tilde z$ is driven by the linearized operator $L_\gamma$
and the right hand side of \eqref{nonlin-param-red}, which couples the dynamics of 
$\tilde z$ and $\bar z$. 
The part of the flow in the tangent space $\tilde Z_\gamma$
is then further decomposed into a part ${\sf x}$, corresponding to the finite dimensional
projection of $\tilde z$ onto the kernel $N(L_\gamma)$ of $L_\gamma$,
and a part ${\sf y}$ in the stable subspace, which basically measures the deviation
of $\tilde z$ from the manifold of equilibria ${\cE}$ in the stable direction.
It is important to note that
$N(L_\gamma)$ coincides with the tangent space of the manifold $\cE$ of equilibria at
a fixed equilibrium $e_\ast$.
All the equations are coupled, and a contraction mapping argument is employed
to obtain the existence of global in time, exponentially decreasing solutions. 

The stability result in \cite{Ha12} has been obtained using the method of higher energy estimates
as well as suitable orthogonality conditions to deal with the nontrivial kernel 
$N(L_\gamma)$ of the linearization. A related idea of flow decomposition
is used:  the flow is decomposed into a finite-dimensional component and a remaining
infinite dimensional component which, in a suitable sense, is transversal to 
the equilibrium manifold $\cE$.
Related ideas have also been used in the study of long-time asymptotics
for the Mullins-Sekerka model \cite{AlFu03, ES98},  
or in the long-time analysis of some curvature driven flows \cite{Str02},
and of solitons, see for instance the review in \cite{Tao09}.

\smallskip

Of ultimate importance is the Lyapunov functional
for (\ref{stefan}), which is given by the negative total entropy
$-\Phi(u,\Gamma)$. It takes bounded global-in-time solutions to the
set of equilibria, and then by the results of Section 5 and
relative compactness of the orbits, any such solution must converge towards an
equilibrium in the topology of the state manifold $\cSM_\gamma$
provided it comes close to a stable equilibrium.

Our analysis is carried out in the framework of
$L_p$-spaces, with $n+2<p<\infty$. We expect that it would be enough
to require $(n+2)/2<p<\infty$  (so unfortunately $p>2$ even in
$2D$!), but for the sake of simplicity we restrict ourselves here to
the stronger assumption $p>n+2$. We also expect that a similar
analysis can be obtained in the framework of the little H\"older
spaces $h^\alpha$, which would, though, require higher order compatibility
conditions.
\smallskip

f) 
Finally, we we would like to address open problems and directions of future research.
We are confident that our approach based on maximal regularity is flexible and general enough
to also investigate more complex models that take {\em surface energy} into consideration.
In fact, the case where the surface tension depends on the temperature
has recently been considered in \cite{PSW11}.
Of considerable interest is also
the case where the Gibbs-Thomson law is replaced by 
$$ [\![\psi(u)]\!]+\sigma \cH =\gamma(u) V- {\rm div}_\Gamma[\alpha(u)\nabla_\Gamma(V/u)]-u\,{\rm div}_\Gamma[\beta(u)\nabla_\Gamma V],$$
with $\alpha,\beta>0$, see \cite{Gur07} for more background information.
The modified Stefan law then reads
\[
[\![d(u)\partial_\nu u]\!] =\Big(l(u)-\gamma(u)V+{\rm
div}_\Gamma[\alpha(u)\nabla_\Gamma(V/u)]+u\,{\rm
div}_\Gamma[\beta(u)\nabla_\Gamma V]\Big) V
\]
and the resulting entropy production becomes
\begin{align*}
 \frac{d}{dt}\Phi(u(t),\Gamma(t)) &=
 \int_\Omega \frac{1}{u^2}d(u)|\nabla u|^2\, dx
 + \int_\Gamma \frac{1}{u}\gamma(u) V^2\,ds\\
 &+ \int_\Gamma [\alpha(u) |\nabla_\Gamma (V/u)|^2+\beta(u)|\nabla_\Gamma V|^2]|\,ds\geq0.
\end{align*}
 
Another direction concerns the situation where the densities of the respective phases
are different. This case results in the occurrence of so-called {\em Stefan currents},
and the corresponding models need to also involve the equations of fluid dynamics,
see for instance \cite{Gur07}.
Additional very interesting problems concern phase transitions in moving viscous fluids,
in which case the motion can be modeled by a thermodynamically consistent Stefan problem
coupled to the Navier-Stokes equations, 
see \cite{Gur07,Ish06}.
First results in this direction are contained in \cite{PSSS11,PS11}.
Of even greater challenge is the case where one fluid is {\em evaporating},
leading to a phase-transition model which couples the equations of phase transitions
to the equations of fluid dynamics with a compressible fluid phase.

As mentioned above, it would be interesting to link 
the thermodynamically consistent Stefan problem \eqref{stefan}
to the occurrence of {\em Ostwald ripening}.

Of considerable interest is also a better understanding of
the occurrence of {\em singularities} for solutions of \eqref{stefan}.
A preliminary result ensuring global existence is contained in Theorem~\ref{Qual}
under rather restrictive assumptions.
We conjecture
that the only obstruction against global existence is related to the break down of the geometry:
if no topological changes take place and the curvatures stay bounded, then the solution exists globally.
More specifically,  we conjecture that assumptions (i), (ii) and (iii) in Theorem~\ref{Qual} 
follow from (iv), provided that at time $t=0$ we have $u_0>0$, and
$l(u_0)\neq0$ in $\bar{\Omega}$ in case $\gamma\equiv0$.

An additional direction that is of great relevance concerns  {\em triple junctions},
for instance the case when the free surface $\Gamma(t)$ is in contact with the solid container wall.
While the situation where $\Gamma(t)$ meets the container wall orthogonally can
likely be handled with the methods developed in this paper (by reflection arguments),
the case of arbitrary {\em contact angles} remains a significant challenge.
For progress in the case of the Hele-Shaw problem with surface tension we refer to
\cite{KM12a,KM12b}.
\section{Transformation to a Fixed Interface}

Let $\Omega\subset\R^n$ be a bounded domain with boundary $\partial
\Omega$ of class $C^2$, and suppose $\Gamma\subset \Omega$ is a
closed hypersurface of class $C^2$, i.e.\ a $C^2$-manifold which is
the boundary of a bounded domain $\Omega_1\subset \Omega$. We then
set $\Omega_2=\Omega\setminus\bar{\Omega}_1$. Note that while
$\Omega_2$ is connected, $\Omega_1$ may be  disconnected. However,
$\Omega_1$ consists of finitely many components only, as $\partial
\Omega_1=\Gamma$ by assumption is a manifold, at least  of class
$C^2$. Recall that the {\em second order bundle} of $\Gamma$ is
given by
$$\cN^2\Gamma:=\{(p,\nu_\Gamma(p),L_\Gamma(p)):\, p\in\Gamma\}.$$
Note that the Weingarten map $L_\Gamma$ (also called the shape operator,
or the second fundamental tensor) is defined by
$$ L_\Gamma(p) = -\nabla_\Gamma \nu_\Gamma(p),\quad p\in\Gamma ,$$
where $\nabla_\Gamma$ denotes the surface gradient on $\Gamma$.
The eigenvalues $\kappa_j(p)$ of $L_\Gamma(p)$ are the principal curvatures of $\Gamma$ at $p\in\Gamma$, and we have
$|L_\Gamma(p)|=\max_j|\kappa_j(p)|.$
The mean curvature $\cH_\Gamma(p)$ is given by
$$(n-1)\cH_\Gamma(p) = \sum_{j=1}^{n-1} \kappa_j(p)={\rm tr} L_\Gamma(p) = -{\rm div}_\Gamma \nu_\Gamma(p),$$
where ${\rm div}_\Gamma$ means surface divergence. Recall also that the
{\em Hausdorff distance} $d_H$ between the two closed subsets $A,B\subset\R^m$
is defined by
$$d_H(A,B):= \max\big\{\sup_{a\in A}{\rm dist}(a,B),\sup_{b\in B}{\rm dist}(b,A)\big\}.$$
Then we may approximate $\Gamma$ by a real analytic hypersurface $\Sigma$ (or merely $\Sigma\in C^3$),
in the sense that the Hausdorff distance of the second order bundles of
$\Gamma$ and $\Sigma$ is as small as we want. More precisely, for each $\eta>0$ there is a real analytic closed
hypersurface such that
$d_H(\cN^2\Sigma,\cN^2\Gamma)\leq\eta$. If $\eta>0$ is small enough, then $\Sigma$
bounds a domain $\Omega_1^\Sigma$ with  $\overline{\Omega^\Sigma_1}\subset\Omega$,
and we set $\Omega^\Sigma_2=\Omega\setminus\bar{\Omega}^\Sigma_1$.

It is well known that such a hypersurface $\Sigma$ admits a tubular neighborhood,
which means that there is $a>0$ such that the map
\begin{eqnarray*}
&&\Lambda:\, \Sigma \times (-a,a)\to \R^n \\
&&\Lambda(p,r):= p+r\nu_\Sigma(p)
\end{eqnarray*}
is a diffeomorphism from $\Sigma \times (-a,a)$
onto $\cR(\Lambda)$. The inverse
$$\Lambda^{-1}:\cR(\Lambda)\mapsto \Sigma\times (-a,a)$$ of this map
is conveniently decomposed as
$$\Lambda^{-1}(x)=(\Pi(x),d_\Sigma(x)),\quad x\in\cR(\Lambda).$$
Here $\Pi(x)$ means the nonlinear orthogonal projection of $x$ to $\Sigma$ and $d_\Sigma(x)$ the signed
distance from $x$ to $\Sigma$; so $|d_\Sigma(x)|={\rm dist}(x,\Sigma)$ and $d_\Sigma(x)<0$ iff
$x\in \Omega_1^\Sigma$. In particular we have $\cR(\Lambda)=\{x\in \R^n:\, {\rm dist}(x,\Sigma)<a\}$.

On the one hand, $a$ is determined by the curvatures of $\Sigma$, i.e.\ we must have
$$0<a<\min\big\{1/|\kappa_j(p)|: j=1,\ldots,n-1,\; p\in\Sigma\big\},$$
where $\kappa_j(p)$ mean the principal curvatures of $\Sigma$ at $p\in\Sigma$.
But on the other hand, $a$ is also connected to the topology of $\Sigma$,
which can be expressed as follows. Since $\Sigma$ is a compact (smooth) manifold of
dimension $n-1$ it satisfies a (interior and exterior) ball condition, which means that
there is a radius $r_\Sigma>0$ such that for each point $p\in \Sigma$
there are $x_j\in \Omega_j^\Sigma$, $j=1,2$, such that $B_{r_\Sigma}(x_j)\subset \Omega_j^\Sigma$, and
$\bar{B}_{r_\Sigma}(x_j)\cap\Sigma=\{p\}$. Choosing $r_\Sigma$ maximal,
we then must also have $a<r_\Sigma$.
In the sequel we fix
$$ a= \frac{1}{2}\min\left\{r_\Sigma, \frac{1}{|\kappa_j(p)|}, \, j=1,\ldots, n-1,\; p\in\Sigma\right\}.$$

For later use we note that the derivatives of $\Pi(x)$ and $d_\Sigma(x)$ are given by
$$
\nabla d_\Sigma(x)= \nu_\Sigma(\Pi(x)),\quad \Pi^\prime(x) = M_0(d_\Sigma(x),\Pi(x))P_\Sigma(\Pi(x)),
$$
where $P_\Sigma(p)=I-\nu_\Sigma(p)\otimes\nu_\Sigma(p)$ denotes the orthogonal projection onto
the tangent space $T_p\Sigma$ of $\Sigma$ at $p\in\Sigma$, and
$M_0(r,p)=(I-r L_\Sigma(p))^{-1}$.
Note that $|M_0(r,p)|\leq 1/(1-r|L_\Sigma (p)|)\leq 2$ for all $|r|\leq a$ and $p\in\Sigma$.

Setting $\Gamma=\Gamma(t)$, we may use the map $\Lambda$ to parametrize the unknown free
boundary $\Gamma(t)$ over $\Sigma$ by means of a height function $\rho(t,p)$ via
$$\Gamma(t): [p\mapsto p+\rho(t,p)\nu_\Sigma(p)],\quad p\in\Sigma,\; t\geq0,$$
for small $t\geq0$, at least.
Extend this diffeomorphism to all of $\bar{\Omega}$ by means of
$$ \Xi_\rho(t,x) = x +\chi(d_\Sigma(x)/a)\rho(t,\Pi(x))\nu_\Sigma(\Pi(x))=:x+\theta_\rho(t,x).$$
Here $\chi$ denotes a suitable cut-off function; more precisely, $\chi\in\cD(\R)$,
$0\leq\chi\leq 1$, $\chi(r)=1$ for $|r|<1/3$, and $\chi(r)=0$ for $|r|>2/3$.
Note that $\Xi_\rho(t,x)=x$ for $|d_\Sigma(x)|>2a/3$, and
$$\Xi_\rho^{-1}(t,x)= x-\rho(t,\Pi(x))\nu_\Sigma(\Pi(x))\quad \mbox{ for }\; |d_\Sigma(x)|<a/3.$$
In particular,
$$\Xi_\rho^{-1}(t,x)= x-\rho(t,x)\nu_\Sigma(x)\quad \mbox{ for }\; x\in\Sigma.$$
 Setting $v(t,x)=u(t,\Xi_\rho(t,x))$, or
$u(t,x)= v(t,\Xi_\rho^{-1}(t,x))$ we have this way transformed the
time varying regions $\Omega\setminus \Gamma(t)$ to the fixed domain
$\Omega\setminus\Sigma$. This is the direct mapping method, also
called Hanzawa transformation.

By means of this transformation, we obtain the following transformed
problem.
\begin{equation}
\label{transformed}
\left\{\begin{aligned}
\kappa(v)\partial_t v +\cA(v,\rho)v&=\kappa(v)\cR(\rho)v
         &&\text{in}&&\Omega\setminus\Sigma\\
\partial_{\nu_\Omega} v&=0 &&\text{on}&&\partial \Omega\\
 [\![v]\!]&=0 &&\text{on}&&\Sigma \\
[\![\psi(v) ]\!] + \sigma \cH(\rho)&= \gamma(v)\beta(\rho)\partial_t\rho
  &&\text{on}&&\Sigma\\
\{l(v)-\gamma(v)\beta(\rho)\partial_t\rho\}\beta(\rho)\partial_t\rho +\cB(v,\rho)v &= 0
 &&\text{on}&&\Sigma\\
v(0)=v_0,\ \rho(0)&=\rho_0.&&
\end{aligned}\right.
\end{equation}
Here $\cA(v,\rho)$ and $\cB(v,\rho)$ denote the transformations of $-{\rm div}(d\nabla )$ and
$-[\![d\partial_\nu]\!]$, respectively. Moreover, $\cH(\rho)$ means the mean curvature of $\Gamma$,
$\beta(\rho)=(\nu_\Sigma|\nu_\Gamma(\rho))$, the term $\beta(\rho)\partial_t\rho$ represents the normal velocity $V$,
and
$$\cR(\rho)v=\partial_t v -\partial_tu\circ \Xi_\rho.$$
\\
The system (\ref{transformed}) is a quasi-linear parabolic problem
on the domain $\Omega$  with fixed interface $\Sigma\subset \Omega$
with a {\em dynamic boundary condition}, namely the fifth equation
which describes the evolution of the interface $\Gamma(t)$.

To elaborate on the structure of this problem in more detail, we calculate
$$\Xi_\rho^\prime = I + \theta_\rho^\prime, \quad \quad \Xi_\rho^{\prime-1} = I - {[I + \theta_\rho^\prime]}^{-1}\theta_\rho^\prime
=:I-M_1(\rho)^{\sf T}$$
and
\begin{align*}
\nabla u\circ\Xi_\rho
=  [{(\Xi_\rho^{-1})}^{\prime{\sf T}}\circ\Xi_\rho]\nabla v
= (I - M_1(\rho))\nabla v,
\end{align*}
and for a vector field $q= \bar{q}\circ \Xi_\rho$
\begin{align*}
(\nabla|\bar{q})\circ\Xi_\rho
=  ([{(\Xi_\rho^{-1})}^{\prime{\sf T}}\circ\Xi_\rho]\nabla|q)
=  ((I - M_1(\rho))\nabla|q).
\end{align*}
Further we have
\begin{align*}
\partial_t u\circ\Xi_\rho
&=  \partial_t v -(\nabla u\circ \Xi_\rho|\partial_t\Xi_\rho)
 =  \partial_t v -( [{(\Xi_\rho^{-1})}^{\prime{\sf T}}\circ\Xi_\rho]\nabla v|\partial_t\Xi_\rho) \\
&=  \partial_t v -(\nabla  v|{[I +
\theta_\rho^\prime]}^{-1}\partial_t\theta_\rho ),
\end{align*}
hence
$$\cR(\rho)v=(\nabla  v|{[I + \theta_\rho^\prime]}^{-1}\partial_t\theta_\rho ).$$
With the Weingarten map $L_\Sigma=-\nabla_\Sigma\nu_\Sigma$  we have
\begin{equation*}
\begin{aligned}
\nu_\Gamma(\rho)&= \beta(\rho)(\nu_\Sigma-\alpha(\rho)),&& \alpha(\rho)= M_0(\rho)\nabla_\Sigma \rho,\\
M_0(\rho)&=(I-\rho L_\Sigma)^{-1},&& \beta(\rho) = (1+|\alpha(\rho)|^2)^{-1/2},
\end{aligned}
\end{equation*}
and
$$V=(\partial_t\Xi|\nu_\Gamma) = (\nu_\Sigma|\nu_\Gamma(\rho))\partial_t \rho
=\beta(\rho)\partial_t \rho.$$
Employing this notation leads to $\theta_\rho^\prime=0$ for $|d_\Sigma(x)|>2a/3$ and
\begin{align*}
\theta_\rho^\prime(t,x)&= \frac{1}{a}\chi'(d_\Sigma(x)/a)\rho(t,\Pi(x))\nu_\Sigma(\Pi(x))\otimes\nu_\Sigma(\Pi(x))\\
& + \chi(d_\Sigma(x)/a)[\nu_\Sigma(\Pi(x))\otimes M_0(d_\Sigma(x))\nabla_\Sigma \rho(t,\Pi(x))]\\
&-\chi(d_\Sigma(x)/a)\rho(t,\Pi(x))L_\Sigma(\Pi(x))M_0(d_\Sigma(x))P_\Sigma(\Pi(x))
\end{align*}
for $0\leq|d_\Sigma(x)|\leq 2a/3$.
In particular, for $x\in \Sigma$ we have
$$\theta_\rho^\prime(t,x)=\nu_\Sigma(x)\otimes\nabla_\Sigma\rho(t,x) -\rho(t,x)L_\Sigma(x)P_\Sigma(x), $$
and
$$(\theta_\rho^\prime)^{\sf T}(t,x)=\nabla_\Sigma\rho(t,x)\otimes\nu_\Sigma(x) -\rho(t,x)L_\Sigma(x), $$
since $L_\Sigma(x)$ is symmetric and has range in $T_x\Sigma$.
Therefore, $[I + \theta_\rho^\prime]$ is boundedly invertible, if $\rho$ and $\nabla_\Sigma \rho$
are sufficiently small, and
\begin{equation*}
\label{hanzawainv}
|{[I + \theta_\rho^\prime]}^{-1}| \leq 2 \quad \mbox{ for }\;
{|\rho|}_\infty \leq \frac{1}{4 ({|\chi'|}_\infty/a+ 2\max_j |\kappa_j|)},
\quad {|\nabla_\Sigma \rho|}_\infty \leq \frac{1}{8}.
\end{equation*}
For the mean curvature $\cH(\rho)$ we have
$$ (n-1)\cH(\rho) = \beta(\rho)\{ {\rm tr} [M_0(\rho)(L_\Sigma+\nabla_\Sigma \alpha(\rho))]
-\beta^2(\rho)(M_0(\rho)\alpha(\rho)|[\nabla_\Sigma\alpha(\rho)]\alpha(\rho))\},$$
an expression involving second order derivatives of $\rho$ only linearly.
Its linearization at $\rho=0$ is given by
$$(n-1)\cH^\prime(0)= {\rm tr}\, L_\Sigma^2 +\Delta_\Sigma.$$
Here $\Delta_\Sigma$ denotes the Laplace-Beltrami operator on $\Sigma$.
The operator $\cB(v,\rho)$ becomes
\begin{align*}
\cB(v,\rho)v&= -[\![d(u)\partial_\nu u]\!]\circ\Xi_\rho=-([\![d(v)(I-M_1(\rho))\nabla v]\!]|\nu_\Gamma)\\
&= -\beta(\rho)([\![d(v)(I-M_1(\rho))\nabla v]\!]|\nu_\Sigma -\alpha(\rho))\\
&= -\beta(\rho)[\![d(v)\partial_{\nu_\Sigma} v]\!]
+\beta(\rho)([\![d(v)\nabla v]\!]|(I-M_1(\rho))^{\sf
T}\alpha(\rho)),
\end{align*}
since $M_1^{\sf T}(\rho)\nu_\Sigma = 0$, and finally
\begin{align*}
\cA(v,\rho)v= & -{\rm div}( d(u)\nabla u)\circ\Xi_\rho= -((I-M_1(\rho))\nabla|d(v)(I-M_1(\rho))\nabla v)\\
= & -d(v)\Delta v + d(v)[M_1(\rho)+M_1^{\sf T}(\rho)-M_1(\rho)M_1^{\sf T}(\rho)]:\nabla^2 v\\
&-d^\prime(v)|(I-M_1(\rho))\nabla v|^2
 + d(v)((I-M_1(\rho)):\nabla M_1(\rho)|\nabla v).
\end{align*}
We recall that for matrices  $A,B\in\R^{n\times n}$, $
A:B=\sum_{i,j=1}^n a_{ij}b_{ij}=\text{tr}\,(AB^{\sf T}) $ denotes
the inner product.

Obviously, the leading part of $\cA(v,\rho)v$ is $-d(v)\Delta v$,
while the leading part of $\cB(v,\rho)v$ is
$-\beta(\rho)[\![d(v)\partial_{\nu_\Sigma} v]\!]$, as $M_1(0)=0$ and
$\alpha(0)=0$; recall that we may assume $\rho$ small in the
$C^2$-norm. It is important to recognize the quasilinear structure
of \eqref{transformed}: derivatives of highest order only appear
linearly in each of the equations.
\section{Local Well-Posedness}
The basic result for local well-posedness in the absence of kinetic undercooling
in an $L_p$-setting is the following.
\begin{theorem} 
\label{Thm3.1}
{\rm ($\gamma \equiv 0$).}
\label{wellposed1} Let $p>n+2$, $\gamma=0$, $\sigma>0$. Suppose
$\psi\in C^3(0,\infty)$, $d\in C^2(0,\infty)$ such that
$$\kappa(u)=-u\psi^{\prime\prime}(u)>0,\quad d(u)>0,\quad u\in(0,\infty).$$
Assume the {\em regularity conditions}
$$u_0\in W^{2-2/p}_p(\Omega\setminus\Gamma_0)\cap C(\bar{\Omega}),\quad u_0>0,
\quad \Gamma_0\in W^{4-3/p}_p,$$
 the {\em  compatibility conditions}
$$
\partial_{\nu_\Omega}u_0=0,\quad
[\![\psi(u_0)]\!]+\sigma \cH(\Gamma_0)=0,
\quad [\![d(u_0)\partial_{\nu_{\Gamma_0}} u_0]\!]\in W^{2-6/p}_p(\Gamma_0),$$
and the {\em well-posedness condition}
$$\quad l(u_0)\neq0\quad\mbox{on}\; \Gamma_0.$$
Then there exists a unique $L_p$-solution for the Stefan problem with surface tension \eqref{stefan}
on some possibly small but nontrivial time interval $J=[0,\tau]$.
\end{theorem}
\noindent Here the notation $\Gamma_0\in W^{4-3/p}_p$ means that
$\Gamma_0$ is a $C^2$-manifold, such that its (outer) normal field
$\nu_{\Gamma_0}$ is of class $W^{3-3/p}_p(\Gamma_0)$. Therefore the
Weingarten tensor
 $L_{\Gamma_0}=-\nabla_{\Gamma_0}\nu_{\Gamma_0}$ of $\Gamma_0$ belongs to $W^{2-3/p}_p(\Gamma_0)$ which embeds into
$C^{1+\alpha}(\Gamma_0)$, with $\alpha=1-(n+2)/p>0$ since $p>n+2$ by
assumption. For the same reason, we also have $u_0\in
C^{1+\alpha}(\bar{\Omega})$, and $V_0\in C^{2\alpha}(\Gamma_0)$. The
notion $L_p$-solution means that $(u,\Gamma)$ is obtained as the
push-forward of an $L_p$-solution $(v,\rho)$ of the transformed
problem \eqref{transformed}. This class  will be discussed below.

There is an analogous result in the presence of kinetic undercooling which reads as follows.
\begin{theorem} 
\label{Thm3.2}
{\rm ($\gamma >0$).}
\label{wellposed2} Let $p>n+2$, $\sigma>0$, and suppose
$\psi,\gamma\in C^3(0,\infty)$, $d\in C^2(0,\infty)$ such that
$$\kappa(u)=-u\psi^{\prime\prime}(u)>0,\quad d(u)>0,\quad \gamma(u)>0,\quad u\in(0,\infty).$$ Assume
the {\em regularity conditions}
$$u_0\in W^{2-2/p}_p(\Omega\setminus\Gamma_0)\cap C(\bar{\Omega}),\quad u_0>0,\quad \Gamma_0\in W^{4-3/p}_p,$$
and the {\em compatibility conditions}
\begin{equation*}
\begin{split}
\partial_{\nu_\Omega}u_0=0, \;\;
\big([\![\psi(u_0)]\!]+\sigma \cH(\Gamma_0)\big)\big(l(u_0)-[\![\psi(u_0)]\!]-\sigma \cH(\Gamma_0)\big)=\gamma(u_0)[\![d(u_0)\partial_\nu u_0]\!].
\end{split}
\end{equation*}
\noindent
Then there exists a unique $L_p$-solution of the Stefan problem with surface tension
and kinetic undercooling \eqref{stefan} on some possibly small but nontrivial time interval $J=[0,\tau]$.
\end{theorem}
\noindent
{\bf Proof of Theorems \ref{wellposed1} and \ref{wellposed2}:}\\
(i) {\em  Direct mapping method: Hanzawa transformation}. \\
As explained in the previous section, we employ a Hanzawa transformation and
study the resulting problem \eqref{transformed} on the domain $\Omega$ with fixed interface $\Sigma$.

In case $\gamma\equiv 0$, for the $L_p$-theory,
the solution of the transformed problem will belong to the class
$$ v\in H^1_p(J;L_p(\Omega))\cap L_p(J;H^2_p(\Omega\setminus\Sigma))
\hookrightarrow C(J;W^{2-2/p}_p(\Omega\setminus\Sigma)),$$
$$\rho\in W^{1-1/2p}_p(J;H^2_p(\Sigma))\cap L_p(J;W^{4-1/p}_p(\Sigma))
\hookrightarrow C(J;W^{4-3/p}_p(\Sigma)),$$
\begin{equation} \label{specialembedding}
\partial_t\rho \in
W^{1/2-1/2p}_p(J;L_p(\Sigma))\cap
L_p(J;W^{2-2/p}_p(\Sigma))\hookrightarrow C(J;W^{2-6/p}_p(\Sigma)),
\end{equation}
see \cite{EPS03} for a proof of the last two embeddings in the case
$\Sigma=\R^n$.

If $\gamma>0$ we have moreover
$$\rho\in W^{2-1/2p}_p(J;L_p(\Sigma))\cap L_p(J;W^{4-1/p}_p(\Sigma))
\hookrightarrow C^1(J;W^{2-3/p}_p(\Sigma)).$$
Note that in both cases, $v\in C(J\times\bar{\Omega})$, $v|_{\Omega_j}\in C(J;C^1(\bar{\Omega}_j))$, $j=1,2$.
Moreover, $\rho\in C(J;C^3(\Sigma))$ and
\begin{equation*}
\begin{split}
\partial_t\rho\in C(J;C(\Sigma))\text{ in case } \gamma=0,
\quad \partial_t\rho\in C(J;C^1(\Sigma))\text{ in case } \gamma>0.
\end{split}
\end{equation*}
We set
\begin{align*}
&\EE_1(J):=\{v\in H^1_p(J;L_p(\Omega))\cap
L_p(J;H^2_p(\Omega\setminus\Sigma):
[\![v]\!]=0,\ \partial_{\nu_\Omega}v=0 \},\\
&\EE_2(J):= W^{3/2-1/2p}_p(J;L_p(\Sigma))\cap
W^{1-1/2p}_p(J;H^2_p(\Sigma))\cap L_p(J;W^{4-1/p}_p(\Sigma)),
\ \gamma\equiv 0,\\
&\EE_2(J):= W^{2-1/2p}_p(J;L_p(\Sigma))\cap L_p(J;W^{4-1/p}_p(\Sigma)),\ \gamma>0,\\
&\EE(J)\;\,:= \EE_1(J)\times \EE_2(J),
\end{align*}
i.e.\ $\EE(J)$ denotes the solution space.
Similarly, we define
\begin{align*}
&\FF_1(J):=L_p(J;L_p(\Omega)),\\
&\FF_2(J) :=W^{1-1/2p}_p(J;L_p(\Sigma))\cap L_p(J;W^{2-1/p}_p(\Sigma)),\\
&\FF_3(J):=W^{1/2-1/2p}_p(J;L_p(\Sigma))\cap L_p(J;W^{1-1/p}_p(\Sigma)),\\
&\FF(J)\;\,:= \FF_1(J)\times\FF_2(J)\times \FF_3(J),
\end{align*}
i.e.\ $\FF(J)$ means the space of data. A left subscript zero means vanishing time trace at $t=0$, whenever it exists. So for example
$${_0}\EE_2(J)=\{\rho\in \EE_2(J):\; \rho(0)=\partial_t\rho(0)=0\}$$
whenever $p>3$. \noindent
Employing the calculations in Section 2 and splitting into the principal linear part and a nonlinear part,
we arrive at the following formulation of problem \eqref{transformed}.
\begin{equation}
\label{transformed-decomposed}
\left\{\begin{aligned}
\kappa_0(x)\partial_t v -d_0(x)\Delta v &=F(v,\rho)
 &&\text{in}&&\Omega\setminus\Sigma\\
\partial_{\nu_\Omega} v &=0
 &&\text{on}&&\partial \Omega\\
[\![v]\!]&=0, &&\text{on}&&\Sigma\\
l_1(t,x)v + \sigma_0\Delta_\Sigma \rho- \gamma_1(t,x)\partial_t\rho &=G(v,\rho)
 &&\text{on}&&\Sigma\\
\mbox{}l_0(x)\partial_t\rho -[\![d_0(x)\partial_\nu v]\!]&=H(v,\rho)
 &&\text{on}&&\Sigma\\
v(0)=v_0,\ \rho(0)&=\rho_0.&&
\end{aligned}\right.
\end{equation}
Here
\begin{equation*}
\begin{split}
&\kappa_0(x)=\kappa(v_0(x)),\quad d_0(x)=d(v_0(x)),\quad l_0(x)=l(v_0(x)),
\quad\sigma_0=\frac{\sigma}{n-1},\\
&l_1(t,\cdot)= [\![\psi^\prime(e^{\Delta_\Sigma t}v_{0\Sigma})]\!],
\quad\gamma_1(t,\cdot)=\gamma(e^{\Delta_\Sigma t}v_{0\Sigma}),
\end{split}
\end{equation*}
where $v_{0\Sigma}$ means the restriction of $v_0$ to $\Sigma$. Note
that $\kappa_0,d_0\in W^{2-2/p}_p(\Omega\setminus\Sigma)$, hence
these functions are in $C^1(\bar{\Omega}_j)$, $j=1,2$. Recall that
$d$ and $\kappa$ may be different in different phases. Further we
have $l_0\in W^{2-3/p}_p(\Sigma)$ which implies $l_0\in
C^1(\Sigma)$. This is good enough for the space $\FF_3(J)$, as
$C^1$-functions are pointwise multipliers for $\FF_3(J)$, but it is
not good enough for $\FF_2(J)$. For this reason, we need to define
the extension $v_b:=e^{\Delta_\Sigma t}v_{0\Sigma}$. This function
as well as $l_1$ and $\gamma_1$ belong to $\FF_2(J)$, hence are
pointwise multipliers for this space, as $\FF_2(J)$ and $\FF_3(J)$
are Banach algebras w.r.t.\ pointwise multiplication, as $p>n+2$.

The nonlinearities $F$, $G$, and $H$ are defined as follows.
\begin{align}
F(v,\rho)&= (\kappa_0-\kappa(v))\partial_t v+(d(v)-d_0)\Delta v - d(v)M_2(\rho):\nabla^2 v\nonumber\\
&\quad +d^\prime(v)|(I-M_1(\rho))\nabla v|^2 - d(v)(M_3(\rho)|\nabla v)+\kappa(v)\cR(\rho)v,\nonumber\\
G(v,\rho)&=-([\![\psi(v)]\!]+\sigma \cH(\rho)) +l_1 v +{\sigma_0} \Delta_\Sigma \rho+(\gamma(v)\beta(\rho)-\gamma_1)\partial_t\rho,
\label{FGHdef}\\
H(v,\rho)&=[\![(d(v)-d_0)\partial_\nu v]\!]+ (l_0-l(v))\partial_t\rho
-([\![d(v)\nabla v]\!]|M_4(\rho)\nabla_\Sigma\rho)\nonumber\\
&\quad +\gamma(v)\beta(\rho)(\partial_t\rho)^2.\nonumber
\end{align}
Here we have set
\begin{align*}M_2(\rho)&=M_1(\rho)+M_1^{\sf T}(\rho)-M_1(\rho)M_1^{\sf T}(\rho),\\
M_3(\rho)&=(I-M_1(\rho)):\nabla M_1(\rho), \\
M_4(\rho)&=(I-M_1(\rho))^{\sf T}M_0(\rho).
\end{align*}
(ii) {\em Maximal regularity of the principal linearized problem}.\\
First we consider
the linear problem defined by the left hand side of \eqref{transformed-decomposed}.
\begin{equation}
\label{prlin}
\left\{\begin{aligned}
\kappa_0(x)\partial_t v -d_0(x)\Delta v &=f
 &&\text{in}&&\Omega\setminus\Sigma\\
\partial_{\nu_\Omega} v &=0
 &&\text{on}&&\partial \Omega\\
[\![v]\!]&=0 &&\text{on}&&\Sigma\\
l_1(t,x)v + \sigma_0\Delta_\Sigma \rho- \gamma_1(t,x)\partial_t\rho &=g
 &&\text{on}&&\Sigma\\
\mbox{}l_0(x)\partial_t\rho -[\![d_0(x)\partial_\nu v]\!]&=h
 &&\text{on}&&\Sigma\\
v(0)=v_0,\ \rho(0)&=\rho_0.&&
\end{aligned}\right.
\end{equation}

\noindent
This inhomogeneous problem can be solved with maximal regularity; see
Escher, Pr\"uss, and Simonett \cite{EPS03} for the constant coefficient
half-space case with $\gamma\equiv 0$, and Denk, Pr\"uss, and Zacher \cite{DPZ08} for the general
one-phase case.
\begin{theorem} \label{linmaxreggamma=0} {\rm ($\gamma\equiv 0$).}
\label{linpeq0} Let $p>n+2$, $\sigma>0$, $\gamma\equiv0$. Suppose
$\kappa_0\in C(\bar{\Omega}_j)$ and $d_0\in C^1(\bar{\Omega}_j)$,
$j=1,2$, $\kappa_0,d_0>0$ on $\bar{\Omega}$, $l_0\in W_p^{2-6/p}(\Sigma)$, and let
\begin{equation*}
l_1\in W^{1-1/2p}_p(J;L_p(\Sigma))\cap L_p(J;W^{2-1/p}_p(\Sigma))
\end{equation*}
  such that $l_0l_1>0$ on $J\times\Sigma$, where $J=[0,t_0]$ is a finite time interval. Then there is a unique solution $z:=(v,\rho)\in\EE(J)$
of (\ref{prlin}) if and only if the data $(f,g,h)$ and $z_0:=(v_0,\rho_0)$ satisfy
\begin{equation*}
(f,g,h)\in \FF(J),\quad z_0\in [W^{2-2/p}_p(\Omega\setminus\Sigma)\cap C(\bar{\Omega})]\times W^{4-3/p}_p(\Sigma),\end{equation*}
and the compatibility conditions
 $$\partial_{\nu_\Omega}v_0=0,\quad
 l_1(0)v_0+\sigma_0 \Delta_\Sigma \rho_0=g(0), \quad h(0)+[\![d_0\partial_\nu v_0]\!]
\in W^{2-6/p}_p(\Sigma).$$
The solution map $[(f,g,h,z_0)\mapsto z=(v,\rho)]$ is continuous between the corresponding spaces.
\end{theorem}


\begin{proof} In the one-phase case this result is proved in \cite[Example 3.4]{DPZ08}.
Therefore, we only indicate the necessary modifications for the two-phase case. The
localization procedure can be carried out in the same way as in the one-phase case \cite{DPZ08},
hence we only need to consider the following model problem with constant coefficients where the interface is flat:
\begin{equation*}
\left\{\begin{aligned}
\kappa_0\partial_t v -d\Delta v &=f
 &&\text{in}&&\dot \R^{n}\\
[\![v]\!]&=0
 &&\text{on}&&\R^{n-1}\\
l_1v + \sigma_0\Delta \rho&=g
 &&\text{on}&&\R^{n-1}\\
l_0\partial_t\rho -[\![d\partial_\nu v]\!]&=h
 &&\text{on}&&\R^{n-1}\\
v(0)=v_0,\ \rho(0)&=\rho_0.&&
\end{aligned}\right.
\end{equation*}
Here $\dot\R^n=\R^{n-1}\times\big(\R\setminus\{0\}\big)$, and
$\R^{n-1}$ is identified with $\R^{n-1}\times\{0\}$. Reflecting the
lower half-plane to the upper, this becomes a problem of the form
studied in \cite{DPZ08}. As in Example 3.4 of that paper it is not
difficult to verify the necessary Lopatinskii-Shapiro conditions.
Then Theorems 2.1 and 2.2 of \cite{DPZ08} can be applied, proving
the assertion for the model problem.
\end{proof}
\begin{remark}  One might wonder where the somewhat unexpected compatibility condition
$h(0)+[\![d_0\partial_\nu v_0]\!] \in W^{2-6/p}_p(\Sigma)$ in the
case $\gamma=0$ comes from. To illuminate this, note that
$$(h(0)+[\![d_0\partial_\nu v_0]\!])/l_0=\partial_t\rho(0)$$
is the trace of $\partial_t\rho$ at time $t=0$. But by the embedding
(\ref{specialembedding}) this implies that
$(h(0)+[\![d_0\partial_\nu v_0]\!])/l_0\in W^{2-6/p}_p(\Sigma)$,
which in turn enforces $h(0)+[\![d_0\partial_{\nu}u_0]\!]\in
W^{2-6/p}_p(\Sigma)$.
\end{remark}
The main result for problem (\ref{prlin}) for $\gamma>0$ is the
following theorem.
\medskip
\begin{theorem}{\rm ($\gamma >0$).}
\label{linpneq0} Let $p>n+2$, $\sigma>0$. Suppose $\kappa_0\in
C(\bar{\Omega}_j)$ and $d_0\in C^1(\bar{\Omega}_j)$, $j=1,2$,
$\kappa_0,d_0>0$ on $\bar{\Omega}$, $l_0\in C^1(\Sigma)$, and let
\begin{equation*}
\gamma_1,l_1\in W^{1-1/2p}_p(J;L_p(\Sigma))\cap L_p(J;W^{2-1/p}_p(\Sigma)),
\end{equation*}
such that $\gamma_1>0$ on $J\times\Sigma$,
where $J=[0,t_0]$ is a finite time interval. Then there is a unique solution $z:=(v,\rho)\in\EE(J)$
of (\ref{prlin}) if and only if the data $(f,g,h)$ and $z_0:=(v_0,\rho_0)$ satisfy
\begin{equation*}
(f,g,h)\in \FF(J),\quad z_0\in [W^{2-2/p}_p(\Omega\setminus\Sigma)
\cap C(\bar{\Omega})]\times W^{4-3/p}_p(\Sigma),
\end{equation*}
and the compatibility conditions
$$\partial_{\nu_\Omega}v_0=0,\quad (l_0l_1(0){v_0}_{|_{\Sigma}}+l_0\sigma_0 \Delta_\Sigma \rho_0-\gamma_1(0)[\![d\partial_\nu v_0]\!]
= \gamma_1(0)h(0)+l_0g(0).$$
The solution map $[(f,g,h,z_0)\mapsto z=(v,\rho)]$ is continuous between the corresponding spaces.
\end{theorem}
\begin{proof} The proof of this result is much simpler than for the case $\gamma=0$.
We could follow the strategy in the proof of Theorem
\ref{linmaxreggamma=0}, employing the methods in \cite{DPZ08} once
more. However, here we want to give a more direct argument that uses
the fact that the term $l_0\partial_t\rho$ is of lower order in case
$\gamma_1>0$. For this purpose, suppose $v_\Sigma:=v_{|_{\Sigma}}$
is known. Consider the problem
$$\gamma_1\partial_t\rho- \sigma_0 \Delta_\Sigma\rho =l_1v_\Sigma-g,\quad t\in J,\quad \rho(0)=\rho_0.$$
Since the Laplace-Beltrami operator is strongly elliptic, we can solve this problem
with maximal regularity to obtain $\rho$ in the proper regularity class. Then we solve
the transmission problem
\begin{equation*}
\left\{\begin{aligned}
\kappa_0\partial_tv-d_0\Delta v &=f &&\text{in}&&\Omega\setminus\Sigma\\
\partial_{\nu_\Omega}v &=0 &&\text{on}&&\partial\Omega\\
 [\![v]\!]&=0  &&\text{on}&&\Sigma\\
 -[\![d_0\partial_\nu v]\!]&=h-l_0\partial_t\rho &&\text{on}&&\Sigma\\
 v(0)&=v_0. \\
\end{aligned}\right.
\end{equation*}
Finally, we take the trace of $v$ to obtain an equation for $v_\Sigma$ of the form
$$v_\Sigma= Tv_\Sigma + w,$$
where $w$ is determined by the data alone, and $T$ is a compact operator from
$\FF_2(J) $ into itself. 
Here compactness follows from the compact embedding $\FF_2(J)\hookrightarrow \FF_3(J)$,
i.e.\ from the regularity of $\partial_t\rho$ which is higher than needed to solve the
transmission problem. Thus  $I-T$ is a Fredholm operator
with index zero, hence invertible  since it is injective by causality. This proves the sufficiency
of the conditions on the data. Necessity is a consequence of trace theory.
\end{proof}
\begin{remark} It is interesting to take a look at the boundary symbol of the linear problem. It is of the form
$$s(\lambda,\xi)=\lambda l^2+(\lambda\gamma +\sigma_0|\xi|^2)
[\sqrt{\lambda\kappa_1+d_1|\xi|^2}+\sqrt{\lambda\kappa_2+d_2|\xi|^2}].$$
Here $\lambda\in\C_+$ denotes the covariable of time $t$, and $\xi\in\R^{n-1}$ that of the
tangential space variable $x^\prime\in\R^{n-1}$.
This symbol is invertible for large $\lambda$, provided $\gamma>0$ or $l\neq0$.
 Note that in case $\gamma>0$ this is a parabolic
symbol of order $3/2$ in time $t$ and of order $3$ in the space variables $x$.
The term $\lambda l^2$ is of lower order, thus  $l$ does not affect well-posedness.
On the other hand, for $l=0$ and $\gamma=0$ the boundary symbol is ill-posed, since it admits the
zeros $(\lambda,0)$ with arbitrarily large ${\rm Re}\, \lambda$. If $\gamma=0$ and $l\neq0$,
then it is well-posed. Note that in this case we have order $1$ in time, $3$ in space, but also
the mixed regularity $1/2$ in time and $2$ in space.
\end{remark}

\bigskip

\noindent
(iii) {\em Reduction to zero initial values.}\\
It is convenient to reduce the problem to zero initial data and inhomogeneities with vanishing time trace. This can be achieved as follows.
We solve the linear problem \eqref{prlin} with initial data $v_0$, $\rho_0$ and inhomogenities
\begin{equation*}
f=0, \quad
g(t)=e^{\Delta_\Sigma t}G(v_0,\rho_0),
\quad
h(t)=e^{\Delta_\Sigma t}\rho_1 \text{ with } \rho_1=H(v_0,\rho_0).
\end{equation*}
Since the Laplace-Beltrami operator $\Delta_\Sigma$ has maximal
$L_p$-regularity, the fact that $G(v_0,\rho_0)\in
W^{2-3/p}_p(\Sigma)$ implies $g\in \FF_2(J)$. Similarly, $h\in\FF_3(J)$
since $\rho_1\in W^{1-3/p}_p(\Sigma)$. The compatibility conditions
yield $[\![d_0\partial_\nu v_0]\!]+\rho_1\in W^{2-6/p}_p(\Sigma)$.
Therefore, the linear problem has a unique solution
$z_*:=(v_*,\rho_*)$ with maximal regularity $z_*\in \EE(J)$. Then we
set $\bar{v}=v-v_*$, $\bar{\rho}=\rho-\rho_*$, and obtain the
following problem for  $\bar{z}=(\bar{v},\bar{\rho})$.
\begin{equation}
\label{trans-decomp-red}
\left\{\begin{aligned}
\kappa_0(x)\partial_t \bar{v} -d_0(x)\Delta \bar{v}&=F(\bar{v}+v_*,\bar{\rho}+\rho_*)
&&\text{in}&& \Omega\setminus\Sigma \\
\partial_{\nu_\Omega} \bar{v}&=0 &&\text{on}&& \partial \Omega\\
\mbox{}[\![\bar{v}]\!]&=0 &&\text{on}&& \Sigma\\
l_1(t,x)\bar{v} + \sigma_0 \Delta_\Sigma \bar{\rho}- \gamma_1(t,x)\partial_t\bar{\rho}
&=\bar{G}(\bar{v},\bar{\rho};v_*,\rho_*) &&\text{on}&& \Sigma\\
l_0(x)\partial_t\bar{\rho} -[\![d_0(x)\partial_\nu \bar{v}]\!]
&=\bar H(\bar{v},\bar{\rho}; v_*,\rho_*) &&\text{on}&& \Sigma\\
\bar{v}(0)=0, \ \bar{\rho}(0)&=0.&&
\end{aligned}\right.
\end{equation}
Here we have set
\begin{equation*}
\begin{split}
&\bar{G}(\bar{v},\bar{\rho};v_*,\rho_*)=G(\bar{v}+v_*,\bar{\rho}+\rho_*)-e^{\Delta_\Sigma t}G(v_0,\rho_0),\\
&\bar H(\bar{v},\bar{\rho};v_*,\rho_*)
= H(\bar{v}+v_*,\bar{\rho}+\rho_*)-e^{\Delta_\Sigma t}H(v_0,\rho_0).\\
\end{split}
\end{equation*}
Note that $\bar{G}(0,0;v_0,\rho_0)=\bar{H}(0,0;v_0,\rho_0)=0$
by construction, which ensures time trace zero at $t=0$.

\bigskip

\noindent
(iv) {\em Solution of the nonlinear problem} \\
We first concentrate on the case $\gamma\equiv0$, and rewrite problem \eqref{trans-decomp-red} in abstract form as
$$ \LL \bar{z} = \NN(\bar{z},z_*),$$
where $\LL:{_0\EE}(0,t_0)\to {_0\FF}(0,t_0)$, defined by
\begin{equation*}
\LL\bar{z} =
\big( \kappa_0\partial_t\bar{v}-d_0\Delta \bar{v}, l_1 \bar{v} +\sigma_0 \Delta_\Sigma \bar{\rho}, l_0\partial_t\bar{\rho}-[\![d_0\partial_\nu\bar{v}]\!]\big),
\end{equation*}
is an isomorphism by Theorem \ref{linpeq0}. The nonlinearity
$$
\NN:{_0\EE}(0,t_0)\times\EE(0,t_0)\to {_0\FF}(0,t_0),
$$
given by the right hand side of\eqref{trans-decomp-red},  is of
class $C^1$, since the coefficient functions satisfy $\kappa\in
C^1$, $d,l\in C^2$, $\psi\in C^3$, and by virtue of the embeddings
\begin{equation*}
\EE_1(J)\hookrightarrow C(J\times\bar{\Omega})\cap C(J;C^1(\bar{\Omega}_j)),\quad \EE_2(J)\hookrightarrow C(J;C^3(\Sigma))\cap C^1(J;C(\Sigma)).
\end{equation*}
Observe that the constants in these embeddings blow up as $t_0\to 0$, however, they are uniform in $t_0$ if one considers the space ${_0\EE}(J)$!


We want to apply the contraction mapping principle. For this purpose we consider a closed ball
$\BB_R(0)\subset{_0\EE}(0,\tau)$, where the radius $R>0$ and the final time $\tau\in(0,t_0]$ are at our disposal.
We rewrite the abstract equation $\LL\bar{z}=\NN(\bar{z},z_*)$ as the fixed point equation
\begin{equation*}
\bar{z}=\LL^{-1}\NN(\bar{z},z_*)=:\TT(\bar{z}), \quad \bar{z}\in \BB_R(0).
\end{equation*}
Since we are working in an $L_p$-setting, by choosing $\tau=\tau(R)$ small enough we can assure that
\begin{equation*}
\Ver\TT(0)\Ver_{\EE(0,\tau)}=\Ver\LL^{-1}\NN(0,z_*)\Ver_{\EE(0,\tau)}\leq R/2.
\end{equation*}
On the other hand, we have
\begin{align*}
\Ver\TT(z_1)-\TT(z_2)\Ver_{\EE(0,\tau)}
&\leq \Ver\LL^{-1}\Ver_{\cB({_0\FF}(0,\tau),{_0\EE}(0,\tau))}\times\\
 &\times\sup_{\Ver\bar{z}\Ver_{{_0\EE}(0,\tau)}\le R}
 \Ver\NN^\prime(\bar{z},z_*)\Ver_{\cB({_0\EE}(0,\tau),{_0\FF}(0,\tau))}\Ver z_1-z_2\Ver_{\EE(0,\tau)},
 \end{align*}
hence $\TT(\BB_R(0))\subset\BB_R(0)$ and $\TT$ is a strict contraction, provided we have
\begin{equation*}
\Ver\LL^{-1}\Ver_{\cB({_0\FF}(0,\tau),{_0\EE}(0,\tau))} \sup_{\Ver\bar{z}\Ver_{{_0\EE}(0,\tau)}
\le R} \Ver\NN^\prime(\bar{z},z_*)\Ver_{\cB({_0\EE}(0,\tau),{_0\FF}(0,\tau))}\leq 1/2.
\end{equation*}
For this we observe that
$$
\Ver\LL^{-1}\Ver_{\cB({_0\FF}(0,\tau),{_0\EE}(0,\tau))}
\leq\Ver\LL^{-1}\Ver_{\cB({_0\FF}(0,t_0),{_0\EE}(0,t_0))}=:C_M<\infty$$
is uniform in $\tau\in(0,t_0)$, since we have vanishing time traces
at $t=0$. So it remains to estimate the Frech\'et-derivative of
$\NN$  on the ball $\BB_R(0)\subset {_0\EE}(0,\tau)$. This is the
content of the next proposition, which also covers the case
$\gamma>0$.
\begin{proposition}
\label{nonlinearities} Let $p>n+2$, $\sigma\in\R$, and suppose
$\psi,\gamma\in C^3(0,\infty)$ and $d\in C^2(0,\infty)$.
Then $\NN:{_0\EE}(0,t_0)\times\EE(0,t_0)\to {_0\FF}(0,t_0)$ is
continuously Fr\'echet-differentiable. There is $\eta>0$ such that
for a given $z_*\in\EE(0,t_0)$ with $|\rho_0|_{C^2(\Sigma)}\leq
\eta$, there are continuous functions $\alpha(R)>0$ and
$\beta(\tau)>0$ with $\alpha(0)=\beta(0)=0$,  such that
\begin{equation*}
 \Ver\NN^\prime(\bar{z}+z_*)\Ver_{\cB({_0\EE}(0,\tau),{_0\FF}(0,\tau))}
 \leq \alpha(R)+\beta(\tau), \quad \bar{z}\in \BB_R\subset {_0\EE}(0,\tau) .
\end{equation*}
\end{proposition}
\begin{proof} We may proceed similarly as in \cite[Section 7]{EPS03}, where the interface is a graph over $\R^{n-1}$.
The additional terms which arise by considering a general geometry
are either of lower order or of the form
$\tilde{M}(\bar{v},\bar{\rho})\nabla_\Sigma \bar{\rho}$ where
$\tilde{M}(\bar{v},\bar{\rho})$ is of highest order (see
(\ref{FGHdef})), but can be controlled by ensuring that
$\nabla_\Sigma \bar{\rho}$ is sufficiently small. The additional
terms due to the presence of $\gamma$ are of highest order, but
small.
\end{proof}

So choosing first $R>0$ and then $\tau>0$ small enough, $\TT$ will be a self-map and a strict contraction on $\BB_R(0)$.
Concluding, the contraction mapping principle yields a unique fixed point $\bar{z}=\bar{z}(z_*)\in \BB_R(0)\subset{_0\EE}(0,\tau)$,
hence $z=z_*+\bar{z}(z_*)$ is the unique solution of \eqref{transformed-decomposed}, i.e.\ of \eqref{transformed}.

The proof in case $\gamma>0$ is similar, employing now Theorem \ref{linpneq0}.

\begin{remark}
 The assumption $p>n+2$ simplifies many arguments since $\FF_2(J)$ as well as $\FF_3(J)$ are Banach algebras
 and $\nabla v\in BC(J\times \Omega)$. If we merely assume $p>(n+2)/2$ then $\FF_2(J)$ is still a Banach algebra, but $\FF_3(J)$ is not, and $\nabla v$ may not be bounded anymore. This leads to much more involved estimates for the nonlinearities.
\end{remark}
\noindent{\bf Local semiflows.} We denote by $\cMH^2(\Omega)$ the
closed $C^2$-hypersurfaces contained in $\Omega$. It can be shown
that $\cMH^2(\Omega)$ is a $C^2$-manifold: the charts are the
parameterizations over a given hypersurface $\Sigma$ according to
Section 2, and the tangent space consists of the normal vector
fields on $\Sigma$. We define a metric on $\cMH^2(\Omega)$ by means
of
$$d_{\cMH^2}(\Sigma_1,\Sigma_2):= d_H(\cN^2\Sigma_1,\cN^2\Sigma_2),$$
where $d_H$ denotes the Hausdorff metric on the compact subsets of $\R^n$ introduced in Section 2.
This way $\cMH^2(\Omega)$ becomes a Banach manifold of class $C^2$.

Let $d_\Sigma(x)$ denote the signed distance for $\Sigma$ as in Section 2. We may then define the {\em level function} $\varphi_\Sigma$ by means of
$$\varphi_\Sigma(x) = \phi(d_\Sigma(x)),\quad x\in\R^n,$$
where
$$\phi(s)=(1-\chi(s/a))\,{\rm sgn}\, s+ s \chi(s/a),\quad s\in \R.$$
Then it is easy to see that $\Sigma=\varphi_\Sigma^{-1}(0)$, and
$\nabla \varphi_\Sigma(x)=\nu_\Sigma(x)$, for $x\in \Sigma$.
Moreover, $0$ is an eigenvalue of $\nabla^2\varphi_\Sigma(x)$, and
the remaining eigenvalues of $\nabla^2\varphi_\Sigma(x)$ are the
principal curvatures of $\Sigma$ at $x\in\Sigma$.

If we consider the subset $\cMH^2(\Omega,r)$ of $\cMH^2(\Omega)$ which consists of all closed
hypersurfaces $\Gamma\in \cMH^2(\Omega)$ such that $\Gamma\subset \Omega$ satisfies a
(interior and exterior) ball condition with fixed radius $r>0$, then the map
\begin{equation}
\Upsilon:\cMH^2(\Omega,r)\to C^2(\bar{\Omega}),\quad
\Upsilon(\Gamma):=\varphi_\Gamma,
\end{equation}
is an isomorphism of the metric space $\cMH^2(\Omega,r)$ onto
$\Upsilon(\cMH^2(\Omega,r))\subset C^2(\bar{\Omega})$.
\medskip\\
Let $s-(n-1)/p>2$. Then we define
\begin{equation}
\label{definition-W-r}
W^s_p(\Omega,r):=\{\Gamma\in\cMH^2(\Omega,r): \varphi_\Gamma\in W^s_p(\Omega)\}.
\end{equation}
In this case the local charts for $\Gamma$ can be chosen of class
$W^s_p$ as well. A subset $A\subset W^s_p(\Omega,r)$ is said to be
(relatively) compact, if $\Upsilon(A)\subset W^s_p(\Omega)$ is
(relatively) compact.

As an ambient space for the
state manifold $\cSM_\gamma$ of the Stefan problem with surface tension we consider
the product space $C(\bar{G})\times \cMH^2$, due to continuity of temperature
and curvature.

We define the state manifolds $\cSM_\gamma$, $\gamma\geq0$, for the Stefan problem \eqref{stefan} as follows.
For $\gamma=0$ we set
\begin{eqnarray}
\label{phasemanif0}
\cSM_0:=&&\hspace{-0.5cm}\{(u,\Gamma)\in C(\bar{\Omega})\times \cMH^2:
 u\in W^{2-2/p}_p(\Omega\setminus\Gamma),\, \Gamma\in W^{4-3/p}_p,\\
&&u>0 \mbox{ in } \bar{\Omega},\; [\![\psi(u)]\!]+\sigma \cH=0,\;l(u)\neq0\mbox{ on } \Gamma,\;
[\![d\partial_\nu u]\!] \in W^{2-6/p}_p(\Gamma)\},\nonumber
\end{eqnarray}
and for $\gamma>0$
\begin{eqnarray}\label{phasemanifg}
\cSM_\gamma:=&&\hspace{-0.5cm}\{(u,\Gamma)\in C(\bar{\Omega})\times \cMH^2:
 u\in W^{2-2/p}_p(\Omega\setminus\Gamma),\, \Gamma\in W^{4-3/p}_p,\\
&&
u>0 \mbox{ in } \bar{\Omega},\;(l(u)-[\![\psi(u)]\!]-\sigma \cH)([\![\psi(u)]\!]+\sigma \cH) =\gamma(u)[\![d\partial_\nu u]\!]
\mbox{ on } \Gamma\}.\nonumber
\end{eqnarray}
Charts for these manifolds are obtained by the charts induced by $\cMH^2(\Omega)$,
followed by a Hanzawa transformation.

Applying Theorem \ref{Thm3.1} or Theorem \ref{Thm3.2}, respectively,
and re-parametrizing the interface repeatedly, we see that
(\ref{stefan}) yields a local semiflow on $\cSM_\gamma$.
\begin{theorem}
\label{semiflow}
Let $p>n+2$, $\sigma>0$ and $\gamma\geq0$.
Then problem (\ref{stefan}) generates a local semiflow
on the state manifold $\cSM_\gamma$. Each solution $(u,\Gamma)$ exists on a maximal time
interval $[0,t_*)$, where $t_*=t_*(u_0,\Gamma_0)$.
\end{theorem}
\medskip
\noindent{\bf Time weights.} For later use we need an extension of
the local existence results to spaces with time weights. 
In particular, we need this extension for a compactness argument in the proof
of Theorem~\ref{Qual}.
Given a UMD-Banach space $Y$ and $\mu\in(1/p,1]$, we define
for $J=(0,t_0)$
$$K^s_{p,\mu}(J;Y):=\{u\in L_{p,loc}(J;Y): \; t^{1-\mu}u\in K^s_p(J;Y)\},$$
where $s\geq0$ and $K\in\{H,W\}$.
It has been shown in \cite{PrSi04} that the operator $d/dt$ in $L_{p,\mu}(J;Y)$ with domain
$$D(d/dt)={_0H}^1_{p,\mu}(J;Y)=\{u\in H^1_{p,\mu}(J;Y):\; u(0)=0\}$$
is sectorial and admits an $H^\infty$-calculus with angle $\pi/2$. However, it does not generate a $C_0$-semigroup, unless $\mu=1$. This is the main tool for extending the results for the linear problem,
i.e.\ Theorems \ref{linpeq0} and \ref{linpneq0},
to the time weighted setting, where the solution space $\EE(J)$ is replaced by
$$ \EE_\mu(J)=\EE_{\mu,1}(J)\times \EE_{\mu,2}(J),$$
with
\begin{equation*}
\begin{split}
&\EE_{\mu,1}(J)\!=\{v\in H^1_{p,\mu}(J;L_p(\Omega))\cap
L_{p,\mu}(J;H^2_p(\Omega\setminus\Sigma):
[\![v]\!]=0,\ \partial_{\nu_\Omega}v=0\},\\
&\EE_{\mu,2}(J)\!:= W^{3/2-1/2p}_{p,\mu}(J;L_p(\Sigma))\!\cap\!
W^{1-1/2p}_{p,\mu}(J;H^2_p(\Sigma))
\!\cap\! L_{p,\mu}(J;W^{4-1/p}_p(\Sigma)),\!\gamma\equiv 0 , \\
&\EE_{\mu,2}(J)\!:= W^{2-1/2p}_{p,\mu}(J;L_p(\Sigma))\cap L_{p,\mu}(J;W^{4-1/p}_p(\Sigma)),
\ \gamma>0.
\end{split}
\end{equation*}
In a similar way, the space of data is defined by
\begin{align*}
&\FF_{\mu,1}(J):=L_{p,\mu}(J; L_p(\Omega)),\\
&\FF_{\mu,2}(J) :=W^{1-1/2p}_{p,\mu}(J;L_p(\Sigma))\cap L_{p,\mu}(J;W^{2-1/p}_p(\Sigma)),\\
&\FF_{\mu,3}(J):=W^{1/2-1/2p}_{p,\mu}(J;L_p(\Sigma))\cap L_{p,\mu}(J;W^{1-1/p}_{p}(\Sigma)),\\
&\FF_\mu(J)\;\;:= \FF_{\mu,1}(J)\times\FF_{\mu,2}(J)\times \FF_{\mu,3}(J).
\end{align*}
The trace spaces for $v$ and $\rho$ for $p>3$ are then given by
\begin{align}\label{tracesp-mu}
v_0\in W^{2\mu-2/p}_p(\Omega\setminus\Sigma),\quad \rho_0\in W^{2+2\mu-3/p}_p(\Sigma), \quad \rho_1\in W^{4\mu-2-6/p}_p(\Sigma),
\end{align}
where for the last trace - which is of relevance only in case
$\gamma\equiv0$ - we need in addition $\mu>1/2+3/2p$. Note that the embeddings
\begin{equation*}
\EE_{\mu,1}(J)\hookrightarrow C(J\times \bar{\Omega})\cap C(J;C^1(\bar{\Omega}_j)),\quad \EE_{\mu,2}(J)\hookrightarrow C(J;C^3(\Sigma))
\end{equation*}
require
$\mu>1/2+(n+2)/2p$, which is feasible since $p>n+2$ by assumption.
This restriction is needed for the estimation of the nonlinearities, i.e.\
Proposition~\ref{nonlinearities} remains valid for $\mu\in(1/2+(n+2)/2p,1)$.

The assertions for the linear problem remain valid for such $\mu$, replacing $\EE(J)$ by $\EE_\mu(J)$, $\FF(J)$ by $\FF_\mu(J)$, for initial data
subject to \eqref{tracesp-mu}. This relies on the fact mentioned above that $d/dt$ admits a bounded $H^\infty$-calculus
with angle $\pi/2$ in the spaces $L_{p,\mu}(J;Y)$.
Therefore the main results in Denk, Pr\"uss and Zacher \cite{DPZ08} remain valid for $\mu\in(1/p,1)$. This has recently been established in \cite{Mey10, MeyS11}.
As a consequence of these considerations we have the following result.
\goodbreak
\begin{corollary}
\label{wellposed3} Let $p>n+2$, $\mu\in (1/2+(n+2)/2p,1]$,
$\sigma>0$, and suppose that $\psi,\gamma\in C^3(0,\infty)$, $d\in
C^2(0,\infty)$ such that $\gamma\equiv 0$ or $\gamma(u)>0$, $u\in
(0,\infty)$, and
$$\kappa(u)=-u\psi^{\prime\prime}(u)>0,\quad d(u)>0,\quad u\in(0,\infty).$$
Assume the {\em regularity conditions}
$$u_0\in W^{2\mu-2/p}_p(\Omega\setminus\Gamma_0)\cap C(\bar{\Omega}),\quad u_0>0,\quad \Gamma_0\in W^{2+2\mu-3/p}_p,$$
and the {\em  compatibility conditions} $\displaystyle \partial_{\nu_\Omega}u_0=0$ and
\begin{itemize}
\item[(a)]
$\displaystyle
[\![\psi(u_0)]\!]+\sigma \cH(\Gamma_0)=0,\ [\![d(u_0)\partial_\nu u_0]\!]\in W^{4\mu-2-6/p}_p(\Gamma_0),$
as well as  the {\em well-posedness condition} $\ l(u_0)\neq0$ on $\Gamma_0$, in case $\gamma\equiv0$.
\vspace{2mm}
\item[(b)]
$\displaystyle
\big([\![\psi(u_0)]\!]+\sigma \cH(\Gamma_0)\big)\big(l(u_0)-[\![\psi(u_0)]\!]-\sigma \cH(\Gamma_0)\big)=\gamma(u_0)[\![d(u_0)\partial_\nu u_0]\!]$
in case $\gamma>0$,
\end{itemize}
Then the transformed problem \eqref{transformed} admits a unique
solution $z=(v,\rho)\in \EE_\mu(0,\tau)$ for some nontrivial time
interval $J=[0,\tau]$. The solution depends continuously on the
data. For each $\delta>0$ the solution belongs to
$\EE(\delta,\tau)$, i.e. regularizes instantly.
\end{corollary}
\section{Equilibria}
Suppose $(u_*,\Gamma_*)$ is an equilibrium for \eqref{stefan}.
Then $\partial_t u_* \equiv0$ as well as $V_*\equiv0$, and we obtain
\begin{equation}
\label{equilibria}
\left\{\begin{aligned}
{\rm div}(d(u_*)\nabla u_*)&=0 &&\text{in}&& \Omega\setminus\Gamma_*\\
 \partial_{\nu_\Omega} u_*&=0  &&\text{on}&& \partial \Omega\\
[\![u_\ast]\!] &=0 &&\text{on}&&  \Gamma_*\\
[\![\psi(u_*)]\!]+\sigma \cH(\Gamma_*)&=0  &&\text{on}&& \Gamma_*\\
[\![d(u_*)\partial_\nu u_* ]\!]&=0 &&\text{on}&& \Gamma_*.\\
\end{aligned}\right.
\end{equation}
This yields
$u_*=const$, hence $\cH(\Gamma_*)=-[\![\psi(u_*)]\!]/\sigma$ is constant as well.
If $\Gamma_*$ is connected (and $\Omega$ is bounded) this implies that
$\Gamma_*$ is a sphere $S_{R_*}(x_0)$ with radius $ R_*=\sigma/[\![\psi(u_*)]\!]$.
Thus there is an $(n+1)$-parameter family of equilibria
$$\cE:=\{(u_*,S_{R_*}(x_0)):\, u_*>0,\,0< R_*=\sigma/[\![\psi(u_*)]\!],\; \bar{B}_{R_*}(x_0)\subset \Omega\}.$$
Otherwise, $\Gamma_*$ is the union of finitely many, say $m$,
nonintersecting spheres of equal radius. It will be shown in the
proof of Theorem~\ref{lin-stability}(vii) that $\cE$  is a
$C^1$-manifold of dimension $(mn+1)$ in
$W^2_p(\Omega\setminus\Gamma_\ast)\times W^{4-1/p}_p(\Gamma_\ast)$.
\medskip
\subsection{Conservation of Energy}
As we have just seen, the equilibria of \eqref{stefan} are constant temperature, and the dispersed phase consists of finitely many non-intersecting balls with the same radius.
To determine $u$ and $R$, taking into account conservation of energy, we have to solve the system
\begin{equation}
\begin{split}
{\sf E}(u,R):= |\Omega_1| \epsilon_1(u) +|\Omega_2| \epsilon_2(u) + \frac{\sigma}{n-1}|\Sigma|&={\sf E}_0,\nonumber\\
[\![\psi(u)]\!]+\sigma \cH&=0.\nonumber
\end{split}
\end{equation}
In order not to overburden the notation, we use $(u,R)$ instead of
$(u_\ast,R_\ast)$. The constant ${\sf E}_0$ means the initial total
energy in the system. Since $\cH=-\sigma/R$ we may eliminate $R$ by
the second equation $R=\sigma/[\![\psi(u)]\!]$,
 and we are left with a single equation for the temperature $u$:
\begin{equation} \label{eq-rel}
\varphi(u):={\sf E}(u,R(u))= |\Omega|\epsilon_2(u)
-\frac{m\omega_n}{n} R^n(u)[\![\epsilon(u)]\!] + \frac{\sigma m
\omega_n}{n-1} R^{n-1}(u)=\varphi_0,
\end{equation}
with $\varphi_0={\sf E}_0$.
Note that only the temperature range $[\![\psi(u)]\!]>0$ is relevant
due to the requirement $R>0$, and with
$$R_m=\sup\{R>0:\, \Omega \mbox{ contains } m \mbox{ disjoint balls of radius } R\}$$
we must also have $R<R_m$, i.e. with  $h(u) = [\![\psi(u)]\!]$
$$h(u)>  \frac{\sigma}{ R_m}.$$
With $\epsilon(u)=\psi(u)-u \psi^\prime(u)$, i.e.\ $[\![\epsilon(u)]\!]=h(u)-uh^\prime(u)$,  we may rewrite $\varphi(u)$ as
$$ \varphi(u) =|\Omega|\epsilon_2(u) + c_n \Big( \frac{1}{h(u)^{n-1}} +(n-1)u\frac{h^\prime(u)}{h(u)^n}\Big),$$
where we have set $c_n = m\frac{\omega_n}{n(n-1)}\sigma^n.$

Next, with
$$ R^\prime(u) = - \frac{\sigma h^\prime(u)}{ h^2(u)} = - \frac{ h^\prime(u) R^2(u)}{\sigma}$$
we obtain
\begin{align*}
\varphi^\prime(u)&= |\Omega|\epsilon^\prime_2(u) -[\![\epsilon^\prime(u)]\!]|\Omega_1| + m\omega_n(\frac{\sigma}{R(u)} -[\![\epsilon(u)]\!])R^{n-1}(u)R^{\prime}(u)\\
&=|\Omega|\kappa_2(u) -[\![\kappa(u)]\!]|\Omega_1|+ m\omega_n u h^\prime(u) R^{n-1}(u) R^\prime(u)\\
&= (\kappa(u)|1)_{L_2(\Omega)} -uh^\prime(u)|\Sigma| \frac{ h^\prime(u) R^2(u)}{\sigma}\\
&= \left\{\frac{\sigma u(\kappa(u)|1)_{L_2(\Omega)}}{l^2(u)R^2(u)|\Sigma|}-1\right\}
\frac{l^2(u)R^2(u)|\Sigma|}{\sigma u},
\end{align*}
with $l(u)=u h^\prime(u)$. It will turn out that in the case of
connected phases the term in the parentheses determines whether an
equilibrium is stable: it is stable if $\varphi^\prime(u)<0$ and
unstable if $\varphi^\prime(u)>0$; see Theorem \ref{stability}
below.

In general it is not a simple task to analyze the equation for the temperature
$$\varphi(u)= |\Omega|\epsilon_2(u) + c_n \Big( \frac{1}{h(u)^{n-1}} +(n-1)u\frac{h^\prime(u)}{h(u)^n}\Big)=\varphi_0,$$
unless more properties of the functions $\epsilon_2(u)$ and in particular of $h(u)$ are known. A natural assumption is that $h$ has exactly one positive zero $u_m>0$, the melting temperature.  Therefore we look at two examples.
\begin{example}
Suppose that the heat capacities are identical, i.e.\ $[\![\kappa]\!]\equiv 0$.
This implies
$$ u h^{\prime\prime}(u)= u[\![\psi^{\prime\prime}(u)]\!]= -[\![\kappa(u)]\!]\equiv 0,$$
which means that $h(u) = h_0 + h_1u$ is linear. The melting temperature then is $0<u_m=-h_0/h_1$, hence we have two cases.

\medskip

\noindent
{\bf Case 1.} \, $h_0<0$, $h_1>0$; this means $l(u_m)>0$.\\
Then the relevant temperature range is $u>u_m$, as $h$ is positive there.
We assume now that $\epsilon_2$ is increasing and convex.
 As $u\to u_m+$ we have $h(u)\to0$, hence $\varphi(u)\to\infty$, and also
 $\varphi(u)\to \infty $ for $u\to \infty$ since $\epsilon_2(u)$ is increasing and convex. Further, we have
\begin{align*}
\varphi^\prime(u)&=|\Omega|\epsilon^\prime_2(u) -n(n-1)c_n \frac{h_1^2u}{(h_0+h_1u)^{n+1}},\\
\varphi^{\prime\prime}(u)&=|\Omega|\epsilon^{\prime\prime}_2(u) +n(n-1)c_n h_1^2
\frac{-h_0+ n h_1u}{(h_0+h_1u)^{n+2}}>0,
\end{align*}
which shows that $\varphi(u)$ is strictly convex for $u>u_m$. Thus $\varphi(u)$ has a unique minimum $u_0>u_m$,
$\varphi(u)$ is decreasing for $u_m<u<u_0$ and increasing for $u>u_0$. Thus there are precisely two equilibrium
temperatures $u_*^+\in(u_0,\infty)$ and $u_*^-\in(u_m,u_0)$ provided
$\varphi_0> \varphi(u_0)$ and none if $\varphi_0<\varphi(u_0)$. The smaller temperature leads to stable equilibria while the larger to unstable ones.

\medskip

\noindent
{\bf Case 2.} \, $h_0>0$, $h_1<0$; this means $l(u_m)<0$.\\
Then the relevant temperature range is $u<u_m$ as $h$ is positive
there. As $u\to u_m-$ we have $h(u)\to0^+$ hence
$\varphi(u)\to-\infty$, and as $u\to 0^+$ we have
$\varphi(u)\to\varphi(0)= |\Omega|\epsilon_2(0) + c_n/h_0^{n-1}$,
assuming that $\epsilon_2(0):=\lim_{u\to 0^+}\epsilon_2(u)$ exists.
Further, for $u$ sufficiently close to zero $\varphi^\prime(u)$ is
positive, since $\kappa_2=\epsilon_2'>0$, and
$\varphi^\prime(u)\to-\infty$ as $u\to u_m-$. Therefore
$\varphi^\prime(u)$ admits at least one zero in $(0,u_m)$. But there
may be more than one unless $\epsilon_2(u)$ is concave, so let us
assume this. Let $u_0\in(0,u_m)$ denote the absolute maximum of
$\varphi(u)$ in $(0,u_m)$. Then there is exactly one equilibrium
temperature $u_*\in(u_0,u_m)$ if $\varphi_0<\varphi(0)$ and it is
stable; there are exactly two equilibria $u_*^-\in(0,u_0)$ and
$u_*^+\in(u_0,u_m)$ if $\varphi(0)<\varphi_0<\varphi(u_0)$, the
first one is unstable, the second is stable. If
$\varphi_0>\varphi(u_0)$, there are no equilibria.

\medskip

\noindent
Note that in both cases these equilibrium temperatures give rise to equilibria only if
the corresponding radius $R(u)$ is smaller than $R_m$.
\end{example}
\begin{example}
Suppose that the internal energies $\epsilon_i(u)$ are linearly
increasing, i.e.
$$ \epsilon_i(u)= a_i +\kappa_i u,\quad i=1,2,$$
where $\kappa_i>0$, and now $[\![\kappa]\!]\neq0$. The identity $\epsilon_i = \psi_i-u \psi_i^\prime$ then leads to
$$\psi_i(u) = a_i+b_iu-\kappa_i u \ln u,\quad i=1,2,$$
where the constants $b_i$ are arbitrary.
This yields with $\alpha=[\![a]\!]$, $\beta=[\![b]\!]$ and $\delta =[\![\kappa]\!]$
$$ h(u)= \alpha +\beta u -\delta u \ln u.$$
Scaling the temperature by $u=u_0w$ with $\beta-\delta\ln u_0=0$ and scaling $h$ we may assume $\beta=0$ and $\delta=\pm 1$.
Then we have to investigate the equation $\varphi(w)=\varphi_1$, where
$$\varphi(w)= cw +\{\frac{1}{h^{n-1}(w)}+(n-1)w\frac{h^\prime(w)}{h^n(w)}\},\quad h(w)= \pm(\alpha+w\ln  w),$$
with $c>0$ and $\alpha,\varphi_1\in\R$. The requirement of existence of a melting temperature $w_m>0$, i.e.\ a zero of $h(w)$ leads to the restriction $\alpha\leq 1/e$.

Actually, the requirement that the melting temperature is unique, i.e.\ that $h$ has exactly one positive zero implies $\alpha<0$. Indeed, for $\alpha\in(0,1/e)$ there is a second zero $w_->0$ of $h$, and $h$ is positive in $(0,w_-)$. Equilibrium temperatures in this range
would not make sense physically.

Also here we have to distinguish
between two cases, namely that of a plus-sign where the relevant temperature range is
$w>w_m$ and in case of a minus-sign it is $(0,w_m)$. Note that $h$ is convex in the first, and concave in the second case.
\medskip\\
\noindent
{\bf Case 1:}\, For the derivatives we get in the first case
\begin{align*}
\varphi^\prime(w)&=c+(n-1)\left\{\frac{h(w)-nw (h^\prime(w))^2}{h^{n+1}(w)}\right\},\\
\varphi^{\prime\prime}(w)&=n(n-1)\frac{h^\prime(w)}{h^{n+2}(w)}\Big\{(n+1)w (h^\prime(w))^2-h(w)(3+h^\prime(w))\Big\}.
\end{align*}
We have $\varphi(w)\to\infty$ for $w\to\infty$ and for $w\to w_m+$, hence $\varphi(w)$ has a global minimum $u_0$ in $(u_m,\infty)$. Further, $\varphi^{\prime\prime}(w)>0$ in $(w_m,\infty)$, hence  the minimum is unique and there are precisely two equilibrium temperatures $w_*^-\in(w_m,w_0)$ and $w_*^+\in(w_0,\infty)$, provided $\varphi_1 >\varphi(w_0)$, the first one is stable, the second unstable.

To prove convexity of $\varphi$ we write
$$(n+1)w (h^\prime(w))^2-h(w)(3+h^\prime(w))=(n-1)w (h^\prime(w))^2+f(w),$$
where
$$f(w)=2w (h^\prime(w))^2-h(w)(3+h^\prime(w))= 2w(1+\ln w)^2 - (\alpha+w\ln w)(4+\ln w).$$
We then have $f(w_m)=2w_m(1+\ln w_m)^2>0$, and
$$f^\prime(w)=(1+\ln w)^2 +1 -\alpha/w> 1-\alpha/w\geq0,$$
for $\alpha\leq 1/e<w_m\leq w$. Let us illustrate the sign in $h$
with the water-ice system, ignoring the density jump of water at
freezing temperature. So suppose that $\Omega_2$ consists of ice and
$\Omega_1$ of water. In this case we have $\kappa_1>\kappa_2$ hence
$\delta<0$ which implies the plus-sign for $h$. Here we obtain
$w_*^\pm>w_m$, i.e.\ the ice is overheated. Equilibria only exist if
$\phi_1$ is large enough, which means that there is enough energy in
the system. If the energy in the system is very large then the
stable equilibrium temperature $w_*^-$ comes close to the melting
temperature $w_m$ and then $R(w)$ will become large, eventually
larger than $R^*$. This excludes equilibria in $\Omega$, the
physical interpretation is that everything will eventually melt.  On
the other hand, if $\Omega_1$ consists of ice and $\Omega_2$ of
water, we have the minus sign, which we want to consider next. Here
we expect under-cooling of the water-phase, existence of equilibria
only for low values of energy, and if the energy in the system is
too small everything will freeze.
\medskip

\noindent {\bf Case 2:} \, Assume the minus-sign for $h$ and let
$\alpha<0$. Then the relevant temperature range is $(0,w_m)$. Here
we have $\varphi(w)\to -\infty$ as $w\to w_m-$ and $\varphi(w)\to
1/|\alpha|^{n-1}>0$ as $w\to 0^+$.

To investigate concavity of $\varphi$ in the interval $(0,w_m)$, we recompute the derivatives of $\varphi$.
\begin{align*}
\varphi^\prime(w)&= c -(n-1) \Big\{\frac{1}{h^n(w)}+ n\frac{w (h^\prime(w))^2}{h^{n+1}(w)}\Big\},\\
\varphi^{\prime\prime}(w)&=n(n-1)\frac{h^\prime(w)}{h^{n+2}(w)}\Big\{(n+1)w (h^\prime(w))^2+h(w)(3-h^\prime(w))\Big\}.
\end{align*}
Setting $w_+=1/e$, for $w\in(w_+,w_m)$ we have $h(w)>0$ and $h^\prime(w)<0$ hence $\varphi^{\prime\prime}(w)<0$. On the other hand, for $w\in(0,w_+)$, both  $h(w)$ and $h^\prime(w)$ are positive. Then we rewrite
$$(n+1)w (h^\prime(w))^2+3h(w)-h(w)h^\prime(w)=(n-1)w(1+\ln w)^2+f(w),$$
where
\begin{align*}
f(w)&=2w (h^\prime(w))^2+h(w)(3-h^\prime(w))\\
&= 2w(1+\ln w)^2 - (\alpha+w\ln w)(4+\ln w),\\
f^\prime(w)&=(1+\ln w)^2 +1-\alpha/w>0,
\end{align*}
provided $\alpha\leq0$. This shows that $f$ is increasing,
$f(w)\to-\infty$ as $w\to 0^+$, and $f(1/e^3)=11/e^3-\alpha>0$. On the
other hand, the function $w(1+\ln w)^2$ is increasing in
$(0,1/e^3)$, hence $\varphi^{\prime\prime}(w)$ has a unique zero
$w_- \in (0,1/e^3)$. Therefore $\varphi$ is concave in $(0,w_-)\cup
(w_+,w_m)$ and convex in $(w_-,w_+)$, and $\varphi^\prime$ has a
minimum at $w_-$ and a maximum at $w_+$. Observe that
$\varphi^\prime(w)<c$, $\varphi^\prime(w)\to -\infty$ for $w\to
w_m-$ and $\varphi^\prime(0)= c-
(n-1)/|\alpha|^n<\varphi^\prime(w_+)$. Therefore,  $\varphi^\prime$
may have no, one, two, or three zeros in $(0,w_m)$ depending on the
value of $c>0$. However, if $c>0$ is large enough then
$\varphi^\prime$ has only one zero $w_1$ which lies in $(w_+,w_m)$.
In this case $\varphi$ is increasing in $(0,w_1)$ and decreasing in
$(w_1,w_m)$, hence for $\varphi_1\in(\varphi(0),\varphi(w_1))$ there
are precisely two equilibrium temperatures, the smaller leads to
unstable, the larger to a stable equilibrium. If
$\varphi_1<\varphi(0)$ there is a unique equilibrium which is
stable, and in case $\varphi_1>\varphi(w_1)$ there is none. However,
in general there may be up to four equilibrium temperatures.
\end{example}
\subsection{Linearization at equilibria}
The linearization at an equilibrium $(u_*,\Gamma_*)$ with
$R_*=\sigma/[\![\psi(u_*)]\!]$, reads
\begin{equation}
\label{lin}
\left\{\begin{aligned}
\kappa_*\partial_tv-d_*\Delta v&=f &&\text{in} && \Omega\setminus\Gamma_*\\
\partial_{\nu_\Omega} v&=0 &&\text{on} && \partial \Omega\\
[\![v]\!]&=0 &&\text{on}&& \Gamma_*\\
(l_*/u_*) v + \sigma A_* \rho -\gamma_* \partial_t\rho&=g &&\text{on}&& \Gamma_*\\
l_*\partial_t \rho -[\![d_*\partial_\nu v]\!]&=h &&\text{in}&& \Gamma_*\\
v(0)=v_0,\ \rho(0)&=\rho_0.&& &&
\end{aligned}\right.
\end{equation}
Here
$$
\kappa_*=\kappa(u_*),\quad d_*=d(u_*), \quad l_*=l(u_*),\quad \gamma_\ast=\gamma(u_\ast),
\quad A_*=\frac{1}{n-1}(\frac{n-1}{R_*^2}+\Delta_*),$$
where $\Delta_*$
denotes the {\em Laplace-Beltrami} operator on $\Gamma_*$.
\medskip\\
We note that if $l_*=0$ and $\gamma_*=0$ then the problem is not
well-posed. On the other hand, if $l_*\neq0$ and $\gamma_*=0$, then
the operator $-L_0$ defined by
\begin{equation}
\label{L-0}
\begin{split}
D(L_0)&=\big\{ (v,\rho)\in [H^2_p(\Omega\setminus\Gamma_*)\cap
C(\bar\Omega)]
\times W^{4-1/p}_p(\Gamma_*): \\
 &\hspace{7mm}\partial_{\nu_\Omega}v=0,\; (l_*/u_*)v+\sigma A_* \rho=0,\;
[\![d_*\partial_\nu v]\!]\in W^{2-2/p}_p(\Gamma_*)\big\},\\
L_0(u,\rho)&= ((-d_*/\kappa_*)\Delta v, -[\![(d_*/l_*)\partial_\nu v]\!]),\\
\end{split}
\end{equation}
generates an analytic $C_0$-semigroup with maximal regularity in
$$X_0:= L_p(\Omega)\times W^{ 2-2/p}_p(\Gamma_*).$$
More precisely, we have the following result.
\begin{theorem}
\label{MR-0}
 Let $3<p<\infty$, $\sigma>0$, suppose $\gamma_*=0$ and let
$l_*\neq0$. Then for each finite interval $J=[0,t_0]$, there is a unique solution $z=(v,\rho)\in \EE(J)$
of (\ref{lin}) if and only if the data $(f,g,h)$ and $z_0=(v_0,\rho_0)$ satisfy
$$(f,g,h)\in\FF(J),\quad
z_0\in [W^{2-2/p}_p(\Omega\setminus\Gamma_*)\cap C(\bar{\Omega})]\times W^{4-3/p}_p(\Gamma_*)$$
and the compatibility conditions
$$\partial_{\nu_\Omega}v_0=0,\quad (l_*/u_*)v_0+\sigma A_* \rho_0=g(0),\quad
h(0)+[\![d_*\partial_\nu v_0]\!]
\in W^{2-6/p}_p(\Gamma_*).$$
The operator $-L_0$ defined above generates an analytic $C_0$-semigroup in $X_0$ with
maximal regularity of type $L_p$.
\end{theorem}
\noindent
In case $\gamma_*>0$, similar assertions are valid for $L_\gamma$  in
$$X_\gamma:= L_p(\Omega)\times W^{2-1/p}_p(\Gamma_*),$$
where
\begin{equation}
\label{L-gamma}
\begin{split}
D(L_\gamma)&=\big\{(v,\rho)\in [H^2_p(\Omega\setminus\Gamma_*)\cap
C(\bar\Omega)]
\times W^{4-1/p}_p(\Gamma_*): \\
&\hspace{15mm}\partial_{\nu_\Omega}v=0,\;
(l_*^2/u_*)v+l_*\sigma A_* \rho=\gamma_*[\![d_*\partial_\nu v]\!]\big\}, \\
L_\gamma(v,\rho)&= ((-d_*/\kappa_*)\Delta v, -(\sigma/\gamma_*)A_*\rho-(l_*/u_*\gamma_*)v).\\
\end{split}
\end{equation}
The main result on the problem (\ref{lin}) for $\gamma_*>0$ is the following.
\begin{theorem}
\label{MR-gamma}
Let $3<p<\infty$, and suppose $\sigma,\gamma_*>0$.
 Then for each finite interval $J=[0,t_0]$, there is a unique solution $z=(v,\rho)\in\EE(J)$
of (\ref{lin}) if and only if the data $(f,g,h)$ and $z_0=(v_0,\rho_0)$ satisfy
$$(f,g,h)\in\FF(J),\quad z_0\in [W^{2-2/p}_p(\Omega\setminus\Gamma_*)\cap C(\bar{\Omega})]\times W^{4-3/p}_p(\Gamma_*)$$
and the compatibility condition
$$\partial_{\nu_\Omega}v_0=0,\quad (l_*^2/u_*)v_0+l_*\sigma A_* \rho_0-\gamma_*[\![d\partial_\nu v_0]\!]= l_*g(0)+\gamma_*h(0).$$
The operator $-L_\gamma$ defined above generates an analytic $C_0$-semigroup in
$X_\gamma$ with maximal regularity of type $L_p$.
\end{theorem}
\begin{proof}[Proof of Theorem \ref{MR-0} and Theorem \ref{MR-gamma}]
These results are, up to the last assertions, special cases of Theorems \ref{linpeq0} resp.\ \ref{linpneq0} in Section 3. In addition, since
the Cauchy problem for $L_\gamma$ has  maximal $L_p$-regularity, we conclude in both cases
by \cite[Proposition 1.2]{Pru03} that $-L_\gamma$ generates an analytic
$C_0$-semigroup in $X_\gamma$. Recall that the spaces $\EE(J)$ are different for $\gamma=0$ and $\gamma>0$.
\end{proof}
\subsection{The eigenvalue problem}
By compact embedding, the spectrum of $L_\gamma$ consists only of countably many discrete eigenvalues of
finite multiplicity and is independent of $p$. The eigenvalue problem reads as follows
\begin{equation}
\label{evp}
\left\{\begin{aligned}
\kappa_*\lambda v-d_*\Delta v&=0 &&\text{in }&&\Omega\setminus\Gamma_*\\
\partial_{\nu_\Omega} v &=0 &&\text{on }&&\partial \Omega\\
[\![v]\!]&=0 &&\text{on }&&\Gamma_*\\
(l_*/u_*)v+\sigma A_* \rho -\gamma_* \lambda\rho &=0 &&\text{on }&&\Gamma_*\\
l_*\lambda \rho -[\![d_*\partial_\nu v]\!]&=0 &&\text{on }&&\Gamma_*.\\
\end{aligned}\right.
\end{equation}
Assume first that $\Gamma_*$ is connected. As shown in
\cite{PrSi08}, $\lambda=0$ is always an eigenvalue, and
$N(L_\gamma)$ is independent of $\gamma_*\geq0$, $\kappa_*$ and
$d_*$. We have
\begin{equation}
\label{N-L}
N(L_\gamma)={\rm span}\left\{(\frac{\sigma u_*}{l_*R^2_*},-1),(0,Y_1),\ldots,(0,Y_n)\right\},
\end{equation}
where the functions $Y_j$ denote the {\em spherical harmonics of degree one},
 normalized by $(Y_j|Y_k)_{L_2(\Gamma_*)}=\delta_{jk}$.
$N(L_\gamma)$ is isomorphic to the tangent space of $\cE$ at $(u_*,\Gamma_*)\in\cE$.
Let $\lambda\neq0$ be an eigenvalue with eigenfunction $(v,\rho)\neq0$.
Then (\ref{evp}) yields
$$\lambda\big\{|\sqrt{\kappa_*}v|^2_{L_2(\Omega)}-\sigma u_*(A_*\rho|\rho)_{L_2(\Gamma_*)}\big\} +
|\sqrt{d_*}\nabla v|^2_{L_2(\Omega)} +\gamma_*
u_*|\lambda|^2|\rho|^2_{L_2(\Gamma_*)}=0.$$ Since $A_*$ is
selfadjoint in $L_2(\Gamma_*)$, this identity shows that all
eigenvalues of $L_\gamma$ are real.
Decomposing $v=v_0+\bar{v}$, $\rho=\rho_0+\bar{\rho}$, with
$(\kappa_*|v_0)_{L_2(\Omega)}=(\rho_0|1)_{L_2(\Gamma_*)}=0$, this
identity can be rewritten as
\begin{equation}\label{evid}
\begin{aligned}
&\lambda\big\{|\sqrt{\kappa_*}v_0|^2_{L_2(\Omega)} -\sigma
u_*(A_*\rho_0|\rho_0)_{L_2(\Gamma_*)}+
\lambda u_*\gamma_*|\rho_0|^2_{L_2(\Gamma_*)}\big\} +|\sqrt{d_*}\nabla v_0|^2_{L_2(\Omega)}\\
&+\big[ \lambda \gamma_*u_* +
l_*^2|\Gamma_*|/(\kappa_*|1)_{L_2(\Omega)} - \sigma u_* /R_*^2\big
]\lambda\bar{\rho}^2 |\Gamma_*| =0.\nonumber
\end{aligned}
\end{equation}
In the case that $\Gamma_*$ is connected,
the bracket determines whether there is a positive eigenvalue.
\\
If $\Gamma_*=\bigcup\nolimits_{1\le l\le m}\Gamma^l_\ast$ consists of $m>1$ spheres $\Gamma^l_\ast$ of equal radius, then
\begin{equation}
\label{N-L-m} N(L_\gamma) = {\rm span}\left\{(\frac{\sigma
u_*}{l_*R^2_*},-1),(0,Y^l_1),\ldots,(0,Y^l_n) : \, 1\le l\le
m\right\},
\end{equation}
where the functions $Y^l_j$ denote the {\em spherical harmonics of degree one} on $\Gamma_\ast^l$
(and $Y^l_j\equiv 0$ on $\bigcup\nolimits_{i\neq l}\Gamma^i_\ast$),
normalized by $(Y^l_j|Y^l_k)_{L_2(\Gamma^l_*)}=\delta_{jk}$.
$N(L_\gamma)$ is isomorphic to the tangent space of $\cE$ at $(u_*,\Gamma_*)\in\cE$,
as will be shown in Theorem \ref{lin-stability} below.

\goodbreak
\begin{theorem}
\label{lin-stability}
Let $\sigma>0$, $\gamma_*\geq0$, $l_*\neq0$,
and assume that the interface $\Gamma_*$ consists of $m\geq1$
components. Let
\begin{equation}
\label{zeta}
 \zeta_\ast= \frac{\sigma u_*(\kappa_*|1)_{L_2(\Omega)}}{l_*^2R_*^2|\Gamma_*|},
\end{equation}
and let $\varphi$ be defined as in \eqref{eq-rel}. Then
\begin{itemize}
\item[{(i)}]
$\varphi^\prime(u_*)= (\zeta_\ast-1)l^2_*R^2_*|\Gamma_*|/({\sigma u_*})$.
\vspace{0mm}
\item[{(ii)}] $0$ is a an eigenvalue of $-L_\gamma$
with geometric multiplicity $(mn+1)$.
\vspace{0mm}
\item[{(iii)}] $0$ is semi-simple if $\zeta_\ast\neq 1$.
\vspace{0mm}
\item[{(iv)}] If $\Gamma_*$ is connected and $\zeta_\ast\le 1$, then all eigenvalues of $-L_\gamma$
are negative, except for $0$.
\vspace{0mm}
\item[{(v)}] If $\zeta_\ast>1$, then there are precisely $m$ positive eigenvalues of
$-L_\gamma$, 
\vspace{0mm}
\item[(vi)] If $\zeta_\ast\leq 1$ then $-L_\gamma$ has precisely $m-1$ positive eigenvalues.
\item[{(vii)}] $N(L_\gamma)$ is isomorphic to the tangent space $T_{(u_*,\Gamma_*)}\cE$ of $\cE$ at
$(u_*,\Gamma_*)\in\cE$.
\vspace{0mm}
\end{itemize}
\end{theorem}
\begin{remarks} (a) The result is also true if $l_*=0$ and
$\gamma_*\neq 0$. In this case $\varphi'(u_*)=
(\kappa_*|1)_{L_2(\Omega)}>0$ and $\zeta_\ast=\infty$, hence the
equilibrium is always unstable.
\vspace{1mm}\\
(b) Note that $\zeta_\ast$ does neither depend on $d$, nor on the
undercooling coefficient $\gamma$.
\vspace{1mm}\\
(c)
For the {\em Mullins-Sekerka problem}, that is, for $\kappa\equiv0$,
we have $\zeta_\ast\equiv0$, in accordance with the result obtained in
\cite{ES98}.
\vspace{1mm}\\
(d) It is shown in \cite{PrSi08} that in case $\zeta_\ast=1$ and
$\Gamma_*$  connected,  the eigenvalue $0$ is no longer semi-simple:
its algebraic multiplicity rises by $1$. This is also true if
$\Gamma_*$ is disconnected.
\end{remarks}
\begin{proof}[Proof of Theorem \ref{lin-stability}]
For the case that $\Gamma_*$ is connected this result is proved in \cite{PrSi08}.
The assertions (i)-(iii) also remain valid in the disconnected case.
However, the proof of  \cite[Theorem 2.1(e)]{PrSi08}, addressing instability, is not completely correct,
as it relies on  the assertions \cite[Proposition 3.2(b) and Proposition 5.1(c)]{PrSi08}
which are incorrect.
(We remark, though, that the instability result of \cite[Theorem 1.3]{PrSi08}
indeed is valid.)
Here we give a modified proof for \cite[Theorem 2.1(e)]{PrSi08}
which also applies in case $\Gamma_*$ is not connected.
\medskip\\
\noindent
It thus remains to prove the assertions in (v), (vi), and (vii).
If the stability condition $\zeta_\ast\leq1$ does not hold or if ${\Gamma_*}$ is disconnected, then there
is always a positive eigenvalue. To prove this we proceed as follows. Suppose $\lambda>0$ is an eigenvalue, and
that $\rho$ is known; solve the elliptic transmission problem
\begin{equation}
\label{NDdiffusion}
\left\{\begin{aligned}
\kappa_*\lambda v -d_*\Delta v &=0 &&\text{in}&& \Omega\setminus{\Gamma_*}\\
\partial_{\nu_\Omega} v&=0 &&\text{on}&&\partial\Omega \\
[\![v]\!]&=0 &&\text{on} && {\Gamma_*}\\
 -[\![d_*\partial_\nu v]\!]&= h &&\text{on}&&  {\Gamma_*}\\
\end{aligned}\right.
\end{equation}
to get $v=S_\lambda h$, with $S_\lambda$ being the solution
operator. Then taking the trace at $\Gamma_*$ we obtain
$v|_{\Gamma_*} =N_\lambda h$, where $N_\lambda$ denotes the
Neumann-to-Dirichlet operator for the transmission problem
\eqref{NDdiffusion}. Setting $h=-\lambda l_*\rho$ this implies with
the linearized Gibbs-Thomson law
 the equation
\begin{equation}\label{Tlambda} [(l_*^2/u_*)\lambda N_\lambda +\gamma_* \lambda]\rho -\sigma A_* \rho =0.
\end{equation}
$\lambda>0$ is an eigenvalue of $-L_\gamma$ if and only if
\eqref{Tlambda} admits a nontrivial solution. We consider this
problem in $L_2({\Gamma_*})$. Then $A_*$ is selfadjoint and
$$-\sigma(A_* g|g)_{L_2(\Gamma_*)} \geq -\frac{\sigma}{R_*^2} |g|^2_{L_2(\Gamma_*)}.$$
On the other hand, we will see below that $N_\lambda$ is
selfadjoint and positive semi-definite on $L_2({\Gamma_*})$.
Moreover, since $A_*$ has compact resolvent, the operator
$$
B_\lambda:=[(l_*^2/u_*)\lambda N_\lambda +\gamma_*
\lambda]-\sigma A_*$$ 
has compact resolvent as well, for each
$\lambda>0$. Therefore the spectrum of $B_\lambda$ consists only
of eigenvalues which, in addition, are real. We intend to prove
that in case either $\Gamma_*$ is disconnected or the stability
condition does not hold, $B_{\lambda_0}$ has $0$ as an eigenvalue,
for some $\lambda_0>0$.

We will need the following result on the Neumann-to-Dirichlet
operator $N_\lambda$. We denote by ${\sf e}$ the function which is
identically to one on $\Gamma_*$.
\begin{proposition}
\label{NHlambda}
The Neumann-to-Dirichlet operator $N_\lambda$ for problem \eqref{NDdiffusion} has the following properties in $L_2(\Gamma_*)$.
\begin{itemize}
\item[(i)]
If $v$ denotes the solution of \eqref{NDdiffusion}, then
\begin{equation*}
\quad (N_\lambda h|h)_{L_2(\Gamma_*)} =
\lambda |\sqrt{\kappa_*}v|^2_{L_2(\Omega)} + |\sqrt{d_*}\nabla v|_{L_2(\Omega)}^2 ,
\quad \lambda>0, \; h\in L_2(\Gamma_*).
\end{equation*}
\item[(ii)]
For each $\alpha\in(0,1/2)$ and $\lambda_0>0$ there is a constant $C>0$ such that
$$(N_\lambda h|h)_{L_2(\Gamma_*)}\geq \frac{\lambda^\alpha}{C}|N_\lambda h|_{L_2(\Gamma_*)}^2,
\quad h\in L_2(\Gamma_*),\; \lambda\geq\lambda_0.$$
In particular,
$N_\lambda$ is injective, and
$$ |N_\lambda|_{\cB(L_2(\Gamma_*))}\leq \frac{C}{\lambda^\alpha}, \quad \lambda\geq\lambda_0.$$
\item[(iii)]
 On $L_{2,0}(\Gamma_*)= \{\theta\in L_2(\Gamma_*):\, (\theta|{\sf e})_{L_2(\Gamma_*)}=0\}$, we even have
$$(N_\lambda h|h)_{L_2(\Gamma_*)}\geq \frac{(1+\lambda)^\alpha}{C}|N_\lambda h|_{L_2(\Gamma_*)}^2,\quad h\in L_{2,0}(\Gamma_*),\; \lambda>0,$$
and
$$ |N_\lambda|_{\cB(L_{2,0}(\Gamma_*),L_2(\Gamma_*))}\leq \frac{C}{(1+\lambda)^\alpha}, \quad \lambda>0.$$
\end{itemize}
\noindent In particular, for $\lambda=0$, \eqref{NDdiffusion} is
solvable if and only if $(h|{\sf e})_{\Gamma_*}=0$, and then the
solution is unique up to a constant.
\end{proposition}
\begin{proof}[Proof of Proposition \ref{NHlambda}]
The first assertion follows from the divergence theorem. The second and third assertions  are consequences of trace theory, combined with Poincar\'e's inequality. The last assertion is a standard statement in the theory of elliptic transmission problems. We refer to \cite{PrSi08}.
\end{proof}
\noindent
{\em Proof of Theorem \ref{lin-stability}, continued:}
\\
(a) \, Suppose first that $\Gamma_*$ is connected. Consider $h={\sf
e}$. Then with $c_*:=l_*^2/u_*\geq0$ we have
$$(B_\lambda {\sf
e}|{\sf e})_{L_2(\Gamma_*)}= c_*\lambda(N_\lambda{\sf e}|{\sf
e})_{L_2(\Gamma_*)} +\lambda\gamma_*|{\sf e}|^2_{L_2(\Gamma_*)}-
\frac{\sigma}{R_*^2}|{\sf e}|_{L_2(\Gamma_*)}^2.$$ We compute the
limit $\lim_{\lambda \to0}\lambda(N_\lambda{\sf e}|{\sf
e})_{L_2(\Gamma_*)}$ as follows. First solve the problem
\begin{equation}
\label{NDdiffusion0}
\left\{\begin{aligned}
 -d_*\Delta v &=-\kappa_* a_0 &&\text{in} && \Omega\setminus{\Gamma_*} \\
\partial_{\nu_\Omega} v&=0 &&\text{on}&& \partial\Omega \\
[\![v]\!]&=0 &&\text{on}&& {\Gamma_*} \\
 -[\![d_*\partial_\nu v]\!]&= {\sf e} &&\text{on}&&{\Gamma_*},\\
\end{aligned}\right.
\end{equation}
where $a_0=|\Gamma_*|/(\kappa_*|1)_{L_2(\Omega)}$, which is solvable
since the necessary compatibility condition holds. Let $v_0$ denote
the solution which satisfies the normalization condition
$(\kappa_*|v_0)_{L_2(\Omega)}=0$. Then $v_\lambda:=S_\lambda{\sf
e}-v_0- a_0/\lambda$ satisfies the problem
\begin{equation}
\label{NDdiffusion1}
\left\{\begin{aligned}
\kappa_* \lambda v_\lambda -d_*\Delta v_\lambda &=-\kappa_*\lambda v_0
&&\text{in}&&\Omega\setminus{\Gamma_*}\\
\partial_{\nu_\Omega} v&=0 &&\text{on}&&\partial\Omega\\
[\![v_\lambda]\!]&=0 &&\text{on}&& {\Gamma_*}\\
 -[\![d_*\partial_\nu v_\lambda]\!]&= 0 &&\text{on}&&{\Gamma_*}.\\
\end{aligned}\right.
\end{equation}
By the normalization $(\kappa_*|v_0)_{L_2(\Omega)}=0$  we see that
$v_\lambda$ is bounded in $W^2_2(\Omega\setminus\Gamma_*)$ as
$\lambda\to0$. Hence we have
$$\lim_{\lambda\to0}\lambda N_\lambda{\sf e} = \lim_{\lambda\to0}[\lambda v_\lambda|_{\Gamma_*}+ \lambda v_0|_{\Gamma_*}+
 a_0]= a_0 =|\Gamma_*|/(\kappa_*|1)_{L_2(\Omega)}.$$
This then implies
$$\lim_{\lambda\to0} (B_\lambda {\sf e}|{\sf e})_{L_2(\Gamma_*)} = c_* \frac{|\Gamma_*|^2}{(\kappa_*|1)_{L_2(\Omega)}}-
\,\frac{\sigma}{R_*^2}\,|\Gamma_*|<0,$$ if the stability condition
does not hold, i.e. if $\zeta_\ast>1$.

\medskip

\noindent (b) \, Next suppose that $\Gamma_*$ is disconnected. If
$\Gamma_*$ consists of $m$ components $\Gamma_*^k$, $k=1,..., m$,
we set ${\sf e}_k=1$ on $\Gamma_*^k$ and zero elsewhere. Let
$h=\sum_k a_k {\sf e}_k\neq0$ with $\sum_k a_k=0$, hence $Q_0h=h$,
where $Q_0$ is the canonical projection onto $L_{2,0}(\Gamma_*)$,
\[
Q_0 h=h-\,\frac{(h|{\sf e})_{L_2(\Gamma_*)}}{|\Gamma_*|}.
\]
Then
$$\lim_{\lambda\to0} \lambda N_\lambda h = \lim_{\lambda\to0} \lambda N_\lambda Q_0h=0,$$
since $N_\lambda Q_0$ is bounded as $\lambda\to0$. This implies
$$\lim_{\lambda\to0} (B_\lambda h|h)_{L_2(\Gamma_*)} = - \,\frac{\sigma}{R_*^2}\, \sum_k |\Gamma_*^k|a_k^2<0.$$

\medskip

\noindent (c) \, Next we consider the behavior of $(B_\lambda
h|h)_{L_2(\Gamma_*)}$ as $\lambda\to\infty$. We want to show that
$B_\lambda$ is positive semi-definite for large $\lambda$. For this
purpose we introduce the projections $P$ and $Q$ by
$$Ph = c_m\sum_{k=1}^m (h|{\sf e}_k)_{L_2(\Gamma_*)}{\sf e}_k,\quad Q=I-P,$$
where $ c_m=m/|\Gamma_*|$ in case $\Gamma_*$ has $m$ components.
Recall that $|\Gamma_*^k|=|\Gamma_*|/m$ for $k=1,\ldots,m$. Then
with $h_k =(h|{\sf e}_k)_{L_2(\Gamma_*)}$
\begin{align*}
|(N_\lambda Ph|Qh)_{L_2(\Gamma_*)}|&\leq c_m \sum_k |h_k|\,|(N_\lambda Qh|{\sf e}_k)_{L_2(\Gamma_*)}|\\
&\leq C\sum_k |h_k||N_\lambda Qh|_{L_2(\Gamma_*)}\leq C\lambda^{-\alpha/2}\sum_k|h_k|(N_\lambda Qh|Qh)^{1/2}_{L_2(\Gamma_*)}\\
&\leq C\lambda^{-\alpha/2}[\sum_k|h_k|^2 +m(N_\lambda Qh|Qh)_{L_2(\Gamma_*)}]\\
&\leq C\lambda^{-\alpha/2}[|Ph|_{L_2(\Gamma_*)}^2 +(N_\lambda
Qh|Qh)_{L_2(\Gamma_*)}],
\end{align*}
for $\lambda>0$, and $C$ standing for a generic positive constant,
which may change from line to line. Hence for $\lambda\ge
\lambda_0$, with $\lambda_0$ sufficiently large, we have
\begin{align*}
(N_\lambda h|h)_{L_2(\Gamma_*)}&= (N_\lambda Qh|Qh)_{L_2(\Gamma_*)}+2(N_\lambda Qh|Ph)_{L_2(\Gamma_*)}+ (N_\lambda Ph|Ph)_{L_2(\Gamma_*)}\\
&\geq \frac{1}{2}(N_\lambda Qh|Qh)_{L_2(\Gamma_*)}+ (N_\lambda
Ph|Ph)_{L_2(\Gamma_*)}-
\frac{C}{\lambda_0^{\alpha/2}}|Ph|_{L_2(\Gamma_*)}^2.
\end{align*}
This implies
\begin{align*}(B_\lambda h|h)_{L_2(\Gamma_*)}
= & \,c_*\lambda(N_\lambda h|h)_{L_2(\Gamma_*)} +\gamma_*\lambda |h|^2_{L_2(\Gamma_*)}
-\sigma (A_*h|h)_{L_2(\Gamma_*)}\\
\geq &\, \frac{c_*\lambda}{2}(N_\lambda Qh|Qh)_{L_2(\Gamma_*)}
+ c_*\lambda(N_\lambda Ph|Ph)_{L_2(\Gamma_*)}\\
&\,-\sigma (A_*Qh|Qh)_{L_2(\Gamma_*)}- c|Ph|_{L_2(\Gamma_*)}^2.
\end{align*}a
Since $N_\lambda$ is positive semi-definite and also $-A_*Q$ has
this property, we only need to prove $\lambda(N_\lambda
Ph|Ph)_{L_2(\Gamma_*)}\to \infty$ as $\lambda\to\infty$.

To prove this, similarly as before we estimate
$$ |(N_\lambda{\sf e}_i|{\sf e}_j)_{L_2(\Gamma_*)}|\leq C|N_\lambda{\sf e}_i|_{L_2(\Gamma_*)}
\leq \tilde{C}\lambda_0^{-\alpha/2}(N_\lambda{\sf e}_i|{\sf
e}_i)^{1/2}_{L_2(\Gamma_*)},$$ and choosing $\lambda_0$ sufficiently
large this yields
$$ (N_\lambda Pg|Pg)_{L_2(\Gamma_*)}\geq c_0\Big[ \min_i (N_\lambda {\sf e}_i|{\sf e}_i)_{L_2(\Gamma_*)}-
\frac{C}{\lambda_0^{\alpha/2}}\Big]|Pg|_{L_2(\Gamma_*)}^2.$$
Therefore it is sufficient to show
\begin{equation}\label{asymptNH}
\lim_{\lambda\to\infty} \lambda(N_\lambda{\sf e}_k|{\sf e}_k)_{L_2(\Gamma_*)}=\infty, \quad k=1,\ldots,m.
\end{equation}
So suppose, on the contrary, that $\lambda_j(N_{\lambda_j} g|g)_{L_2(\Gamma_*)}$ is bounded, for some $g={\sf e}_k$ and some
sequence $\lambda_j\to\infty$. Then the corresponding solution
$v_j$ of \eqref{NDdiffusion} is such that $\lambda_j v_j$ is
bounded in $L_2(\Omega)$, hence
 has a weakly convergent subsequence.
 W.l.o.g.\ $\lambda_jv_j\to v_\infty$ weakly in $L_2(\Omega)$.
 Fix a test function $\psi\in \cD(\Omega\setminus\Gamma_*)$. Then
\begin{align*}
\lambda_j(\kappa_* v_j|\psi)_{L_2(\Omega)} & = (d_*\Delta
v_j|\psi)_{L_2(\Omega)}=
 (v_j|d_*\Delta\psi)_{L_2(\Omega)}\\
& =
 (\lambda_j v_j|d_*\Delta\psi)_{L_2(\Omega)}/\lambda_j\to 0
\end{align*}
as $j\to\infty$, hence $v_\infty=0$ in $L_2(\Omega)$. On the other
hand we have
\begin{align*}
0<|\Gamma_*|/m &= \int_{\Gamma_*} g\, ds = \int_{\Gamma_*} -[\![d_*\partial_\nu v_j]\!]\, ds\\
&=\int_\Omega d_*\Delta v_j\, dx = \lambda_j\int_\Omega\kappa_*v_j\, dx 
\to\int_\Omega \kappa_* v_\infty\, dx,
\end{align*}
hence $v_\infty$ is nontrivial, a contradiction. This implies that \eqref{asymptNH} is valid, provided $l_*>0$.

On the other hand, in case $l_*=0$ we have $\gamma_*>0$, hence
$\lambda\gamma_*|g|_{L_2(\Gamma_*)}^2\to\infty$, so also in this
case $B_\lambda$ is positive semi-definite for large $\lambda$.

\medskip

\noindent (d) \, Summarizing, we have shown that $B_\lambda$ is
not positive semi-definite for small $\lambda>0$ if either
$\Gamma_*$ is not connected or the stability condition does not
hold, and $B_\lambda$ is always positive semi-definite for large
$\lambda$. Set
$$\lambda_0 = \sup\{\lambda>0:\, B_\mu \mbox{ is not positive semi-definite for each } \mu\in(0,\lambda]\}.$$
Since $B_\lambda$ has compact resolvent, $B_\lambda$ has a negative eigenvalue for each $\lambda<\lambda_0$. This implies that $0$ is an eigenvalue of
$B_{\lambda_0}$, thereby proving that $-L_\gamma$ admits the positive eigenvalue $\lambda_0$.

Moreover, we have also shown that
$$ B_0h= \lim_{\lambda\to0} c_*\lambda N_\lambda h - \sigma A_* h =
 c_*|\Gamma_*|/(\kappa_*|1)_{L_2(\Omega)} P_0 h - \sigma A_* h ,$$
where $P_0h:=(I-Q_0)h=(h|{\sf e})_{L_2(\Gamma_*)}/|\Gamma_*|$.
Therefore, $B_0$ has the eigenvalue
$c_*|\Gamma_*|/(\kappa_*|1)_{L_2(\Omega)}-\sigma/R_*^2$ with
eigenfunction ${\sf e}$, and in case $m>1$ it also possesses the
eigenvalue $-\sigma/R_*^2$ with precisely $m-1$ linearly
independent eigenfunctions of the form $\sum_k a_k {\sf e}_k$ with
$\sum_k a_k=0$. This implies that $-L_\gamma$ has exactly $m$
positive eigenvalues if the stability condition does not hold, and
$m-1$ otherwise.
\medskip\\
(e) \, It remains to show assertion (vii).
Suppose for the moment that $\Gamma_*$ consists of a single sphere of radius
$R_*=\sigma/[\![\psi(u_*)]\!]$, centered at the origin of $\R^n$.
Suppose ${\mathcal S}\subset \Omega$ is a sphere that is
sufficiently close to $\Gamma_*$. Denote by $(z_1,\ldots,z_n)$ the
coordinates of its center and let $z_0$ be such that
$\sigma/[\![\psi(u_*+z_0)]\!]$
corresponds to its radius.
Then, by \cite[Section 6]{ES98}, the
sphere ${\mathcal S}$ can be parametrized over $\Gamma_*$ by the
distance function
\[
\rho(z)=\sum_{j=1}^n z_j Y_j-R_*+\sqrt{(\sum_{j=1}^n z_j
Y_j)^2+(\sigma/[\![\psi(u_*+z_0)]\!])^2-\sum_{j=1}^n z_j^2}.
\]
Denoting by $O$ a sufficiently small neighborhood of $0$ in
$\R^{n+1}$, the mapping
\begin{equation*}
[z\mapsto \Psi(z):=(u_*+z_0,\rho(z))]:O\to W^2_p(\Omega)\times W^{4-1/p}_p(\Gamma_*)
\end{equation*}
is $C^1$ (in fact $C^k$ if $\psi$ is $C^k$), and the derivative at $0$ is given by
\begin{equation*}
\Psi'(0)h=\big(1,-\sigma [\![\psi^\prime(u_*)]\!]/[\![\psi(u_*)]\!]^2\big)h_0
+\big(0,\sum_{j=1}^n h_j Y_j\big),
\quad h\in \R^{n+1}.
\end{equation*}
Noting that $\sigma [\![\psi^\prime(u_*)]\!]/
[\![\psi(u_*)]\!]^2=l_*R^2_*/(\sigma u_*)$ we can conclude that
near $(u_*,\Gamma_*)$ the set $\cE$ of equilibria is a
$C^1$-manifold in $W^2_p(\Omega)\times W^{4-1/p}_p(\Gamma_*)$ of
dimension $n+1$, and that the tangent space
$T_{(u_*,\Gamma_*)}(\cE)$ coincides with $N(L_\gamma)$, see
\eqref{N-L}. It is now easy to see that this result remains valid
for the case of $m$ spheres of the same radius $R_\ast$. The
dimension of $\cE$ is then given by $(mn+1)$, as $mn$ parameters are
needed to locate the respective centers, and one additional
parameter is needed to track the temperature.
\end{proof}
\section{Nonlinear Stability and Instability of Equilibria}
Before we discuss nonlinear stability of equilibria, we need some preparations.
The first observation is that the equations near an equilibrium $(u_*,\Gamma_*)\in\cE$
can be restated as
\begin{equation}
\label{nonlineq0}
\left\{\begin{aligned}
\kappa_*\partial_tv-d_*\Delta v&=F_*(v,\rho)&&\text{in}&& \Omega\setminus\Gamma_* \\
\partial_{\nu_\Omega} v &=0 &&\text{on}&& \partial \Omega\\
[\![v]\!]&=0 &&\text{on}&&\Gamma_*\\
(l_*/u_*) v + \sigma A_* \rho-\gamma_* \partial_t\rho &=G_*(v,\rho)&&\text{on}&&\Gamma_*\\
l_*\partial_t \rho -[\![d_*\partial_\nu v]\!]&=H_*(v,\rho)&&\text{on}&&\Gamma_*\\
v(0)=v_0,\ \rho(0)&=\rho_0.
\end{aligned}\right.
\end{equation}
where
\begin{align*}
F_*(v,\rho)&= (\kappa_\ast-\kappa(u_*+v))\partial_t v +(d(u_*+v)-d_*)\Delta v+ d(u_*+v)M_2(\rho):\nabla^2 v\\
&\quad-d^\prime(u_*+v)|(I-M_1(\rho))\nabla v|^2 + d(u_*+v)(M_3(\rho)|\nabla v)\\
&\quad +\kappa(u_*+v)\cR(\rho)(u_*+v),\\
G_*(v,\rho)&=-([\![\psi(u_*+v)]\!]+\sigma \cH(\rho))\!+\!(l_*/u_*) v + \sigma A_* \rho
+(\gamma(u_*+v)\beta(\rho)\!-\!\gamma_*)\partial_t\rho,\\
H_*(v,\rho)&=[\![(d(u_*+v)-d_*)\partial_\nu v]\!]+ (l_*-l(u_*+v))\partial_t\rho \\
&\quad -([\![d(u_*+v)\nabla v]\!]|M_4(\rho)\nabla_\Sigma\rho)
+\gamma(u_*+v)\beta(\rho)(\partial_t\rho)^2,
\end{align*}
see Sections 2 and 3 for the definition of
$M_j(\rho)$, $j=1,\ldots,4$.
Here we replace $\partial_t\rho$ in the nonlinearities $G_\ast(v,\rho)$ and 
$H_\ast(v,\rho)$ by the following expressions:
\begin{equation*}
\begin{aligned}
\partial_t\rho & = 
\frac{1}{l(u_\ast + v)}
\big([\![d(u_\ast + v)\partial_\nu v]\!] 
+ ([\![d(u_*+v)\nabla v]\!]|M_4(\rho)\nabla_\Sigma\rho)\big)\quad\text{if}\quad \gamma=0,\\
\partial_t\rho & =
\frac{1}{\beta(\rho)\gamma(u_\ast+v)}\big([\![\psi(u_*+v)]\!]+\sigma \cH(\rho)\big)
\quad\text{if}\quad \gamma >0.
\end{aligned}
\end{equation*} 
From the equilibrium equation $[\![\psi(u_\ast)]\!]+\sigma\cH(0)=0 $
follows that
the nonlinearities satisfy
$F_*(0,0)=G_*(0,0)=H_*(0,0)=0$.
Moreover, we have
$F_*^\prime(0,0)=G_*^\prime(0,0)=H_*^\prime(0,0)=0$.
\\
The state manifold for problem \eqref{nonlineq0} near the equilibrium $(u_\ast,\Gamma_\ast)$
can then be described by
\begin{align}
\label{local-pm0} \cSM_0&=\big\{(v,\rho)\in
W^{2-2/p}_p(\Omega\setminus\Gamma_*)\times W^{4-3/p}_p(\Gamma_*):
\; \partial_{\nu_\Omega}v=0,\; [\![v]\!]=0, \\
&\hspace{1cm} (l_*/u_*) v \!+\! \sigma A_* \rho =G_*(v,\rho),\;
[\![d_*\partial_\nu v]\!]+H_*(v,\rho)\in W^{2-6/p}_p(\Gamma_*)\big\},\nonumber
\end{align}
for $\gamma_*=0$, in case $l_*\neq0$ (otherwise the linear problem is not well-posed) and
\begin{align}
\label{local-pm-gam} \cSM_\gamma&=\big\{(v,\rho)\in
W^{2-2/p}_p(\Omega\setminus\Gamma_*)\times W^{4-3/p}_p(\Gamma_*):
\; \partial_{\nu_\Omega}v=0,\; [\![v]\!]=0, \\
&\hspace{1cm} (l_*^2/u_*) v + l_*\sigma A_* \rho -\gamma_*[\![d_*\partial_\nu v]\!]=l_*G_*(v,\rho)+\gamma_*H_*(v,\rho)
 \big\},\nonumber
\end{align}
in case $\gamma_*>0$.
\medskip\\
We would like to parametrize these manifolds over their tangent spaces at $(0,0)$, given by
\begin{align}
\label{lin-pm0}
\cX_0&=\big\{(\tilde{v},\tilde{\rho})\in [W^{2-2/p}_p(\Omega\setminus\Gamma_*)\cap C(\bar\Omega)]
\times W^{4-3/p}_p(\Gamma_*):\\
&\hspace{7mm}\;\partial_{\nu_\Omega}\tilde{v}=0,\;
(l_*/u_*) \tilde{v} + \sigma A_* \tilde{\rho} =0,\;
[\![d_*\partial_\nu \tilde{v}]\!]\in W^{2-6/p}_p(\Gamma_*)\big\},\nonumber
\end{align}
respectively, for $\gamma_*>0$
\begin{align}
\label{lin-pm-gam} \cX_\gamma&=\big\{(\tilde{v},\tilde{\rho})\in
[W^{2-2/p}_p(\Omega\setminus\Gamma_*)\cap C(\bar\Omega)]
\times W^{4-3/p}_p(\Gamma_*):\\
&\hspace{7mm} \;\partial_{\nu_\Omega}\tilde{v}=0,\;
(l_*^2/u_*)\tilde{v} + l_*\sigma A_* \tilde{\rho} -\gamma_*[\![d_*\partial_\nu \tilde{v}]\!]=0\big\}.\nonumber
\end{align}
Note that the norm in $\cX_\gamma $ for $\gamma=0$ is given by
\begin{equation*}
|(\tilde{v},\tilde{\rho})|_{\cX_0}
=|\tilde{v}|_{W^{2-2/p}_p}+|\tilde{\rho}|_{W^{4-3/p}_p}
+ |[\![d_*\partial_\nu \tilde{v}]\!]|_{ W^{2-6/p}_p},
\end{equation*}
while for $\gamma>0$
it is given by
$|(\tilde{v},\tilde{\rho})|_{\cX_\gamma}=|\tilde{v}|_{W^{2-2/p}_p}+|\tilde{\rho}|_{W^{4-3/p}_p}.$
\smallskip\\
Is should be observed that $\tilde Z_\gamma$ is a linear space. The parametrization
of $\cSM_\gamma $ over the tangent space $\tilde Z_\gamma$ will facilitate the use of maximal regularity results for the stability/instability analysis. 
\smallskip
\goodbreak
\noindent
In order to determine a parameterization, we consider the linear problem
\begin{equation}
\label{lin-param}
\left\{\begin{aligned}
\kappa_*\omega v-d_*\Delta v &=0 &&\text{in}&& \Omega\setminus\Gamma_* \\
\partial_{\nu_\Omega} v&=0 &&\text{on}&&\partial \Omega\\
\mbox{}[\![v]\!]&=0 &&\text{on}&&\Gamma_*\\\
(l_*/u_*) v + \sigma A_* \rho -\gamma_* \omega\rho&=g &&\text{on}&& \Gamma_*\\
l_*\omega\rho -[\![d_*\partial_\nu v]\!]&=h &&\text{on}&&\Gamma_*.\\
\end{aligned}\right.
\end{equation}
We have the following result.
\begin{proposition}
\label{param} Suppose $p>3$, $\gamma_*\geq0$, $l_*\neq0$ in case
$\gamma_*=0$, and $\omega>0$ is sufficiently large. Then problem
\eqref{lin-param} admits a unique solution $(v,\rho)$ with
regularity
$$ v\in W^{2-2/p}_p(\Omega\setminus\Gamma_*) ,\quad \rho \in W^{4-3/p}_p(\Gamma_*)$$
if and only if the data $(g,h)$ satisfy
$$ g\in W^{2-3/p}_p(\Gamma_*),\quad h\in W^{1-3/p}_p(\Gamma_*).$$
The solution map $[(g,h)\mapsto(v,\rho)]$ is continuous in the corresponding spaces.
\end{proposition}
\begin{proof}
This purely elliptic problem can be solved in the same way as the
corresponding linear parabolic problems, cf. Theorems 4.3 and 4.4.
\end{proof}
For the parametrization we pick $\omega>0$ sufficiently large.
Given $\tilde z=(\tilde{v},\tilde{\rho})\in \cX_\gamma$
sufficiently small, we can solve the auxiliary problem
\begin{equation}
\label{nonlineq-param}
\left\{\begin{aligned}
\kappa_*\omega \bar{v}-d_*\Delta \bar{v}&=0 &&\text{in}&&\Omega\setminus\Gamma_*\\
\partial_{\nu_\Omega} \bar{v}&=0 &&\text{on}&&\partial \Omega\\
\mbox{}[\![\bar{v}]\!]&=0 &&\text{on}&&\Gamma_*\\\
(l_*/u_*) \bar{v} + \sigma A_* \bar{\rho}-\gamma_* \omega\bar{\rho}
&=G_*(\tilde{v}+\bar{v},\tilde{\rho}+\bar{\rho}) &&\text{on}&&\Gamma_*\\
l_*\omega \bar{\rho} -[\![d_*\partial_\nu \bar{v}]\!]&=H_*(\tilde{v}+\bar{v},\tilde{\rho}+\bar{\rho})
&&\text{on}&&\Gamma_*\\
\end{aligned}\right.
\end{equation}
by means of  the implicit function theorem, employing Proposition
\ref{param}. This yields a unique solution $\bar z=(\bar v,\bar
\rho)=\phi(\tilde z)\in W^{2-2/p}_p(\Omega\setminus\Gamma_*)\times W^{4-3/p}_p(\Gamma_*)$
with a $C^1$-function $\phi$ such that
$\phi(0)=0$ as well as $\phi^\prime(0)=0$. One readily verifies
that $z=\tilde{z}+ \phi(\tilde{z})\in \cSM_\gamma$. To prove
surjectivity of this map, given $(v,\rho)\in \cSM_\gamma$, we
solve the linear problem
\begin{equation}
\left\{\begin{aligned}
\kappa_*\omega \bar{v}-d_*\Delta \bar{v}&=0 &&\qquad\quad\text{in}&& \Omega\setminus\Gamma_*\\
\partial_{\nu_\Omega} \bar{v}&=0 &&\qquad\quad\text{on}&&\partial \Omega\\
[\![\bar{v}]\!]&=0 &&\qquad\quad\text{on}&&\Gamma_*\\
(l_*/u_*) \bar{v} + \sigma A_* \bar{\rho}-\gamma_* \omega\bar{\rho}&=G_*(v,\rho) &&\qquad\quad\text{on}&&\Gamma_*\\
l_*\omega \bar{\rho} -[\![d_*\partial_\nu \bar{v}]\!]&=H_*(v,\rho)&&\qquad\quad\text{on}&&\Gamma_*\\
\end{aligned}\right.
\end{equation}
and set $\tilde{z}=z-\bar{z}$. Then $\tilde z\in\tilde Z_\gamma$ and $\bar{z}=\phi(\tilde{z})$,
hence the map $[\tilde{z}\mapsto \tilde{z}+\phi(\tilde{z})]$ is also surjective near $0$.
We have, thus, obtained a local parametrization of $\cSM_\gamma$ near zero over the
tangent space $\tilde Z_\gamma$.
\medskip\\
\noindent
Next we derive a similar decomposition for the solutions of
problem \eqref{nonlineq0}. Let $z_0=(\tilde z_0,\phi(\tilde
z_0))\in\cSM_\gamma$ be given and let $z\in\EE(a)$, where we set
\begin{equation}
\label{EE(a)}
\EE(a):=\EE([0,a]),
\end{equation}
be the solution of \eqref{nonlineq0} with initial value $z_0$. Then we would like
to devise a decomposition $z=z_\infty\! +\tilde{z}+\bar{z}$, where
$\tilde{z}(t)\in \tilde Z_\gamma$ for all $t\in [0,a]$, and where
$z_\infty=\tilde z_\infty\! +\phi(\tilde z_\infty)$ is an
equilibrium for \eqref{nonlineq0}. In order to achieve this, we
consider the coupled systems of equations
\begin{equation}
\label{nonlin-param}
\left\{\begin{aligned}
\kappa_*\omega\bar{v}+\kappa_*\partial_t\bar{v}-d_*\Delta \bar{v}&=F_*(z_\infty+\tilde z+\bar z)-F_*(z_\infty) \\
\partial_{\nu_\Omega} \bar{v}& =0 \\
[\![\bar{v}]\!]&=0 && \\
(l_*/u_*) \bar{v} + \sigma A_* \bar{\rho}-\gamma_* (\partial_t\bar{\rho} +\omega\bar{\rho})
&=G_*(z_\infty+\tilde z+\bar z)-G_*(z_\infty)\\
l_*\omega \bar{\rho}+ l_*\partial_t \bar{\rho} -[\![d_*\partial_\nu \bar{v}]\!]
&=H_*(z_\infty+\tilde z+\bar z)-H_*(z_\infty)\\
\bar{z}(0)&=\phi(\tilde z_0)-\phi(\tilde z_\infty),
\end{aligned}\right.
\end{equation}
\\
and
\\
\begin{equation}
\label{nonlin-param-red}
\left\{\begin{aligned}
\kappa_*\partial_t\tilde{v}-d_*\Delta \tilde{v}
&=\kappa_*\omega\bar{v}\\
\partial_{\nu_\Omega} \tilde{v}&=0 \\
[\![\tilde{v}]\!]&=0 \\
(l_*/u_*) \tilde{v} + \sigma A_* \tilde{\rho}-\gamma_* \partial_t\tilde{\rho}&=-\gamma_*\omega\bar{\rho}\\
l_*\partial_t \tilde{\rho} -[\![d_*\partial_\nu \tilde{v}]\!]&=l_*\omega\bar{\rho}\\
\tilde{z}(0) &=\tilde{z}_0-\tilde{z}_\infty.
\end{aligned}\right. \hspace{2.8cm}
\end{equation}
It should be mentioned that $F_*(z_\infty)=0$, as can be seen from the
equilibrium equation for \eqref{nonlineq0} and the fact that $v_\infty=$ constant
for $z_\infty=(v_\infty,\rho_\infty)$.
For reasons of symmetry and consistency we will, nevertheless, include this term.

Equations \eqref{nonlin-param}--\eqref{nonlin-param-red}
can be rewritten in the more condensed form
\begin{equation}
\label{condensed}
\begin{aligned}
\LL_{\gamma,\omega} \bar{z}&= N(z_\infty\! +\tilde{z}+\bar{z})-N(z_\infty),
\quad &&\bar{z}(0)=\phi(\tilde z_0)-\phi(\tilde{z}_\infty), \\
\dot{\tilde{z}}+L_\gamma\tilde{z}
&=\omega \bar{z},
\quad &&\tilde{z}(0)=\tilde{z}_0-\tilde{z}_\infty,
\end{aligned}
\end{equation}
where we use the abbreviation $\mathbb{L}_{\gamma,\omega}$ to denote the linear operator on
the left hand side of \eqref{nonlin-param}, and $N$ to denote the nonlinearities
on the right hand side of \eqref{nonlin-param}, respectively.
We are now ready to formulate the main theorem of this section.
\goodbreak
\begin{theorem}
\label{stability}
Suppose $\sigma>0$, $\gamma_*=\gamma(u_*)\geq0$ and $l_*=l(u_*)\neq0$ in case $\gamma_*=0$.
Then in the topology of the state manifold $\cSM_\gamma$ we have:
\begin{itemize}
\item[(a)]
 $(u_*,\Gamma_*)\in\cE$ is stable if $\Gamma_*$ is connected and $\zeta_\ast<1$.\\
 Any solution starting in a neighborhood of such a stable equilibrium exists globally and converges to another stable equilibrium
exponentially fast.
\vspace{1mm}
\item[(b)] $(u_*,\Gamma_*)\in\cE$
is unstable if  $\Gamma_*$ is disconnected or if $\zeta_\ast>1$.\\
Any solution starting and staying in a neighborhood of such an
unstable equilibrium converges to another unstable equilibrium
exponentially fast.

\end{itemize}
\end{theorem}
\begin{proof}
{\bf (a)} \, We begin with the case that $(u_*,\Gamma_*)$ is linearly stable.
Then according to Theorem \ref{lin-stability}
we have $X_\gamma=N(L_\gamma)\oplus R(L_\gamma)$. Let $P^c$ denote the projection onto $X^c_\gamma:=N(L_\gamma)$ along $X^s_\gamma:=R(L_\gamma)$ and $P^s=I-P^c$ the complementary projection onto $R(L_\gamma)$.
We parametrize the set of equilibria $\cE$ near $0$ over $N(L_\gamma)$ via the $C^1$-map
$[\cxx\mapsto \cxx +\psi(\cxx)+\phi(\cxx+\psi(\cxx))]$ such that $\psi(0)=\psi^\prime(0)=0$ and $\phi(0)=\phi^\prime(0)=0$.
It follows from the equilibrium equation associated to
\eqref{nonlineq0} (recall that $F_\ast (z_e)$ vanishes for any equilibrium $z_e$),
and  from the definition of $\phi$
that the mapping $\psi$ is determined by the equation
\begin{equation}
\label{definition-psi}
L_\gamma^s \psi({\cxx})= P^s\omega \phi(\cxx+\psi(\cxx)), \quad {\cxx}\in B_{X^c_\gamma}(r).
\end{equation}
Since $L_\gamma^s$ is invertible on $X^s_\gamma$, $\psi\in C^1(B_{X^c_\gamma}(r),D(L^s_\gamma))$
is well-defined by the implicit function theorem and $\psi(0)=\psi^\prime(0)=0$.
\medskip\\
For ${\cxx}_\infty\in X_\gamma^c$ sufficiently small we set
$z_\infty:={\cxx}_\infty +\psi({\cxx}_\infty)+\phi({\cxx}_\infty + \psi({\cxx}_\infty))$.
Then $z_\infty$ is an equilibrium for
\eqref{nonlineq0} and we will now consider the decomposition
$z=z_\infty +\tilde{z}+\bar{z}$ introduced in \eqref{nonlin-param}--\eqref{nonlin-param-red},
or \eqref{condensed}, respectively.
With the ansatz
\begin{equation}
\label{ansatz}
\tilde{z}={\cxx}+\psi({\cxx}_\infty +{\cxx})-\psi({\cxx}_\infty)+{\cy},
\quad ({\cxx},{\cy})\in X^c_\gamma\times X^s_\gamma,
\end{equation}
for ${\cxx}, {\cxx}_\infty\in X^c_\gamma$ small enough, the second line in \eqref{condensed} becomes
\begin{equation}
\label{normalform}
\left\{
\begin{aligned}
\dot{\cxx}&= P^c \omega\bar z,\quad &&\cxx(0)=\cxx_0-{\cxx}_\infty,\\
\dot{{\cy}} + L_\gamma^s{\cy} &= S({\cxx}_\infty,\cxx,\bar z),\quad &&{\cy}(0)={\cy}_0,\\
 \end{aligned}
 \right.
\end{equation}
where
\begin{equation*}
\begin{aligned}
S({\cxx}_\infty,\cxx,\bar z)&= P^s\omega \bar z-\psi^\prime({\cxx}_\infty +{\cxx})P^c\omega\bar z
   -L^s_\gamma[\psi({\cxx}_\infty +{\cxx})-\psi({\cxx}_\infty)],\\
\end{aligned}
\end{equation*}
and
\begin{equation}
\label{tilde-z0}
\tilde{z}_0={\cxx}_0+\psi({\cxx}_0)+{\cy}_0,
\quad ({\cxx}_0,{\cy}_0)\in X^c_\gamma\times (X^s_\gamma\cap \tilde Z_\gamma).
\end{equation}
Next we show that the system of equations \eqref{normalform} admits
a unique global solution
$({\cxx}_\infty,{\cxx},{\cy})$, where the functions
$({\cxx}, {\cy})$ are exponentially decaying, provided
$\bar z$ is exponentially decaying and
$({\cxx}_0,{\cy}_0)$ is sufficiently small.
For this let us first introduce some more notation.
For $\delta\ge 0$ we set
\begin{equation*}
\begin{split}
&\EE_i(\R_+,\delta):=\{v: e^{\delta t}v\in\EE_i(\R_+)\},\quad i=1,2, \\
&\FF_j(\R_+,\delta):=\{v: e^{\delta t}v\in\FF_j(\R_+)\},\quad j=1,2,3,
\end{split}
\end{equation*}
endowed with the norms
\begin{equation*}
\Ver v\Ver _{\EE_i(\R_+,\delta)}=\Ver e^{\delta t}v\Ver_{\EE_i(\R_+)}, \\
\quad
\Ver v\Ver_{\FF_j(\R_+,\delta)}=\Ver e^{\delta t}v\Ver_{\FF_j(\R_+)}.
\end{equation*}
The spaces $\EE(\R_+,\delta)$ and $\FF(\R_+,\delta)$ are then defined
analogously as in Section 3.
We also need the space
\begin{equation}
\label{WX}
\XX(\R_+,\delta):=H^1_p(\R_+,\delta;X_\gamma)\cap L_p(\R_+,\delta;D(L_\gamma)),
\end{equation}
where $H^k_p(\R_+,\delta;E)$ denotes all functions $v:\R_+\to E$ such that
$e^{\delta t}v\in H^k_p(\R_+;E)$, with $E$ a given Banach space.
Finally, let
\begin{equation*}
\begin{split}
\BB_1(r,\delta):&=\{({\cxx}_0,{\cy}_0,\bar z)\in X^c_\gamma\times (X^s_\gamma\cap \tilde Z_\gamma)
\times \EE(\R_+,\delta):
|({\cxx}_0,{\cy}_0)|_{\tilde Z_\gamma}<r\}. \\
\end{split}
\end{equation*}
\noindent
For given $({\cxx}_0,{\cy}_0,\bar z)\in \BB_1(r_0,\delta)$,
with $r_0$ sufficiently small, we set
\begin{equation}
\label{A}
\begin{split}
{\cxx}_\infty:= &\ {\cxx_0}+\int_0^\infty P^c\omega \bar z(\tau)\,d\tau,\\
{\cxx}:= &  -\int_t^\infty P^c\omega\bar z(\tau)\,d\tau, \\
{\cy}:= & \left(\frac{d}{dt}+L^s_\gamma, \text{tr}\right)^{\!\!-1}
\big (S({\cxx}_\infty,{\cxx},\bar z),{\cy}_0\big).\\
\end{split}
\end{equation}
Here we used the notation ${\rm tr\, }w:=w(0)$.
It should be observed that the functions $({\cxx}_\infty,{\cxx})$ occurring in
the third line of \eqref{A} are defined through the first two lines in \eqref{A}.
We now set
\begin{equation}
\mathfrak{S}({\cxx}_0,{\cy}_0,\bar z):=({\cxx}_\infty,{\cxx},{\cy}),
\quad ({\cxx}_0,{\cy}_0,\bar z)\in \BB_1(r_0,\delta),
\end{equation}
where $r_0$ is chosen sufficiently small.

Next we will show that there exists a number $\delta_0>0$ such that
for any $\delta\in [0,\delta_0]$ the mapping
$\mathfrak{S}$ has the following properties:
\begin{equation}
\label{properties-Gamma}
\mathfrak{S}\in C\big(\BB_1(r_0,\delta),
X^c_\gamma\times \XX^c(\R_+,\delta)\times \XX^s(\R_+,\delta)\big),
\quad \mathfrak{S}(0)=0,
\end{equation}
where
\begin{equation*}
\begin{split}
&\XX^c(\R_+,\delta):=H^1_p(\R_+,\delta;X_\gamma^c), \quad
\XX^s(\R_+,\delta):=H^1_p(\R_+,\delta; X^s_\gamma)\cap L_p(\R_+,\delta;D(L^s_\gamma)).
\end{split}
\end{equation*}
Writing $\mathfrak{S}=(\mathfrak{S}_1,\mathfrak{S}_2,\mathfrak{S}_3)$
we readily observe that
\begin{equation}
\label{Gamma-1} \mathfrak{S}_1\in
C^\infty(\BB_1(r_0,\delta),X^c_\gamma),\quad \mathfrak{S}_1(0)=0.
\end{equation}
For $g\in L_p(\R_+,\delta; X^c_\gamma)$, let
$ (Kg)(t):=\int_t^\infty g(\tau)\,d\tau$
and note that
\begin{equation*}
e^{\delta t}(Kg)(t)=\int_t^\infty e^{\delta(t-\tau)}e^{\delta\tau}g(\tau)\,d\tau.
\end{equation*}
Young's inequality for convolution integrals readily yields
\begin{equation*}
K\in \cB\big(L_p(\R_+,\delta;X^c_\gamma),H^{1}_p(\R_+,\delta;X^c_\gamma)\big),
\end{equation*}
and this shows that $\mathfrak{S}_2\in\XX^c(\R_+,\delta)$.
Hence we have
\begin{equation}
\label{Gamma-2} \mathfrak{S}_2\in C^\infty\big(\BB_1(r_0,\delta),
\XX^c(\R_+,\delta)\big),\quad \mathfrak{S}_2(0)=0.
\end{equation}
Concerning the function $\mathfrak{S}_3$,
we know from Theorem~\ref{lin-stability}(v) that  $s(-L^s_\gamma)$, the spectral bound
of $(-L^s_\gamma)$, is negative. Fixing $\delta_0>0$ with $s(-L^s_\gamma)<-\delta_0$
it follows from semigroup theory and
the $L_p$-maximal regularity results stated in
Theorem~\ref{MR-0} and Theorem~\ref{MR-gamma} that
\begin{equation}
\begin{split}
(\frac{d}{dt}+L^s_\gamma, \text{tr})^{-1}
\in \cB\big(L_p(\R_+,\delta;X^s_\gamma)\times \tilde{Z}^s_\gamma,\, {\mathbb X^s}(\R_+,\delta)\big),
\quad \delta\in [0,\delta_0], \\
\end{split}
\end{equation}
where $\tilde{Z}^s_\gamma=X^s_\gamma\cap \tilde Z_\gamma$.
This in conjunction with \eqref{Gamma-1}--\eqref{Gamma-2} and the definition of $S$ implies
\begin{equation}
\label{Gamma-3} \mathfrak{S}_3\in C\big(\BB_1(r_0,\delta),
\XX^s(\R_+,\delta)\big),\quad \mathfrak{S}_3(0)=0.
\end{equation}
Combining \eqref{Gamma-1}--\eqref{Gamma-3} then yields
\eqref{properties-Gamma}.
\medskip\\
For  given $({\cxx}_0,{\cy}_0,\bar z)\in\BB_1(r_0,\delta)$
let  $({\cxx}_\infty,{\cxx}, {\cy})=\mathfrak{S}({\cxx}_0,{\cy}_0,\bar z)$.
Then we have
\begin{equation*}
\begin{split}
{\cxx}(t)&=-\int_t^\infty P^c\omega \bar z(\tau)\,d\tau
=-\int_0^\infty P^c\omega \bar z(\tau)\,d\tau
+ \int_0^t P^c\omega \bar z(\tau)\,d\tau \\
&={\cxx}_0-{\cxx}_\infty + \int_0^t P^c\omega \bar z(\tau)\,d\tau,
\end{split}
\end{equation*}
thus showing that
${\cxx}$ solves the first equation in \eqref{normalform}.
In summary, we have shown that $({\cxx}_\infty,{\cxx}, {\cy})=\mathfrak{S}({\cxx}_0,{\cy}_0,\bar z)$
is for every $({\cxx}_0,{\cy}_0,\bar z)\in \BB_1(r_0,\delta)$
the unique solution of \eqref{normalform} in $X^c_\gamma\times \XX^c(\R_+,\delta)\times\XX^s(\R_+,\delta)$,
where $\delta\in [0,\delta_0]$.
\\
Setting
\begin{equation}
\label{definition-tilde-Z}
\begin{split}
\tilde z &=\tilde{\mathfrak{Z}}({\cxx}_0,{\cy}_0,\bar z):=
{\cxx}+\psi({\cxx}_\infty +{\cxx})-\psi({\cxx}_\infty)+{\cy},\\
z_\infty &=\mathfrak{Z}_\infty({\cxx}_0,{\cy}_0,\bar z):={\cxx}_\infty+\psi({\cxx}_\infty)+\phi({\cxx}_\infty+\psi({\cxx}_\infty))
\end{split}
\end{equation}
for $({\cxx}_\infty,{\cxx}, {\cy})=\mathfrak{S}({\cxx}_0,{\cy}_0,\bar z)$, we see that
\begin{equation*}
\tilde{\mathfrak{Z}}\in C(\BB_1(r_0,\delta), \XX(\R_+,\delta)),
\quad \tilde{\mathfrak{Z}}(0)=0,
\end{equation*}
and
\begin{equation}
\label{regularity-Z-infty}
\mathfrak{Z}_\infty\in C(\BB_1(r_0,\delta), Z_\infty),
\quad \mathfrak{Z}_\infty(0)=0,
\end{equation}
where $Z_\infty=[W^2_p(\Omega\setminus\Gamma_*)\cap C(\bar\Omega)]\times W^{4-1/p}_p(\Gamma_*)$.
It then follows from the derivation of
\eqref{ansatz}--\eqref{normalform} that
\begin{equation*}
(z_\infty,\tilde z)=({\mathfrak Z}({\cxx}_0,{\cy}_0,\bar z),\tilde{\mathfrak Z}({\cxx}_0,{\cy}_0,\bar z))
\end{equation*}
is for every given
$({\cxx}_0,{\cy}_0,\bar z)\in\BB_1(r_0,\delta)$ the unique (global) solution of \eqref{nonlin-param-red}
with $\tilde z$ in the regularity class $\XX(\R_+,\delta)$.
In a next step we shall show that
$\tilde z$ in fact has better regularity properties, namely
\begin{equation}
\label{regularity-tilde-Z}
\tilde{\mathfrak{Z}}\in C(\BB_1(r_0,\delta), \EE(\R_+,\delta)),
\quad \tilde{\mathfrak{Z}}(0)=0.
\end{equation}
In order to see this, let us first consider the case $\gamma\equiv 0$
(which implies $\gamma_\ast=0$).
From the fourth line of \eqref{nonlin-param-red}, the fact that
$\tilde z\in \XX(\R_+,\delta)$, and
\begin{equation*}
[v\mapsto v|_{\Gamma_\ast}]\in
\cB(\XX(\R_+,\delta), W^{1-1/2p}_p(\R_+,\delta;L_p(\Gamma_*)))
\end{equation*}
follows
\begin{equation*}
\tilde\rho=(\mu-\sigma A_*)^{-1}((l_*/u_*)\tilde
v+\mu\tilde\rho)\in W^{1-1/2p}_p(\R_+,\delta;H^2_p(\Gamma_*)),
\end{equation*}
where $\mu$ is in the resolvent set of $\sigma A_*$.
From the fifth line of \eqref{nonlin-param-red},
the fact that $(\tilde z,\bar z)\in \XX(\R_+,\delta)\times\EE(\R_+,\delta)$,
and trace theory for $\tilde v$
follows
\begin{equation*}
l_*\partial_t \tilde{\rho} =[\![d_*\partial_\nu \tilde{v}]\!]+l_*\omega\bar{\rho}
\in W^{1/2-1/2p}_p(\R_+,\delta;L_p(\Gamma_*)),
\end{equation*}
implying that $\tilde\rho\in W^{3/2-1/2p}_p(\R_+,\delta;L_p(\Gamma_*))$.
Hence \eqref{regularity-tilde-Z} holds for $\gamma=0$.
\medskip\\
If $\gamma>0$ (and thus $\gamma_*>0$) we use the embedding
\begin{equation*}
H^1_p(\R_+,\delta;W^{2-1/p}_p(\Gamma_*))\cap
L_p(\R_+,\delta;W^{4-1/p}_p(\Gamma_\ast)) \ein
W^{1-1/2p}_p(\R_+,\delta;H^2_p(\Gamma_*))
\end{equation*}
and the fourth equation in \eqref{nonlin-param-red} to conclude that
$\tilde\rho\in W^{2-1/2p}_p(\R_+,\delta;L_p(\Gamma_*))$.
Hence \eqref{regularity-tilde-Z} holds in this case as well.
\medskip\\
Let us now turn our attention to equation \eqref{nonlin-param}, or equivalently, the
first line of \eqref{condensed}.
In a similar way as in the proof of \cite[Proposition~10]{LPS06}
(extra consideration is needed in order to deal with the additional terms involving $z_\infty$)
one verifies that the  mapping
\begin{equation*}
[(z_\infty,z)\mapsto N(z_\infty+z)-N(z_\infty)]:
\mathbb U(\delta) \to \FF(\R_+,\delta)
\end{equation*}
is $C^1$ and vanishes together with its Fr\'echet derivative at $(0,0)$.
Here $\mathbb U(\delta)$ denotes an open neighborhood of $(0,0)$ in
$Z_\infty\times\EE(\R_+,\delta)$.
Let
\begin{equation*}
\BB(r,\delta)=\!\{({\cxx}_0,{\cy}_0,\bar z)\in X^c_\gamma\times (X^s_\gamma\cap \tilde Z_\gamma)
\times \EE(\R_+,\delta):
|({\cxx}_0,{\cy}_0,\bar z)|_{[\tilde Z_\gamma]^2\times \EE(\R_+,\delta)}<r_0\}, \\
\end{equation*}
and let
$
{\rm ext}_\delta\in
\cB\big(W^{2-2/p}_p(\Omega\setminus\Gamma_*)\cap C(\overline\Omega))\times W^{4-3/p}_p(\Gamma_*),
\EE(\R_+,\delta)\big)
$
be an appropriate extension operator with $({\rm ext}_\delta w_0)(0)=w_0$.
\smallskip\\
For $({\cxx}_0,{\cy}_0,\bar z)\in \BB(r_0,\delta)$, with $r_0$ sufficiently small,
we define
\begin{equation*}
M({\cxx}_0,{\cy}_0,\bar z):=N(z_\infty+\tilde z
+{\rm ext}_\delta (\phi(\tilde z_0)-\phi(\tilde z_\infty)-\bar z(0))+\bar z)-N(z_\infty).
\end{equation*}
It follows from
\eqref{definition-tilde-Z}-\eqref{regularity-tilde-Z} that
$M\in C(\BB(r_0,\delta),\FF(\R_+,\delta))$, $M(0,0,0)=0$,
and $D_3M(0,0,0)=0$.
Moreover,
\begin{equation}
\label{properties-M-2}
M({\cxx}_0,{\cy}_0,\bar z)(0)=N(z_0)-N(z_\infty),
\quad ({\cxx}_0,{\cy}_0,\bar z)\in \BB(r_0,\delta),
\end{equation}
where we recall that $\tilde z(0)=\tilde z_0-\tilde z_\infty$,
$z_0=\tilde{z}_0+\phi(\tilde{z}_0)$,
and $z_\infty=\tilde{z}_\infty+\phi(\tilde{z}_\infty)$.

Finally, for $({\cxx}_0,{\cy}_0,\bar z)\in \BB(r_0,\delta)$ let
\begin{equation}
\label{definition-K}
K({\cxx}_0,{\cy}_0,\bar z):=(\LL_{\gamma,\omega},{\rm tr})^{-1}
(M({\cxx}_0,{\cy}_0,\bar z),\phi(\tilde z_0)-\phi(\tilde z_\infty)).
\end{equation}
It follows from \eqref{properties-M-2} and the definition of $\phi$ that the functions
$$(M({\cxx}_0,{\cy}_0,\bar z),\phi(\tilde z_0)-\phi(\tilde z_\infty))$$
satisfy the necessary compatibility conditions, whenever
$({\cxx}_0,{\cy}_0,\bar z)\in \BB(r_0,\delta)$. Slight
modifications of the results in \cite{DPZ08} then imply that
$K:\BB(r_0,\delta)\to \EE(\R_+,\delta)$ is well-defined, provided
$\omega$ is large enough (and $\delta$ is in $[0,\delta_0]$ with
$\delta_0$ as above). From the properties of the mappings $N$,
$\psi$ and $\phi$, the definition of $\tilde z_0$ and $\tilde
z_\infty$ (recall that $\tilde
z_0={\cxx}_0+\psi({\cxx}_0)+{\cy}_0$, $\tilde
z_\infty={\cxx}_\infty+\psi({\cxx}_\infty)$), and the contraction
mapping theorem follows that $K$, defined in
\eqref{definition-K}, has for each $({\cxx}_0,{\cy}_0)$
sufficiently small a unique fixed point
\begin{equation*}
\bar z=\bar z({\cxx}_0,{\cy}_0)\in \EE(\R_+,\delta),
\end{equation*}
and that the mapping $[({\cxx}_0,{\cy}_0)\mapsto \bar z({\cxx}_0,{\cy}_0)]$
is continuous and vanishes at $(0,0)$.
By construction it follows that $\bar z=\bar z({\cxx}_0,{\cy}_0)$
solves
\begin{equation*}
\LL_{\gamma,\omega}\bar z=N(z_\infty+\tilde z+\bar z)-N(z_\infty),
\quad
\bar z(0)=\phi(\tilde z_0)-\phi(\tilde z_\infty).
\end{equation*}
In summary, we have shown that for each $z_0\in \cSM_\gamma$ small enough,
there exists
\begin{equation*}
(z_\infty,\tilde z,\bar z)\in Z_\infty\times \EE(\R_+,\delta)\times\EE(\R_+,\delta)
\end{equation*}
such that
$z=z_\infty+\tilde z+\bar z$ is the unique global solution of \eqref{nonlineq0}.
In particular we have shown that for every $z_0\in \cSM_\gamma$ small enough
there exists a unique equilibrium $z_\infty=z_\infty(z_0)$ such that
the solution of \eqref{nonlineq0} exists for all $t\ge 0$ and converges to
$z_\infty$ in $\cSM_\gamma$ at an exponential rate.
\noindent
\bigskip\\
{\bf (b)} Now we consider the linearly unstable case
and we first show that the equilibrium~$0$ is unstable for the nonlinear equation \eqref{nonlineq0}.
Using the same notation as in part (a) we consider the
system of equations
\begin{equation}
\label{condensed-noequilibrium}
\begin{aligned}
\LL_{\gamma,\omega}\bar{z}&= N(\tilde z+\bar z),\quad && \bar{z}(0) =\phi(\tilde z_0),\\
\dot{\tilde z} + L_\gamma\tilde z &= \omega \bar z,\quad &&\tilde z(0)  =\tilde z_0.\\
  \end{aligned}
\end{equation}
Given $\alpha\in\R$ one verifies (by similar considerations as in \cite[Proposition 10]{LPS06}) that
there is a nondecreasing function
 $\eta: \R_+\to \R_+$ such that $\eta(r)\to 0$ as $r\to 0$ and
\begin{equation}
\label{NR-estimates}
\Ver e^{\alpha t}N(z)\Ver_{\FF(a)}\le \eta(r) \Ver e^{\alpha t}z\Ver_{\EE(a)},
\quad e^{\alpha t}z\in\EE(a),
\end{equation}
whenever $|z(t)|_{Z_\gamma}\le r$ for $0\le t\le a$.
Here $a>0$ is an arbitrary fixed number,
$\EE(a):=\EE([0,a])$ and $\FF(a):=\FF([0,a])$.
For later use we note that
\begin{equation}
\label{einbettung}
 \EE(a)\hookrightarrow L_p([0,a];X_\gamma),
\end{equation}
where the embedding constant is independent of $a$.
\medskip\\
Let $\sigma^{+}$ be the collection of all positive eigenvalues of $(-L_\gamma)$.
and let $P^{+}$ be the spectral projection related to the spectral set $\sigma^{+}$.
Additionally, let $P^{-}:=I-P^{+}$ and $X^{\pm}_\gamma:=P^{\pm}(X_\gamma)$.
Then $X^+_\gamma$ is finite dimensional and we obtain the decomposition
\begin{equation*}
X=X^{+}_\gamma\oplus X^{-}_\gamma,\quad L_\gamma=L^{+}_\gamma\oplus L^{-}_\gamma.
\end{equation*}
We note that $\sigma(-L^{+}_\gamma)=\sigma^{+}$ and $\sigma(-L^{-}_\gamma)\subset [\text{Re}\,z\le 0]$,
where $\sigma(-L^{\pm}_\gamma)$ denotes the spectrum of $(-L^{\pm}_\gamma)$, respectively.
Let $\lambda_\ast$ be the smallest positive eigenvalue of $(-L^{+}_\gamma)$
and choose positive numbers
$\kappa,\mu$ such that
$[\kappa-\mu,\kappa+\mu]\subset (0,\lambda_\ast)$.
We remind that the spectrum of $(-L_\gamma)$ consists of real eigenvalues,
so that the strip $[\kappa-\mu\le\text{Re}\,z\le \kappa+\mu]$
does not contain any spectral values of $(-L_\gamma)$.
Therefore, t here exists a constant $M\ge 1$ such that
\begin{equation}
\label{estimate-SG}
|e^{-L^-_\gamma t}|\le Me^{(\kappa-\mu)t},\quad
|e^{L^+_\gamma t}|\le Me^{-(\kappa+\mu)t},\quad t\ge 0.
\end{equation}
\\
Suppose now, by contradiction, that the equilibrium $0$ is stable for \eqref{nonlineq0}.
Then for every $r>0$ there is a number $\delta>0$ such that
\eqref{nonlineq0} admits a global solution $z\in\EE(\R_+)$ with $|z(t)|\le r$ for all $t\ge 0$
whenever $z_0\in \bar B_\delta(0)$.
\\
In the following we will use the decomposition $z=\tilde z+\bar z$,
where $(\tilde z,\bar z)$ is the solution of the linear system
\eqref{condensed-noequilibrium}. (The function $z=\tilde z+\bar z$ is known,
so that the first equation has a unique solution $\bar z$. With $\bar z$ determined,
$\tilde z=z-\bar z$ is the unique solution of the second equation.)
The functions $P^\pm\tilde z$ satisfy
\begin{equation}
\label{pm-equations}
\frac{d}{dt}P^{\pm}\tilde z+L^\pm_\gamma P^\pm\tilde z=P^\pm \omega \bar z,
\quad P^\pm \tilde z(0)=P^\pm \tilde z_0.
\end{equation}
Next we shall show that $P^+\tilde z$ is given by the formula
\begin{equation}
\label{P-plus-formula}
P^+\tilde z(t)=-\int_t^\infty e^{-L^+_\gamma(t-\tau)}P^+\omega \bar z\,d\tau,
\quad t\ge 0.
\end{equation}
Given any $a>0$ it follows from $|P^+\tilde z(t)|_{X^+_\gamma}\le r$ that
\begin{equation}
\label{r-P-plus}
\Ver e^{-\kappa t}P^+\tilde z\Ver_{L_p([0,a];X^+_\gamma)}
\le r\Big(\int_0^a e^{-\kappa pt}\,dt\Big)^{1/p}\le  C(\kappa,p) r.
\end{equation}
From the relation
\begin{equation}
\label{d-dt-plus}
\frac{d}{dt}e^{-\kappa t}P^+\tilde z=
 (-\kappa -L^+_\gamma)e^{-\kappa t}P^+\tilde z
+ e^{-\kappa t}P^+ \omega \bar z,
\end{equation}
\eqref{r-P-plus}-\eqref{d-dt-plus} and \eqref{einbettung} follows
\begin{equation}
\label{P-plus-estimate}
\Ver e^{-\kappa t}P^+\tilde z\Ver_{\XX(a)}
\le C_1\big(r + \Ver e^{-\kappa t}\bar z\Ver_{\EE(a)}\big),
\end{equation}
with a universal constant $C_1$.
Here $\XX(a)$ is defined as in \eqref{WX}, with the difference that
$\R_+$ is replaced by the interval $[0,a]$ and $\delta=0$.
We also recall that $X^+_\gamma$ is finite dimensional,
so that the spaces $X^+_\gamma$ and $D(L^+_\gamma)$ coincide
(and therefore carry equivalent norms).
From semigroup theory, maximal regularity, \eqref{estimate-SG}-\eqref{pm-equations}
and \eqref{einbettung} follows
\begin{equation}
\label{P-minus-estimate}
\begin{split}
\Ver e^{-\kappa t}P^-\tilde z\Ver_{\XX(a)}
  &\le M\big(|P^-\tilde z_0|+\Ver e^{-\kappa t}P^-\omega \bar z\Ver_{L_p([0,a];X_\gamma)}\big) \\
& \le M \big(|P^-\tilde z_0|+ C_2 \Ver e^{-\kappa t}\bar z\Ver_{\EE(a)}\big).
\end{split}
\end{equation}
Combining \eqref{P-plus-estimate}-\eqref{P-minus-estimate} results in
\begin{equation}
\label{z-estimate}
\Ver e^{-\kappa t}\tilde z\Ver_{\XX(a)}
\le C_3\big(r +|P^-\tilde z_0|+\Ver e^{-\kappa t}\bar z\Ver_{\EE(a)}\big),
\end{equation}
where $C_3$ is a universal constant.
Similarly as in part (a) we  can infer from the equation for $\tilde z$ that
\begin{equation}
\label{z-tilde-inEE}
\Ver e^{-\kappa t}\tilde z\Ver_{\EE(a)}
\le c(\Ver e^{-\kappa t}\tilde z\Ver_{\XX(a)}+\Ver e^{-\kappa t}\bar z\Ver_{\EE(a)}),
\end{equation}
and this implies
\begin{equation}
\label{tilde-z-estimate}
\Ver e^{-\kappa t}\tilde z\Ver_{\EE(a)}
\le  C_4\big(r+ |P^-\tilde z_0| + \Ver e^{-\kappa t}\bar z\Ver_{\EE(a)}\big)
\end{equation}
with $C_4=c(1+C_3)$.
On the other hand we obtain from the equation for $\bar z$ and \eqref{NR-estimates}
\begin{equation*}
\begin{split}
\Ver e^{-\kappa t}\bar z\Ver_{\EE(a)}
&\le \bar C\big(|\phi(\tilde z_0)|+\Ver e^{-\kappa t}N(\tilde z+\bar z)\Ver_{\EE(a)}\big) \\
&\le \bar C\big( |\phi(\tilde z_0)|
+\eta(r)(\Ver e^{-\kappa t}\tilde z\Ver_{\EE(a)}+\Ver e^{-\kappa t}\bar z\Ver_{\EE(a)})\big).
\end{split}
\end{equation*}
If $r$ is chosen small enough such that $\bar C\eta(r)\le 1/2$ then
\begin{equation}
\label{bar-z}
\Ver e^{-\kappa t}\bar z\Ver_{\EE(a)}
\le 2\bar C\big(|\phi(\tilde z_0)|+\eta(r)\Ver e^{-\kappa t}\tilde z\Ver_{\EE(a)}\big).
\end{equation}
We can, at last, combine \eqref{tilde-z-estimate}--\eqref{bar-z}
to the result
\begin{equation}
\label{combined-estimate} \Ver e^{-\kappa t}\tilde
z\Ver_{\EE(a)}+\Ver e^{-\kappa t}\bar z\Ver_{\EE(a)} \le C_5
\big(r+|P^-\tilde z_0|+ |\phi(\tilde z_0)|\big),
\end{equation}
provided $r$ is chosen small enough so that $2(1+C_4)\bar C\eta(r)\le 1/2$.
Since all estimates are independent of $a$ we conclude that
$e^{-\kappa t}z\in\EE(\R_+)$.
From  \eqref{combined-estimate}
and H\"older's inequality follows
\begin{equation*}
\begin{split}
 &e^{-\kappa t}\int_t^\infty |e^{-L^+_\gamma(t-\tau)}P^+\omega \bar z(\tau)|_{X^+_\gamma}\,d\tau \\
&\le M\Big(\int_t^\infty e^{\mu p^\prime\, (t-\tau)}\,d\tau\Big)^{1/p'}
 \Ver e^{-\kappa \tau}\omega\bar z\Ver_{L_p(\R_+:X_\gamma)}
\le C\Ver e^{-\kappa t}\bar z\Ver_{\EE(\R_+)} <\infty,
\end{split}
\end{equation*}
thus showing that the integral $\int_t^\infty e^{-L^+_\gamma(t-\tau)}P^+\omega \bar z\,d\tau$
exists in $X^+_\gamma$ for every $t\ge 0$. Moreover, its norm in
$X^+_\gamma$ grows no faster than the exponential function
$e^{\kappa t }$.
It follows from the variation of parameters formula that
\begin{equation*}
\label{plus-unstable}
\begin{aligned}
e^{L^+_\gamma t}
\big(P^+\tilde z(t)
+ \int_t^\infty e^{-L^+_\gamma(t-\tau)} P^+\omega \bar z(\tau)\,d\tau\big)
=P^+\tilde z_0
+\int_0^\infty e^{L^+_\gamma \tau}P^+\omega \bar z(\tau)\,d\tau,
\end{aligned}
\end{equation*}
and the estimate
\begin{equation*}
\big|e^{L^+_\gamma t}
(P^+\tilde z(t)
+ \int_t^\infty\!\! e^{-L^+_\gamma(t-\tau)} P^+\omega \bar z\,d\tau)\big|_{X^+_\gamma}
\!\le Me^{-({\kappa+\mu})t}(r+Ce^{\kappa t}),\quad t\ge 0,
\end{equation*}
then shows that
$P^+\tilde z_0 +\int_0^\infty e^{L^+_\gamma \tau}P^+ \omega \bar z\,d\tau=0$.
Thus the representation \eqref{P-plus-formula} holds as claimed.
With this established, we obtain from Young's inequality for convolution integrals
\begin{equation*}
\begin{split}
\Ver e^{-\kappa t}P^+\tilde z(t)\Ver_{L_p(\R_+,X^+_\gamma)}
\le   {M}{\mu^{-1}}\Ver e^{-\kappa t}P^+\omega \bar z\Ver_{L_p(\R_+;X^+_\gamma)}.
\end{split}
\end{equation*}
It then follows from \eqref{einbettung} and \eqref{d-dt-plus}  that
\begin{equation}
\begin{split}
\Ver e^{-\kappa t}P^+\tilde z\Ver_{\XX(\R_+)}
\le C\Ver e^{-\kappa t}\bar z\Ver_{\EE(\R_+)}.
\end{split}
\end{equation}
We can now imitate the estimates in \eqref{P-minus-estimate}-\eqref{bar-z},
with the interval $[0,a]$ replaced by $\R_+$, to conclude that
\begin{equation}
\label{combined-estimate-II}
\Ver e^{-\kappa t}\tilde z\Ver_{\EE(\R_+)}+\Ver e^{-\kappa t}\bar z\Ver_{\EE(\R_+)}
\le C_6 \big(|P^-\tilde z_0|+ |\phi(\tilde z_0)|\big).
\end{equation}
This, in combination with
\eqref {estimate-SG}, \eqref{P-plus-formula}, and
H\"older's inequality, yields the estimate
\begin{equation*}
\begin{split}
|P^+\tilde z_0|_{X^+_\gamma}
&\le M \int_0^\infty e^{-\mu\tau}|e^{-\kappa \tau}P^+\omega\bar z|_{X^+_\gamma}\,d\tau \\
&\le C
\Ver e^{-\kappa t}P^+\omega\bar z\Ver_{L_p(\R_+;X^+_\gamma)}
\le C\big(|P^-\tilde z_0|+|\phi(\tilde z_0)|\big).
\end{split}
\end{equation*}
By decreasing $\delta$ if necessary, we can assume that
$C|\phi(\tilde z_0)|\le 1/2(|P^+\tilde z_0|+|P^-\tilde z_0|)$
for all $\tilde z_0\in B_\delta(0)$.
(Recall that $\phi(0)=\phi^\prime (0)=0$.)
Hence
\begin{equation}
\label{instability}
|P^+\tilde z_0|_{\tilde Z_\gamma}\le C_7|P^-\tilde z_0|_{\tilde Z_\gamma},
\quad \tilde z_0\in B_\delta(0),
\end{equation}
with a uniform constant $C_7$, and this shows that $0$ cannot be stable for \eqref{nonlineq0}.
\medskip\\
It remains to show the last assertion of Theorem~\ref{stability}(b).
For this we consider the projection $P^u=I-P^c-P^s$ which projects onto
$X^u_\gamma$, the unstable subspace of $X_\gamma$
associated with the (finitely many) unstable eigenvalues.
As in part (a) we will
show that there exists an equilibrium $z_\infty$ such that
any solution that stays in a small neighborhood of $0$ converges to $z_\infty=z_\infty(z_0)$
exponentially fast as $t\to\infty$.
Using the decomposition ${\cy}={\cy}_s+{\cy}_u$, we obtain  as in (a)
the following system of equations:
\begin{equation}
\label{normalform-unstable}
\left\{
\begin{aligned}
\dot{\cxx}&= P^c\omega \bar z, \quad &&\cxx(0)=\cxx_0-{\cxx}_\infty,\\
\dot{{\cy}}_s + L_\gamma^s\,{\cy}_s &= S_s({\cxx}_\infty,\cxx,\bar z),\quad &&{\cy}_s(0)={\cy}_0^s,\\
\dot{{\cy}}_u + L_\gamma^u\,{\cy}_u &= S_u({\cxx}_\infty,\cxx,\bar z),\quad &&{\cy}_u(0)={\cy}_0^u,\\
\end{aligned}
\right.
\end{equation}
with
\begin{equation*}
\begin{aligned}
S_j({\cxx}_\infty,\cxx,\bar z)&= P^j\omega \bar{z}-\psi_j^\prime({\cxx}_\infty\! +{\cxx})P^c\omega\bar z
   -L^j_\gamma[\psi_j({\cxx}_\infty\! +{\cxx})-\psi_j({\cxx}_\infty)],
\end{aligned}
\end{equation*}
where $j\in\{s,u\}$, and where the functions $\psi_j$ are defined
similarly as in \eqref{definition-psi}.
\noindent
Suppose we have a global solution $z\in \EE(\R_+)$ of \eqref{nonlineq0}
with $z(0)=z_0\in\cSM_\gamma$
which satisfies $|z|_{\cX_\gamma}\leq r$, where $r>0$ is sufficiently small.
By similar arguments as above 
(the presence of the function $S_u$ does not cause any principal difficulties) 
we infer that
\begin{equation}
\label{y-unstable}
{\cy}_u(t)= -\int_t^\infty e^{-L^u_\gamma(t-\tau)}
S_u({\cxx}_\infty,\cxx,\bar z)\,d\tau, \quad t\ge 0.
\end{equation}
For $({\cxx}_0,{\cy}^s_0,\bar z)\in \BB_1(r_0,\delta)$, with $r_0$ sufficiently small,
we set
\begin{equation}
\label{B}
\begin{split}
{\cxx}_\infty:= &\ {\cxx_0}+\int_0^\infty P^c\omega \bar z(\tau)\,d\tau,\\
{\cxx}(t):= &  -\int_t^\infty P^c\omega\bar z(\tau)\,d\tau, \\
{\cy}_s:= & \,\left(\frac{d}{dt}+L^s_\gamma, \text{tr}\right)^{\!\!-1}
\big (S_s({\cxx}_\infty,{\cxx},\bar z),{\cy}^s_0\big),\\
{\cy}_u(t):=& -\int_t^\infty e^{-L^u_\gamma(t-\tau)}
S_u({\cxx}_\infty,\cxx,\bar z)\,d\tau.
\end{split}
\end{equation}
As in part (a) we conclude that
\eqref{B} admits for each
$({\cxx}_0,{\cy}^s_0,\bar z)\in \BB_1(r_0,\delta)$, with $r_0$ sufficiently small, a unique solution
\begin{equation*}
({\cxx}_\infty,{\cxx}, {\cy}_s,{\cy}_u)=
\mathfrak{S}({\cxx}_0,{\cy}^s_0,\bar z)
\in  X^c_\gamma\times \XX^c(\R_+,\delta)\times\XX^s(\R_+,\delta)\times\XX^u(\R_+,\delta).
\end{equation*}
Following the arguments of part (a) then renders a solution
\begin{equation*}
\begin{split}
&\mathfrak{Z}({\cxx}_0,{\cy}^s_0)=z_\infty\! +{\cxx}+\psi({\cxx}+{\cxx}_\infty)-\psi({\cxx}_\infty)+{\cy}_s+{\cy}_u+\bar z\\
\end{split}
\end{equation*}
of \eqref{nonlineq0}
with $z_0={\cxx}_0+\psi({\cxx}_0)+{\cy}^u_0+{\cy}^s_0+
\phi({\cxx}_0+\psi({\cxx}_0)+{\cy}^u_0+{\cy}^s_0)$,
where ${\cy}^u_0$ is determined by
\begin{equation}
\label{stable-manifold}
\begin{split}
&{\cy}^u_0=-\int_0^\infty e^{L^u_\gamma\,\tau}
    S_u({\cxx}_\infty,\cxx,\bar z)\,d\tau .\\
\end{split}
\end{equation}
The solution $\mathfrak{Z}({\cxx}_0,{\cy}^s_0)$ converges exponentially
fast toward the equilibrium $z_\infty$.
In addition, we have shown that the initial value $z_0$ necessarily lies
on the stable manifold belonging to $z_\infty$,
determined by the relation \eqref{stable-manifold}.

Due to uniqueness of (local) solutions to \eqref{nonlineq0}, the
solution $\mathfrak{Z}({\cxx}_0,{\cy}^s_0)$ coincides with the given
global solution $z$, and the proof of part (b) is now complete.
\end{proof}
{\bf Global existence and convergence.} There are several
obstructions against global existence for the Stefan problem
\eqref{stefan}:
\begin{itemize}
\item {\em regularity}: the norms of either $u(t)$, $\Gamma(t)$,
and in addition  $[\![d\partial_\nu u(t)]\!]$ in case $\gamma\equiv0$, become unbounded;
\item {\em well-posedness}: in case $\gamma\equiv0$
the well-posedness condition $l(u)\neq0$ may become violated; or
$u$ may become $0$;
\item {\em geometry}: the topology of the interface changes;\\
    or the interface touches the boundary of $\Omega$;\\
    or the interface contracts to a point.
\end{itemize}
Note that  the compatibility conditions  $[\![\psi(u)]\!]+\sigma \cH=0$ in case $\gamma\equiv0$, and
$$(l(u)-[\![\psi(u)]\!]-\sigma \cH)([\![\psi(u)]\!]+\sigma \cH) =\gamma(u)[\![d\partial_\nu u]\!] $$
in case $\gamma>0$ are preserved by the semiflow.
\\
\goodbreak
\noindent
Let $(u,\Gamma)$ be a solution in the state manifold $\cSM_\gamma$.
By a {\em uniform ball condition} we mean the existence of a radius
$r_0>0$ such that for each $t$, at each point $x\in\Gamma(t)$ there
exist centers $x_i\in \Omega_i(t)$ such that $B_{r_0}(x_i)\subset
\Omega_i$ and $\Gamma(t)\cap \bar{B}_{r_0}(x_i)=\{x\}$, $i=1,2$.
Note that this condition bounds the curvature of $\Gamma(t)$,
prevents it from shrinking to a point, from touching the outer
boundary $\partial \Omega$, and from undergoing topological changes.

With this property, combining the semiflow for (\ref{stefan}) with
the Lyapunov functional and compactness we obtain the following
result.
\begin{theorem}
\label{Qual}
Let $p>n+2$, $\sigma>0$, suppose $\psi,\gamma\in C^3(0,\infty)$, $d\in C^2(0,\infty)$ such that
either $\gamma\equiv0$ or $\gamma(u)>0$ on $(0,\infty)$, and assume
$$\kappa(u)=-u\psi^{\prime\prime}(u)>0,\quad d(u)>0,\quad u\in(0,\infty).$$
Suppose that $(u,\Gamma)$ is a solution of
(\ref{stefan}) in the state manifold $\cSM_\gamma$ on its maximal time interval $[0,t_*)$.
Assume the following on $[0,t_*)$:
\begin{itemize}
\item[(i)] $|u(t)|_{W^{2-2/p}_p}+|\Gamma(t)|_{W^{4-3/p}_p}\leq M<\infty$;
\vspace{2mm}
\item[(ii)] $|[\![d(u(t))\partial_\nu u(t)]\!]|_{W^{2-6/p}_p}\leq M<\infty$ in case $\gamma\equiv0$;
\vspace{2mm}
\item[(iii)] $|l(u(t))|\geq 1/M$ in case $\gamma\equiv 0$;
\vspace{2mm}
\item[(iv)] $u(t)\ge 1/M$;
\vspace{2mm}
\item[(v)] $\Gamma(t)$ satisfies a uniform ball condition.
\end{itemize}
Then $t_*=\infty$, i.e.\ the solution exists globally. If its
limit set contains a stable equilibrium
$(u_\infty,\Gamma_\infty)\in\cE$, i.e.\
$\varphi^\prime(u_\infty)<0$, then it converges in $\cSM_\gamma$
to this equilibrium. On the contrary, if $(u(t),\Gamma(t))$ is a
global solution in $\cSM_\gamma$ which converges to an equilibrium
$(u_*,\Gamma_*)$ with $l(u_*)\neq0$ in case $\gamma\equiv0$ in
$\cSM_\gamma$ as $t\to\infty$, then the properties $(i)$-$(v)$ are
valid.
\end{theorem}
\begin{proof}
Assume that assertions (i)--(v) are valid. Then
$\Gamma([0,t_*))\subset W^{4-3/p}_p(\Omega,r)$ is bounded, hence
relatively compact in $W^{4-3/p-\ve}_p(\Omega,r)$. (See
\eqref{definition-W-r} for the definition of $W^{s}_p(\Omega,r)$.)

Thus we may cover this set by finitely many balls with centers $\Sigma_k$ real analytic in such a way that
${\rm dist}_{W^{4-3/p-\ve}_p}(\Gamma(t),\Sigma_j)\leq \delta$ for some $j=j(t)$, $t\in[0,t_*)$. Let $J_k=\{t\in[0,t_*):\, j(t)=k\}$. Using for each $k$ a Hanzawa-transformation $\Xi_k$, we see that the pull backs $\{u(t,\cdot)\circ\Xi_k:\, t\in J_k\}$ are bounded in $W^{2-2/p}_p(\Omega\setminus \Sigma_k)$, hence relatively compact in $W^{2-2/p-\ve}_p(\Omega\setminus\Sigma_k)$. Employing now Corollary \ref{wellposed3}  we obtain solutions
 $(u^1,\Gamma^1)$ with initial configurations $(u(t),\Gamma(t))$ in the state manifold on a common
time interval, say $(0,\tau]$, and by uniqueness we have
$$(u^1(\tau),\Gamma^1(\tau))=(u(t+\tau),\Gamma(t+\tau)).$$
Continuous dependence implies then relative compactness of
$(u(\cdot),\Gamma(\cdot))$ in $\cSM_\gamma$. In particular,
$t_*=\infty$ and the orbit $(u,\Gamma)(\R_+)\subset\cSM_\gamma$ is
relatively compact. The negative total entropy is a strict Lyapunov
functional, hence the limit set $\omega(u,\Gamma)\subset
\cSM_\gamma$ of a solution is contained in the set $\cE$ of
equilibria.  By compactness $\omega(u,\Gamma)\subset \cSM_\gamma$ is
non-empty, hence the solution comes close to $\cE$, and stays there.
Then we may apply the convergence result Theorem \ref{stability}.
The converse is proved by a compactness argument.
\end{proof}
\begin{remark}
We believe that the extra assumption $\varphi^\prime(u_\infty)<0$
in Theorem \ref{Qual} can be replaced by
$\varphi^\prime(u_\infty)\neq0$. However, to prove this requires
more technical efforts, and we refrain from doing this here.
\end{remark}
\bigskip
\noindent
{\bf Acknowledgment:}
J.P.\ and R.Z.\ express their thanks for hospitality
to the Department of Mathematics at Vanderbilt University, where important parts of this work originated.
We thank Mathias Wilke for helpful discussions.
Moreover, we thank one of the anonymous reviewers for constructive and insightful remarks, 
and for pointing out some additional references related to the subject. 


\end{document}